\documentclass[11pt,english]{format}

\usepackage[english]{babel}
\usepackage{formatthm}
\usepackage{amssymb}

\theoremstyle{plain}
\newtheorem{theorem}{Theorem}
\newtheorem{lemma}[theorem]{Lemma}

\newtheorem{corollary}[theorem]{Corollary}
\newtheorem{proposition}[theorem]{Proposition}

\theoremstyle{definition}
\newtheorem{definition}{Definition}

\theoremstyle{remark}
\newtheorem{remark}{Remark}

\def\poi#1#2#3#4#5#6#7{\def\un{#5#6#7}\def\deux{#6#7}
\def\trois{#2#4} \def\cinq{#3#4#5}
\ifx\un\empty {#1}_{#2}{#3}{#1}_{#4} \else
\ifx\deux\empty {#5}(#1_{#2}){#3}{#5}(#1_{#4}) \else
\ifx\trois\empty {#5}_{#6}(#1){#3}{#5}_{#7}(#1) \else
{#5_{#6}}(#1_{#2}){#3}{#5_{#7}}(#1_{#4}) \fi \fi \fi}
\def\rond{\raisebox{.3mm}{\scriptsize$\circ$}}

\def\dv#1#2{\langle {#1},{#2}\rangle}

\def\tk#1#2{{#2}\otimes _{#1}}

\def\ec#1#2#3#4#5{\def\un{#3#4#5}\def\deux{#3#5}\def\trois{#3}
\def\four{#2#4#5}\def\five{#2#5}\def\six{#2}\def\seven{#3#4}
\def\eight{#2#4} \def\nine{#2#3#4}
\ifx\nine\empty {\rm #1}_{#5} \else
\ifx\un\empty {\rm #1}({\goth #2}) \else
\ifx\deux\empty {\rm #1}({\goth #2}_{#4}) \else
\ifx\trois\empty {\rm #1}_{#5}({\goth #2}_{#4}) \else
\ifx\four\empty {\rm #1}(#3) \else
\ifx\five\empty {\rm #1}(#3_{#4}) \else
\ifx\six\empty {\rm #1}_{#5}(#3_{#4}) \else
\ifx\seven\empty {\rm #1}_{#5} ({\goth#2})\else
\ifx\eight\empty {\rm #1}_{#5}({#3})
\fi \fi \fi \fi \fi \fi \fi \fi \fi}
\def\hec#1#2#3#4#5{\def\un{#3#4#5}\def\deux{#3#5}\def\trois{#3}
\def\four{#2#4#5}\def\five{#2#5}\def\six{#2}\def\seven{#3#4}
\def\eight{#2#4} \def\nine{#2#3#4}
\ifx\nine\empty \hat{{\rm #1}}_{#5} \else
\ifx\un\empty \hat{{\rm #1}}({\goth #2}) \else
\ifx\deux\empty \hat{{\rm #1}}({\goth #2}_{#4}) \else
\ifx\trois\empty \hat{{\rm #1}}_{#5}({\goth #2}_{#4}) \else
\ifx\four\empty \hat{{\rm #1}}(#3) \else
\ifx\five\empty \hat{{\rm #1}}(#3_{#4}) \else
\ifx\six\empty \hat{{\rm #1}}_{#5}(#3_{#4}) \else
\ifx\seven\empty \hat{{\rm #1}}_{#5} ({\goth#2})  \else
\ifx\eight\empty \hat{{\rm #1}}_{#5}({#3})
\fi \fi \fi \fi \fi \fi \fi \fi \fi}
\def\e#1#2{\ec {#1}#2{}{}{}}
\def\es#1#2{\ec {#1}{}{#2}{}{}}

\def\Bb{\mathbb}
\def\goth{\mathfrak}
\def\cali{\mathcal}

\def\ran#1#2{\def\deux{#2} \ifx\deux\empty {\rm rk}\hskip .005em{{\goth #1}} \else {\rm rk}\hskip .005em{{\goth #1}_{#2}} \fi}
\def\bor#1#2{\def\deux{#2} \ifx\deux\empty {\rm b}_{{\goth #1}} \else {\rm b}_{{\goth #1}_{#2}} \fi}

\def\ai#1#2#3{\def\deux{#2#3} \def\trois{#3} \def\quatre{#2}
\ifx\deux\empty \es S{{\goth #1}}^{{\goth #1}} \else
\ifx\trois\empty \es S{{\goth #1}(#2)}^{{\goth #1}(#2)} \else
\ifx\quatre\empty \es S{{\goth #1}_{#3}}^{{\goth #1}_{#3}} \else
\es S{{\goth #1}_{#3}(#2)}^{{\goth #1}_{#3}(#2)} \fi \fi \fi}
\def\an#1#2{\def\deux{#2} \ifx\deux\empty {\cali O}_{#1} \else {\cali O}_{#1,#2} \fi }
\def\han#1#2{\def\deux{#2} \ifx\deux\empty {\hat{{\cali O}}}_{#1} \else {\hat{{\cali O}}}_{#1,#2} \fi }

\def\reg#1#2#3{\relax \def\un{#3} \relax \ifx\un\empty #1_{{\goth #2}} \else #1_{{\goth #2},{\goth #3}} \fi}

\def\dim{{\rm dim}\hskip .125em}

\def\dd{{\rm d}}
\def\rk{{\rm rk}\hskip .1em}
\def\ad{{\rm ad}\hskip .1em}

\def\reg{{\rm reg}}

\def\gtg#1{{\goth #1}\times {\goth #1}}
\def\cg#1{\tk {{\Bb C}}{\e S#1}\e S#1}
\def\pol#1#2#3{\def\deux{#2} \ifx\deux\empty {\rm pol}_{#3}{\rm S}({\goth #1})^{{\goth #1}} \else {\rm pol}_{#3}({\rm S}({\goth #1}_{#2}))^{{\goth #1}_{#2}} \fi}

\def\sqg#1{G\times _{{\bf B}}{\goth b}_{#1}}

\def\ord#1{{\rm ord}(#1)}

\begin{document}

\title[Nilpotent bicone]{Nilpotent bicone and characteristic submodule of a reductive
Lie algebra.}
\author[J-Y Charbonnel]{Jean-Yves Charbonnel}
\address{Universit\'e Paris 7 - CNRS \\
Institut de Math\'ematiques de Jussieu \\
Th\'eorie des groupes \\
Case 7012 \\ 2 Place Jussieu \\
75251 Paris Cedex 05, France}
\email{jyc@math.jussieu.fr}
\author[A. Moreau]{Anne Moreau}
\address{ETH Z\"urich\\
Departement Mathematik \\
HG G66.4\\
R\"amistrasse 101\\
8092 Z\"urich, Switzerland}
\email{anne.moreau@math.ethz.ch}



\vspace{-2.5cm}
\maketitle

\vspace{-1.5cm}
\begin{abstract}Let ${\goth g}$ be a finite dimensional complex reductive Lie algebra and $\mathrm{S}({\goth g})$ its
symmetric algebra. The {\it nilpotent bicone} of ${\goth g}$ is the subset of elements
$(x,y)$ of $\gtg g{}$ whose subspace generated by $x$ and $y$ is contained in the
nilpotent cone. The nilpotent bicone is naturally endowed with a scheme structure, as
nullvariety  of the augmentation ideal of the subalgebra of $\cg g{}$ generated by the
$2$-order polarizations of invariants of $\mathrm{S}({\goth g})$. The main result of this note is that
the nilpotent bicone is a complete intersection of dimension $3(\bor g{}-\rk\mathfrak{g})$, where
$\bor g{}$ and $\rk\mathfrak{g}$ are the dimension of Borel subalgebras and the rank of
${\goth g}$ respectively. This  affirmatively answers a conjecture of Kraft-Wallach
 concerning the nullcone \cite{KrW2}. In addition, we introduce and study in
this note the {\it characteristic submodule of ${\goth g}$}. The properties of the
nilpotent bicone and the characteristic submodule are known to be very important for the
understanding of the commuting variety and its ideal of definition. The main difficulty
encountered for this work is that the nilpotent bicone is not reduced. To deal with this
problem, we introduce an auxiliary reduced variety, {\it the principal bicone}. The
nilpotent bicone, as well as the principal bicone, are linked to jet schemes.  We study
their dimensions using arguments from motivic integration. Namely, we follow methods
developed by M. Musta{\c{t}}{\v{a}} in \cite{Mu}. At last, we give applications of our
results to invariant theory.

\end{abstract}

\vspace{-.8cm}
\tableofcontents
\section*{Introduction}

\renewcommand\thesubsection{\arabic{subsection}}
\subsection{Main results} Let ${\goth g}$ be a finite dimensional complex reductive Lie algebra, let
$G$ be its adjoint group and let ${\goth N}_{{\goth g}}$ be its nilpotent cone. For
$(x,y)$ in $\gtg g{}$, we denote by $P_{x,y}$ the subspace generated by $x$ and $y$. A
subset of $\gtg g{}$ is called {\it bicone} if it is stable under the maps \sloppy
\hbox{$(x,y)\mapsto (sx,ty)$} where $s$ and $t$ are in ${\Bb C}^*$. The
{\it  nilpotent bicone} of ${\goth g}$ is the subset
\begin{eqnarray*}
{\mathcal{N}}_{{\goth g}} & := & \{(x,y) \in \gtg g{}  \ \vert \ P_{x,y}  \subset
{\goth N}_{{\goth g}}\} \mbox{ .}
\end{eqnarray*}
Since ${\goth N}_{{\goth g}}$ is a $G$-invariant closed cone of ${\goth g}$, the subset
${\cali N}_{{\goth g}}$ is a closed bicone of $\gtg g{}$, invariant under the diagonal
action of $G$ on $\gtg g{}$.

Let $\e Sg$ be the symmetric algebra of ${\goth g}$ and let $\ai g{}{}$ be the subalgebra
of $G$-invariant elements of $\e Sg$. According to a Chevalley's result, the algebra
$\ai g{}{}$ is polynomial in $\rk\mathfrak{g}$ variables, where $\rk\mathfrak{g}$ is the rank of
${\goth g}$. Let $\poi p1{,\ldots,}{\ran g{}}{}{}{}$ be homogeneous generators of
$\ai g{}{}$ of degrees $\poi d1{,\ldots,}{\ran g{}}{}{}{}$ respectively. We can suppose that
the sequence $\poi d1{,\ldots,}{\ran g{}}{}{}{}$ is weakly increasing. For
$i=1,\ldots,\rk\mathfrak{g}$, the {\it $2$-order polarizations} $p_{i,m,n}$ of $p_i$ are the
unique elements of $(\cg g{})^{{\goth g}}$ satisfying the following relation:
\begin{eqnarray}\label{pi}
p_i(ax+by)&=&\sum\limits_{m+n=d_i} a^m b^n p_{i,m,n}(x,y) \mbox{ , }
\end{eqnarray}
for all $(a,b)$ in ${\Bb C}^2$ and $(x,y)$ in $\gtg g{}$. The nilpotent cone ${\goth N}_{{\goth g}}$ is the nullvariety in ${\goth g}$ of the ideal
generated by the polynomials $\poi p1{,\ldots,}{\ran g{}}{}{}{}$. Therefore, according to
Relation (\ref{pi}), ${\cali N}_{{\goth g}}$ is
the nullvariety in $\gtg g{}$ of the ideal of $\cg g{}$ generated by the polynomials
$p_{i,m,n}$, for $i=1,\ldots,\rk \mathfrak{g}$ and $m+n=d_i$. Thus, ${\cali N}_{{\goth g}}$ is
naturally endowed with a scheme structure. From here, we study ${\cali N}_{{\goth g}}$
as the subscheme of $\gtg g{}$ corresponding to the ideal generated by the polynomials
$p_{i,m,n}$, for $i=1,\ldots,\rk \mathfrak{g}$ and $m+n=d_i$.

Let $\bor g{}$ be the dimension of a Borel subalgebra of ${\goth g}$. By a classical
result \cite{Bou}, we have $d_1 + \cdots +d_{\ran g{}}=\bor g{}$. So ${\cali N}_{{\goth g}}$ is
the nullvariety in $\gtg g{}$ of $\bor g{}+\rk \mathfrak{g}$ polynomial functions. As a result, the
dimension of any irreducible component of ${\cali N}_{{\goth g}}$ is at least
$3(\bor g{}-\rk \mathfrak{g})$, since $\gtg g{}$ has dimension $2(2\bor g{}-\rk \mathfrak{g})$. The main result of
this note is the following theorem (see Theorem \ref{td}, (iv), Proposition \ref{pd2},
(ii), and Theorem \ref{tp1}):

\begin{theorem}\label{tint} The nilpotent bicone is a nonreduced complete intersection in
$\gtg{g}$ of dimension $3(\bor g{}-\rk \mathfrak{g})$. Moreover, the images of any irreducible
component of ${\cali N}_{{\goth g}}$ by the first and second projections from $\gtg{g}$ to
${\goth g}$ are equal to ${\goth N}_{{\goth g}}$.
\end{theorem}
This result affirmatively answers a conjecture \cite{KrW2}(Section \ref{sc}) of H. Kraft and
N. Wallach concerning the {\it nullcone} of the $G$-module $\gtg{g}$
(see Theorem \ref{ti2}). Clearly, the images of ${\cali N}_{{\goth g}}$ by the first and
second projections from $\gtg{g}$ to ${\goth g}$ are equal to ${\goth N}_{{\goth g}}$.
Theorem \ref{tint} specifies that this is true for any irreducible component of
${\cali N}_{{\goth g}}$. Notice that this statement is a much stronger result.

We introduce in addition in this note the {\it characteristic submodule}  of ${\goth g}$.
It is a $\cg g{}$-submodule of $\tk {{\Bb C}}{\cg g{}}{\goth g}$. The properties of the
nilpotent bicone and the characteristic submodule are known to be very important for the
understanding of the commuting variety. Recall that the {\it  commuting variety }
${\cali C}_{{\goth g}}$ of ${\goth g}$ is the set of elements $(x,y)$ of $\gtg g{}$ such
that  $[x,y]=0$. The commuting variety has been studied for many years. According to a
result of R. W. Richardson \cite{Ri1}, ${\cali C}_{{\goth g}}$ is irreducible. Moreover,
${\cali C}_{{\goth g}}$ is the nullvariety of the ideal  generated by the elements
$(x,y) \mapsto \dv v{[x,y]}$, where $v$ runs through ${\goth g}$.  An old unsolved
question is to know whether this ideal is prime \cite{LS}. In other words, we want to
know if this ideal is the ideal of definition of ${\cali C}_{{\goth g}}$ since
${\cali C}_{{\goth g}}$ is irreducible.  The study of the commuting variety and of its
ideal of definition is a main motivation for our work.

\subsection{Description of the paper} We denote by ${\goth g}_{\reg}$ the subset of regular elements of
${\goth g}$ and we set
$$\Omega_{{\goth g}} := \{ (x,y) \in {\goth g} \times {\goth g} \ \vert \
P_{x,y} \setminus \{0\} \subset {\goth g}_{\reg} \ , \ \dim P_{x,y} = 2 \} \mbox{ .}$$
The properties of the subset $\Omega_{{\goth g}}$
were studied by A. V. Bolsinov in \cite{Bol} and recently by D. I. Panyushev and O.
Yakimova in \cite{PYa}. The set of smooth points of ${\cali N}_{{\goth g}}$ with respect
to its scheme structure is the set of elements $(x,y) \in {\goth g} \times {\goth g}$ at
which the differentials of the $p_{i,m,n}$, for $i=1,\ldots,\rk\mathfrak{g}$ and $m+n=d_i$, are
linearly independent. The set $\Omega _{{\goth g}}$ is an open subset of $\gtg g{}$
and its intersection with the nilpotent bicone turns out to be the set of smooth points
of ${\cali N}_{{\goth g}}$ (see Proposition \ref{psc4} and Remark \ref{rsc4}). The
description of this subset is therefore very important for us. We give various properties
of the subset $\Omega_{{\goth g}}$ in Section 1. Next, we introduce in this section
the {\it characteristic submodule} $\mathrm{B}_{{\goth g}}$ of ${\goth g}$ (see Definition
\ref{dsc3}) whose study is heavily related to $\Omega_{{\goth g}}$. The
characteristic submodule was introduced by the first author a few years ago in a work on
the commuting variety. We prove here that B$_{{\goth g}}$ is a free $\cg g{}$-module
of rank $\bor g{}$, and we provide an explicit basis for B$_{{\goth g}}$
(see Theorem \ref{tsc3}). Even if this result is not directly useful for the nilpotent
bicone, it could be interesting in itself.

The results obtained in Section \ref{sc} concerning $\Omega_{{\goth g}}$ do not enable to provide
``enough" smooth points of ${\cali N}_{{\goth g}}$ to apply the criterion of Kostant
\cite{Kos2} (see Remark \ref{rsc4}). Actually, the intersection of ${\cali N}_{{\goth g}}$
and $\Omega_{{\goth g}}$ is an empty set in many cases. This observation
makes the study of the scheme ${\cali N}_{{\goth g}}$ very difficult. In order to deal
with this problem, we introduce an auxiliary reduced scheme, the {\it principal bicone}
of ${\goth g}$
\begin{eqnarray*}
{\cali X}_{{\goth g}} & := & \{(x,y) \in \gtg g{}  \ \vert \ P_{x,y}  \subset
{\goth X}_{{\goth g}} \} \mbox{ ,}
\end{eqnarray*}
where ${\goth X}_{{\goth g}}$ is the {\it principal cone}
 of ${\goth g}$, that is the Zariski closure of the set of principal semisimple
elements (see Definitions \ref{dsu1} and \ref{dsu2}). The study of the varieties
${\goth X}_{{\goth g}}$ and ${\cali X}_{{\goth g}}$ is the main topic of Section \ref{su}.

We observe that ${\cali N}_{{\goth g}}$ and ${\cali X}_{{\goth g}}$ can be identified with
constructible subsets of jet schemes of ${\goth N}_{{\goth g}}$ and
${\goth X}_{{\goth g}}$ respectively. In \cite{Mu}, M. Musta{\c{t}}{\v{a}} uses the
theory of motivic integration, as developed by M. Kontsevich \cite{Kon}, J. Denef and
F. Loeser \cite{DL2}, and V. Batyrev \cite{Ba}, to prove a result concerning the jet
schemes of locally complete intersections. In particular, his result can be applied to
${\goth N}_{{\goth g}}$ and ${\goth X}_{{\goth g}}$. Thus, in the appendix of \cite{Mu},
D. Eisenbud and E. Frenkel extend results of Kostant concerning the nilpotent cone of a
reductive Lie algebra in the setting of jet schemes. In Section \ref{jm}, after reviewing
some facts about motivic integration, we state technical results concerning motivic
integrals, following methods developed by M. Musta{\c{t}}{\v{a}} in \cite{Mu}, which
will be useful for Section \ref{d}. Using all of this, we prove in
Section \ref{d} the following theorem  (see Theorem \ref{td}, (i), and Proposition
\ref{pd2}, (i)):

\begin{theorem}\label{t2int}
The principal bicone of ${\goth g}$ is a reduced complete intersection of
dimension \sloppy \hbox{$3(\bor g{}-\rk \mathfrak{g}+1)$}. Moreover, the images of any irreducible
component of ${\cali X}_{{\goth g}}$ by the first and second projection from $\gtg{g}$ to
${\goth g}$ are equal to the principal cone ${\goth X}_{{\goth g}}$.
\end{theorem}

Then, we deduce Theorem \ref{tint} from Theorem \ref{t2int}. We give in Section \ref{i}
applications of our results to invariant theory, mainly in relation with the
nullcone of the $G$-module ${\goth g}\times {\goth g}$. In Section \ref{p}, we obtain
additional properties about the irreducible components of the nilpotent bicone. We prove
that ${\cali N}_{{\goth g}}$ is a nonreduced scheme (see Theorem \ref{tp1}) and we give a lower
bound for the number of irreducible components of ${\cali N}_{{\goth g}}$
(see Proposition \ref{pp2}).\\

We precise that Section \ref{jm} is entirely independent on Sections \ref{sc} and \ref{su}
and does not deal with Lie algebras.

\subsection{Additional notations and conventions}\label{notation} In this note, the
ground field is ${\Bb C}$. All topological terms refer to the Zariski topology. If $X$
is an algebraic variety, an open subset of $X$ whose complementary in $X$ has
codimension at least $2$ is called {\it a big open subset} of $X$.\\

The set of integers, the set of nonnegative integers and the set of positive integers are
denoted by ${\Bb Z}$, ${\Bb N}$ and ${\Bb N}^{*}$ respectively. As usual, the subset
of non zero elements of ${\Bb C}$ is denoted by ${\Bb C}^{*}$.\\

If $E$ is a finite set, its cardinality is denoted by $\vert E \vert$.\\

For $x$ in ${\goth g}$, we denote by ${\goth g}(x)$ the centralizer of $x$ in
${\goth g}$. Thus $x$ belongs to ${\goth g}_{\reg}$ if and only if ${\goth g}(x)$ has
dimension $\ran g{}$.\\

The action of $G$ in $\gtg g{}$ will be always the diagonal action. For $x$ in ${\goth g}$
or in $\gtg g{}$, we denote by $G.x$ its $G$-orbit under the corresponding $G$-action.\\

If ${\goth g}$ is commutative, then ${\cali N}_{{\goth g}}$ is reduced to $\{(0,0)\}$. Henceforth, we suppose that ${\goth g}$ is not commutative.\\

Let $\poi {{\goth g}}1{,\ldots,}{m}{}{}{}$ be the simple factors of ${\goth g}$. Since
the nilpotent cone ${\goth N}_{{\goth g}}$ is the product of the nilpotent cones
$\poi {{\goth N}}{{\goth g}_{1}}{,\ldots,}{{\goth g}_{m}}{}{}{}$,
${\cali N}_{{\goth g}}$ is the product of the nilpotent bicones
$\poi {{\cali N}}{{\goth g}_{1}}{,\ldots,}{{\goth g}_{m}}{}{}{}$. Furthermore, as the
equality
$$3(\bor g{}-\rk \mathfrak{g})=3(\bor g{1}-\rk \mathfrak{g}_1) + \cdots +3(\bor g{m}-\rk \mathfrak{g}_{m}) $$
holds, it suffices to prove Theorem \ref{tint} in the case where
${\goth g}$ is simple. It will be sometimes useful to make the assumption that
${\goth g}$ is simple (see e.g. Sections \ref{su}, \ref{i} and \ref{p}).\\

The first and second projections from $\gtg g{}$ to ${\goth g}$ are denoted
by $\varpi_1$ and $\varpi_2$ respectively. The following lemma will be very handy
throughout the note.

\begin{lemma}\label{lint} Let $Y$ be a $G$-invariant closed bicone in $\gtg{g}$. Then the
subsets $\varpi_{1}(Y)$ and $\varpi _{2}(Y)$ are $G$-invariant closed cones of
${\goth g}$.
\end{lemma}

\begin{proof}
As $Y$ is a $G$-invariant bicone, $\varpi_{1}(Y)$ and $\varpi _{2}(Y)$ are $G$-invariant
cones of ${\goth g}$. Moreover, $\varpi_{1}(Y)\times \{0\}$ and
$\{0\}\times \varpi _{2}(Y)$ are the intersections of $Y$ with
${\goth g}\times \{0\}$ and $\{0\}\times {\goth g}$ respectively, since $Y$ is a closed
bicone. So $\varpi_{1}(Y)$ and $\varpi _{2}(Y)$ are closed subsets of ${\goth g}$.
\end{proof}

We fix a principal ${\goth {sl}}_2$-triple $(e,h,f)$ of ${\goth g}$. Thus, the
following relations
$$[h,e]=2e \mbox{ , } \hspace{.5cm} [e,f]=h \mbox{ , } \hspace{.5cm} [h,f]=-2f \mbox{ ,}$$
are satisfied, $e$ and $f$ are regular nilpotent elements and $h$ is a regular
semisimple element of ${\goth g}$. If $\ad$ is the adjoint representation of ${\goth g}$,
then $\ad h$ induces a ${\Bb Z}$-grading on ${\goth g}$ and we have
${\goth g}=\sum\limits_{i \in {\Bb Z}} {\goth g}_i$, where ${\goth g}_{i}$ is the
$i$-eigenspace of $\ad h$. Moreover, all the eigenvalues of $\ad h$ are
even integers. The centralizer $\mathfrak{g}(h)={\goth g}_{0}$ of $h$ in ${\goth g}$ is a
Cartan subalgebra, that we will also denote by ${\goth h}$. The unique Borel subalgebra
containing $e$ is
$${\goth b}:= {\goth h} \oplus  \sum\limits_{i > 0} {\goth g}_i \mbox{ , }$$
and the nilpotent radical of ${\goth b}$ is
${\goth u} := \sum\limits_{i > 0} {\goth g}_i$. We additionally set
$${\goth b}_{-} \ := \ \sum\limits_{i \leq 0} {\goth g}_i \mbox{ , } \
{\goth u}_- \ := \ \sum\limits_{i < 0} {\goth g}_i \mbox{ .}$$
Let ${\bf B}$ and ${\bf B}_{-}$ be the normalizers of
${\goth b}$ and ${\goth b}_{-}$ in $G$ and let ${\bf H}$ and $N_{G}({\goth h})$ be the
centralizer and the normalizer of ${\goth h}$ in $G$. Then the quotient $W_G({\goth h})$
of $N_{G}({\goth h})$ by ${\bf H}$ is the Weyl group of ${\goth g}$ with respect to
${\goth h}$.

Let ${\cali R}$ be the root system of $({\goth g},{\goth h})$, let ${\cali R}_+$ be the
positive root system of ${\cali R}$ defined by ${\goth b}$, and let $\Pi$ be the basis of
${\cali R}_+$. For any $\alpha $ in ${\cali R}$, we denote by ${\goth g}^{\alpha}$ the
corresponding root subspace. Denote by $w_0$ the longest element of $W_{G}({\goth h})$
with respect to $\Pi$. Then $w_{0}(\Pi )$ is equal to $-\Pi $ and for any representative
$g_{0}$ of $w_{0}$ in $N_{G}({\goth h})$, $g_{0}({\goth b})$ is equal to
${\goth b}_{-}$.\\

Let $\dv ..$ be a nondegenerate $G$-invariant bilinear form on
$\gtg g{}$ which extends the Killing form of the semisimple part of ${\goth g}$. In the
remainder of this note, the orthogonality will refer to $\dv ..$.\\

\subsection*{Acknowledgment} We are grateful to Michel Duflo for bringing \cite{Mu} to
our attention. We would like also to thank Fran\c{c}ois
Loeser for his e-mails concerning motivic integration questions, and Karin Baur for her
comments. In addition, we thank Vladimir Popov for his support and the referees for their
numerous and judicious advices and their careful attention to our note.

\section{Characteristic submodule}\label{sc}
In this section, we introduce the characteristic submodule of ${\goth g}$ and we describe
some of its properties.

\subsection{Preliminaries} Let us recall that $\poi p1{,\ldots,}{\ran g{}}{}{}{}$ are homogeneous
generators of $\ai g{}{}$ of degrees $\poi d1{,\ldots,}{\ran g{}}{}{}{}$ respectively such
that the sequence $\poi d1{,\ldots,}{\ran g{}}{}{}{}$ is weakly increasing. For
$i=1,\ldots,\rk \mathfrak{g}$, let $\varepsilon_i$ be the element of
$\tk {{\Bb C}}{\e Sg}{\goth g}$ defined by the following relation:
$$\dv {\varepsilon_i(x)}v = p_{i}'(x)(v) \mbox{ ,}$$
for $x,v$ in ${\goth g}$, where $p_{i}'(x)$ is the differential of $p_i$ at $x$. The
first statement of the following lemma comes from \cite{Ri3}(Lemma 2.1) while the second
statement comes from \cite{Kos2}(Theorem 9):

\begin{lemma}\label{lsc1}
Let $x$ be in ${\goth g}$.

{\rm i)} For $i=1,\ldots,\rk \mathfrak{g}$, $\varepsilon _{i}(x)$ belongs to the center of
${\goth g}(x)$.

{\rm ii)} The elements $\poi x{}{,\ldots,}{}{\varepsilon }{1}{\ran g{}}$ are linearly
independent in ${\goth g}$ if and only if $x$ is regular. Moreover, if so,
$\poi x{}{,\ldots,}{}{\varepsilon }{1}{\ran g{}}$ is a basis of ${\goth g}(x)$.
\end{lemma}

Let $i$ be in $\{1,\ldots,\rk \mathfrak{g}\}$. The $2$-order polarizations $p_{i,m,n}$ of $p_i$ are
defined by the following relation:
$$p_i(ax+by)=\sum\limits_{m+n=d_i} a^m b^n p_{i,m,n}(x,y) \mbox{ ,}$$
for any $(a,b)$ in ${\Bb C}^2$ and any $(x,y)$ in $\gtg g{}$. For $(m,n)\in{\Bb N}^2$
such that $m+n=d_i$, let $\varepsilon_{i,m,n}$ be the element of
$\tk {{\Bb C}}{\cg g{}}{\goth g}$ such that the linear functional
$v\mapsto \dv {\varepsilon_{i,m,n}(x,y)}v$ on ${\goth g}$ is the differential at $x$ of
the function $x\mapsto p_{i,m,n}(x,y)$ for $(x,y)$ in $\gtg g{}$. In particular,
$\varepsilon _{i,0,d_{i}}=0$ for $i=1,\ldots, \rk \mathfrak{g}$. In addition, since
$\poi d1{+\cdots +}{\ran g{}}{}{}{}$ is equal to $\bor g{}$, the cardinality of the family
$\{\varepsilon_{i,1,d_i-1},\ldots,\varepsilon_{i,d_i,0}\mbox{ , }i=1,\ldots,\rk \mathfrak{g}\}$ is
equal to $\bor g{}$.

\begin{lemma}\label{l2sc1} For $i=1,\ldots,\rk \mathfrak{g}$ and $(x,y)\in \gtg g{}$, we have
$$ \varepsilon _{i}(ax+by) =  \sum_{m=1}^{d_{i}}
a^{m-1}b^{d_{i}-m}\varepsilon _{i,m,d_{i}-m}(x,y)\mbox{ ,}$$
for any $(a,b)$ in ${\Bb C}^{2}$. In particular,
$\varepsilon _{i,1,d_{i}-1}(x,y)=\varepsilon _{i}(y)$
and \sloppy \hbox{$\varepsilon _{i,d_{i},0}(x,y)=\varepsilon _{i}(x)$}.
\end{lemma}

\begin{proof} Let $v$ be in ${\goth g}$. For any $(a,b)$ in ${\Bb C}^*\times {\Bb C}$,
we have
\begin{eqnarray*} \dv {\varepsilon _{i}(ax+by)}v &=&
\frac{\dd }{\dd t}p_{i}(ax+by+tv) \left \vert \right. _{t=0}\\
&= &\sum_{m+n=d_i}
a^{m}b^{n} \frac{\dd }{\dd t}p_{i,m,n}(x+ta^{-1}v,y) \left \vert \right.
_{t=0} \\
&= &\sum_{m+n=d_i}
a^{m}b^{n} \dv {\varepsilon _{i,m,n}(x,y)}{a^{-1}v} \mbox{ , }
\end{eqnarray*}
whence the lemma, since $\dv ..$ is a nondegenerate $G$-invariant bilinear form on $\gtg g{}$.
\end{proof}

Recall that we have introduced the subset
$$\Omega_{{\goth g}} := \{ (x,y) \in {\goth g} \times {\goth g} \ \vert \
P_{x,y} \setminus \{0\} \subset {\goth g}_{\reg} \ , \ \dim P_{x,y} = 2 \} \mbox{ , }$$
which is clearly $G$-invariant. We denote by ${\goth h}'$,
${\goth g}_{2}'$ and ${\goth g}_{-2}'$ the intersections of ${\goth g}_{\reg}$ with
${\goth h}$, ${\goth g}_{2}$ and ${\goth g}_{-2}$ respectively. Set:
$${\goth u}_{+} \ := \ \sum\limits_{i \geq 4} {\goth g}_i \mbox{ .}$$

In the following lemma, we explicitly provide elements of $\Omega_{{\goth g}}$ .

\begin{lemma}\label{l3sc1} {\rm i)} If $x\in {\goth h}'+{\goth u}$ and
$y\in {\goth g}_{2}'+{\goth u}_+$, then $(x,y) \in \Omega_{{\goth g}}$.

{\rm ii)} If $x \in {\goth g}_{-2}'$ and $y \in {\goth g}_{2}'$,
then $(x,y) \in \Omega_{{\goth g}}$.
\end{lemma}

\begin{proof}
i) The elements $x$ and $y$ are clearly linearly independent. In addition, for any $t$ in
${\Bb C}$, $x+ty$ is conjugate to $x$ under $G$. So $x+ty$ is regular for any $t$ in
${\Bb C}$. Since $y$ is regular, $(x,y)$ belongs to $\Omega_{{\goth g}}$.

ii) The elements $x$ and $y$ are clearly linearly independent.
As $x$ is regular, $x+sy$ is regular for any $s$ in an open subset of ${\Bb C}$
containing $0$. Let $t \mapsto g(t)$ be the one-parameter subgroup of ${\bf H}$ generated
by $\ad h$. As ${\goth g}_{\reg}$ is $G$-invariant, for any $t$ in
${\Bb C}^{*}$, $g(t)(x+sy)$ is regular for any $s$ in an open subset of
${\Bb C}$ containing $0$. So, from the relations $[h,x]=-2x$ and $[h,y]=2y$, we deduce
that for any $t$ in ${\Bb C}^{*}$, $t^{-2}x+st^2 y$ is regular for any $s$
in an open subset of ${\Bb C}$ containing $0$. As ${\goth g}_{\reg}$ is an open cone,
$x+st^{4}y$ is consequently regular for any $s$ in an open subset of ${\Bb C}$
containing $\{0\}$. So $x+ty$ is regular for any $t$ in ${\Bb C}$. As $y$ is regular, we
deduce that $(x,y)$ belongs to $\Omega_{{\goth g}}$.
\end{proof}

\begin{remark}\label{rsc1} The subset $\Omega_{{\goth g}}$ is stable under the
involution \sloppy \hbox{$(x,y) \mapsto (y,x)$}. So, by Lemma \ref{l3sc1}, the images of
the subsets $({\goth h}'+{\goth u}) \times ({\goth g}_{2}'+{\goth u}_+)$ and
${\goth g}_{-2}' \times {\goth g}_{2}'$ by this involution are contained in
$\Omega_{{\goth g}}$. For example, the elements $(e,h)$, $(f,h)$ and $(e,f)$ are in
$\Omega_{{\goth g}}$.\\
\end{remark}

The following lemma is well-known:

\begin{lemma}\label{l4sc1}
Let $x$ be in ${\goth g}_{\reg}$.

{\rm i)} If $x$ belongs to ${\goth b}$, then ${\goth g}(x)$ is contained in ${\goth b}$.

{\rm ii)} If $x$ belongs to ${\goth u}$, then ${\goth g}(x)$ is contained in ${\goth u}$.
\end{lemma}

For $(x,y)$ in $\gtg g{}$, we denote by ${\goth V}(x,y)$ the subspace generated by the
elements
$$\poi {x,y}{}{,\ldots,}{}{\varepsilon }{i,1,d_{i}-1}{i,d_{i},0} \mbox{ , }
i=1,\ldots,\rk \mathfrak{g} \mbox{ ,}$$
and we set:
$${\goth V}'(x,y) := \sum\limits_{(a,b) \in {\Bb C}^2\setminus{(0,0)}} {\goth g}(ax+by)
\mbox{ .}$$
We collect in the following lemma some results concerning the subset
$\Omega_{{\goth g}}$ in part obtained by A. V. Bolsinov in \cite{Bol} and recently by
D. I. Panyushev and O. Yakimova in \cite{PYa}:

\begin{lemma}\label{l5sc1}
Let $(x,y)$ be in $\gtg g{}$.

{\rm i)} The subspace ${\goth V}(x,y)$ has dimension at most $\bor g{}$. Moreover, it has
dimension $\bor g{}$ if and only if $(x,y)$ belongs to $\Omega _{{\goth g}}$.

{\rm ii)} The subspace ${\goth V}(x,y)$ is contained in ${\goth V}'(x,y)$. Moreover,
the equality occurs when $(x,y)$ belongs to $\Omega _{{\goth g}}$.

{\rm iii)} If $(x,y)$ belongs to $\gtg b{}$, then ${\goth V}(x,y)$ is contained in
${\goth b}$.
\end{lemma}

\begin{proof}
i) As the sum of $\poi d1{,\ldots,}{\ran g{}}{}{}{}$ is equal to $\bor g{}$, ${\goth V}(x,y)$
has dimension at most $\bor g{}$. Let us suppose that $(x,y)$ belongs
to $\Omega _{{\goth g}}$. Then by \cite{PYa}(Theorem 2.4), the subspace ${\goth V}(x,y)$
has dimension $\bor g{}$. Conversely, let us suppose that ${\goth V}(x,y)$ has dimension
$\bor g{}$. In particular, $\poi x{}{,\ldots,}{}{\varepsilon }{1}{\ran g{}}$ are linearly
independent. Hence by Lemma \ref{lsc1}, (ii), $x$ is regular. By Lemma \ref{l2sc1}, for
any $(a,b)$ in ${\Bb C}^{2}\backslash \{(0,0)\}$, ${\goth V}(x,y)$ is equal to
${\goth V}(ax+by,y)$. Hence $(x,y)$ belongs to $\Omega _{{\goth g}}$.

ii) By Lemma \ref{lsc1}, ${\goth V}'(x,y)$ contains $\varepsilon _{i}(ax+by)$ for
$i=1,\ldots,\rk \mathfrak{g}$ and any $(a,b)$ in ${\Bb C}^{2}$. Hence by Lemma \ref{l2sc1},
${\goth V}'(x,y)$ contains ${\goth V}(x,y)$. Moreover, by (i) and Lemma \ref{lsc1},
${\goth V}(x,y)$ is equal to ${\goth V}'(x,y)$ when $(x,y)$ belongs to
$\Omega _{{\goth g}}$.

iii) Let us suppose that $(x,y)$ belongs to $\gtg b{}$. By Lemma \ref{l3sc1} and Remark
\ref{rsc1}, the intersection of $\Omega _{{\goth g}}$ and $\gtg b{}$ is a nonempty open
subset. Moreover, when $(x,y)$ belongs to this intersection, for any $(a,b)$ in
${\Bb C}^{2}\backslash \{(0,0)\}$, ${\goth g}(ax+by)$ is contained in ${\goth b}$, by
Lemmas \ref{lsc1}, (ii) and \ref{l4sc1}, (i). Hence
for $i=1,\ldots,\rk \mathfrak{g}$ and $m=1,\ldots,d_{i}$, $\varepsilon _{i,m,d_{i}-m}(x,y)$ belongs
to ${\goth b}$, by (ii). So ${\goth V}(x,y)$ is contained in ${\goth b}$.
\end{proof}

\subsection{Closed irreducible subsets of $\gtg s{}$ invariant under the actions of
${\bf S}$ and ${\rm GL}_{2}({\Bb C})$} The following automorphisms give an action of ${\rm GL}_{2}({\Bb C})$ in
$\gtg g{}$:
$$(x,y) \longmapsto (ax+by,cx+dy), \textrm{ where }
g=\left[\begin{array}{cc}a &b\\ c & d\end{array}\right] \mbox{ .}$$
Let ${\goth s}$ be the subspace of ${\goth g}$ generated by $e$, $h$, $f$ so that
${\goth s}$ is isomorphic to ${\goth {sl}}_2({\Bb C})$. Let ${\bf S}$ be the closed
connected subgroup of $G$ whose Lie algebra is $\ad {\goth s}$. We start this subsection
by describing the closed irreducible subsets of $\gtg s{}$ invariant under the actions of
${\bf S}$ and ${\rm GL}_{2}({\Bb C})$. This will be used in Corollary \ref{csc2} and
Lemma \ref{l2su2}.

Let $T_{3}$ be the closed bicone of $\gtg s{}$ generated by the diagonal of
${\goth N}_{{\goth s}}\times {\goth N}_{{\goth s}}$, let $T_{4}$ be the closed bicone of
$\gtg s{}$ generated by the diagonal of ${\goth s} \times {\goth s}$ and let $T_{5}$ be
the subset of elements $(x,y)$ of $\gtg{s}$ such that $x$ and $y$ belong to the same Borel
subalgebra of ${\goth s}$. As defined, the sets $T_{3}$, $T_{4}$ and $T_{5}$ are
irreducible closed subsets of $\gtg s{}$, invariant under the actions of ${\bf S}$ and
${\rm GL}_{2}({\Bb C})$. Moreover, they have dimension $3$, $4$ and $5$ respectively. The
verification of these claims is left to the reader.

\begin{lemma}\label{lsc2} The subsets $\{0\}$, $T_{3}$, $T_{4}$, $T_{5}$ and $\gtg{s}$ are
the only nonempty irreducible closed subsets of $\gtg s{}$, invariant under the actions
of ${\bf S}$ and ${\rm GL}_{2}({\Bb C})$.
\end{lemma}

\begin{proof} Let us first remark that ${\goth N}_{{\goth s}}$ is the only proper
${\bf S}$-invariant closed cone of ${\goth s}$. Let $T$ be a nonempty irreducible closed
subset of $\gtg s{}$, invariant under the actions of ${\bf S}$ and
${\rm GL}_{2}({\Bb C})$.  By Lemma \ref{lint} and the preceding remark, if
$0 < \dim T < 4$, then $\varpi _{1}(T)$ is equal to ${\goth N}_{{\goth s}}$, whence
$T=T_{3}$ by ${\rm GL}_{2}({\Bb C})$-invariance. We assume now $\dim T \geq 4$.
As $\varpi _{1}(T) \supset {\goth N}_{{\goth s}}$ and $\dim T \geq 4$, $T$ contains an
element $(e,ae+bf+ch)$, with $b$ or $c$ different from $0$. Hence $T$ contains
$(e,bf+ch)$. If $c\not= 0$, then $\varpi _{1}(T)$ contains $bf+ch$ which is semisimple.
Otherwise, $T$ contains $(e,f)$ and $\varpi _{1}(T)$ contains $e+f$ which is semisimple,
too. So, in any case, $\varpi _{1}(T)={\goth s}$. In particular, $T\supset T_{4}$ and the
equality holds as soon as $\dim T = 4$. At last, let us suppose $\dim T \geq 5$. As
$\varpi _{1}(T)={\goth s}$, there exist $a$, $b$, $c$ in ${\Bb C}$ such that
$(h,ae+bf+ch)$ belongs to $T$ with $a$ or $b$ different from $0$. If $ab=0$, using the
invariance of $T$ under ${\bf S}$ and ${\mathrm {GL}}_{2}({\Bb C})$, we deduce that
$\{h\}\times {\goth b}_{{\goth s}}$ is contained in $T$.
In this case, $T$ contains $T_{5}$. If $ab\neq 0$, using the invariance of $T$ under
${\mathrm {GL}}_{2}({\Bb C})$ and the one-parameter subgroup of ${\bf H}$ generated by
$\ad h$, we deduce that $(h,at^{2}e+bt^{-2}f)$ belongs to $T$ for any $t$ in
${\Bb C}^{*}$. So $\{h\}\times {\goth s}$ is contained in $T$ and $T$ has dimension $6$,
whence the lemma.
\end{proof}

\begin{corollary}\label{csc2} The subset $\Omega _{{\goth g}}$ is a big open subset of
$\gtg g{}$.
\end{corollary}

\begin{proof} Suppose that $\Omega _{{\goth g}}$ is not a big open subset of
$\gtg{g}$. Then, the complementary of
$\Omega _{{\goth g}}$ in $\gtg{g}$ has an irreducible component $\Sigma $ of codimension
$1$ in $\gtg{g}$. As $\Omega _{{\goth g}}$ is invariant under the action of $G$
and GL$_{2}({\Bb C})$, $\Sigma $ is invariant under these actions too. The intersection
$T$ of $\Sigma $ and $\gtg{s}$ contains $(0,0)$. So $T$ is a nonempty closed cone of
$\gtg s{}$, invariant under the actions of ${\bf S}$ and GL$_{2}({\Bb C})$. As $\Sigma $
is an hypersurface of $\gtg{g}$, $T$ has codimension  at most $1$ in $\gtg{s}$. Hence
Lemma \ref{lsc2}, (ii), implies that $(h,e)$ belongs to $T$. But by Lemma \ref{l3sc1},
(i), $(h,e)$ belongs to $\Omega _{{\goth g}}$, whence the expected contradiction.
\end{proof}

\begin{remark}\label{rsc2} We can also deduce Corollary \ref{csc2} from
\cite{PYa}. Indeed, according to \cite{PYa}(Lemma 3.1), for any $x \in {\goth g}_{\reg}$,
the set \sloppy \hbox{$\{y \in {\goth g} \ \vert \ (x,y) \in \Omega_{{\goth g}}\}$} is a
big open subset of ${\goth g}$, whence we can readily deduce Corollary \ref{csc2}.
\end{remark}

\subsection{Characteristic submodule} By a result of J. Dixmier \cite{Di}(\S 2), the $\e Sg$-submodule of
elements $\varphi $ in $\tk {{\Bb C}}{\e Sg}{\goth g}$ such that $\varphi(x)$ belongs to
${\goth g}(x)$ for all $x$ in ${\goth g}$ is a free module of basis
$\poi {\varepsilon }1{,\ldots,}{\ran g{}}{}{}{}$.

\begin{definition}\label{dsc3} The {\it characteristic  submodule} $\mathrm{B}_{{\goth g}}$ of
${\goth g}$ is the $\cg g{}$-submodule of elements $\varphi $ in
$\tk {{\Bb C}}{\cg g{}}{\goth g}$ such that $\varphi(x,y)$ belongs to ${\goth V}'(x,y)$,
for any $(x,y)$ in a nonempty open subset of $\gtg g{}$.
\end{definition}

The following result can be viewed as a generalization of the previous result of
J. Dixmier.

\begin{theorem}\label{tsc3} The submodule $\mathrm{B}_{{\goth g}}$ of
$\tk {{\Bb C}}{\cg g{}}{\goth g}$ is a free \sloppy \hbox{$\cg g{}$}-module of rank
$\bor g{}$. Moreover, the family
$\{ \poi {\varepsilon }{i,1,d_{i}-1}{,\ldots,}{i,d_{i},0}{}{}{} $, \sloppy
\hbox{$i=1,\ldots,\rk\mathfrak{g}\}$} is a basis of $\mathrm{B}_{{\goth g}}$.
\end{theorem}

\begin{proof} By Lemma \ref{l5sc1},
(ii), for $i=1,\ldots,\rk \mathfrak{g}$ and $m=1,\ldots,d_{i}$,
$\varepsilon_{i,m,d_{i}-m}$ belongs to $\mathrm{B}_{{\goth g}}$. Moreover, by Lemma
\ref{l5sc1}, (i), these elements are linearly independent over $\cg g{}$. It remains to
prove that they generate $\mathrm{B}_{{\goth g}}$ as $\cg g{}$-module. Let $\varphi$ be in
$\mathrm{B}_{{\goth g}}$. By Lemma \ref{l5sc1}, (i) and (ii), $\varphi(x,y)$ belongs to
${\goth V}(x,y)$ for any $(x,y)$ in $\Omega_{{\goth g}}$. So there exist regular
functions $\psi_{i,m,d_i-m}$ on $\Omega _{{\goth g}}$ for  $i=1,\ldots,\rk \mathfrak{g}$ and
$m=1,\ldots,d_i$, such that:
$$\varphi(x,y)=\sum\limits_{i=1,\ldots, \ran g{} \atop m=1,\ldots,d_i }
\psi_{i,m,d_i-m}(x,y) \varepsilon_{i,m,d_i-m}(x,y) \mbox{ ,}$$
for any $(x,y)$ in $\Omega_{{\goth g}}$. By Corollary \ref{csc2},
$\Omega_{{\goth g}}$ is a big open subset of $\gtg g{}$. Hence the regular functions
$\psi_{i,m,d_i-m}$ have regular extensions to $\gtg g{}$ since $\gtg g{}$ is normal. As
a result, the family $\{ \poi {\varepsilon }{i,1,d_{i}-1}{,\ldots,}{i,d_{i},0}{}{}{}
\mbox{ , }i=1,\ldots,\rk \mathfrak{g}\}$ is a basis of $\mathrm{B}_{{\goth g}}$.
\end{proof}

\subsection{Smooth points of ${\cali N}_{{\goth g}}$}\label{sc4} In this subsection, we establish a link between the open subset
$\Omega_{{\goth g}}$ and the nilpotent bicone. For $i=1,\ldots,\rk\mathfrak{g}$, we denote by
$\sigma _{i}$ the map
$$ \gtg{g} \longrightarrow {\Bb C}^{d_{i}+1} \mbox{ , }
(x,y) \mapsto (\poi {x,y}{}{,\ldots,}{}{p}{i,0,d_{i}}{i,d_{i},0}) \mbox{ .}$$

\begin{lemma}\label{lsc4} For $(x,y)$ in $\gtg{g}$ and $i=1,\ldots,\rk\mathfrak{g}$, the
differential of $\sigma _{i}$ at $(x,y)$ is the linear map
$$\begin{array}{lcl} (v,w) &\longmapsto &(\ \dv {\varepsilon _{i,1,d_{i}-1}(x,y)}w\ , \
\dv {\varepsilon _{i,1,d_{i}-1}(x,y)}v + \dv {\varepsilon _{i,2,d_{i}-2}(x,y)}w \ ,\\
&&\\
&&\cdots\ ,\  \dv {\varepsilon _{i,d_{i}-1,1}(x,y)}v +
\dv {\varepsilon _{i,d_i,0}(x)}w \ ,\ \dv {\varepsilon _{i}(x)}v \ ) \mbox{ . }
\end{array} $$
\end{lemma}

\begin{proof} For $(x,y)$ in $\gtg{g}$, we denote by $p'_{i,m,d_{i}-m}(x,y)$ the
differential of $p_{i,m,d_{i}-m}$ at $(x,y)$. From Lemma \ref{l2sc1} and the equality
$$p_{i}(tx+y) = \sum_{m=0}^{d_{i}} t^{m}p_{i,m,d_{i}-m}(x,y) $$
for $(x,y)$ in $\gtg{g}$, we deduce the equality
\begin{eqnarray*}
\sum_{m=1}^{d_{i}} t^{m} \dv {\varepsilon _{i,m,d_{i}-m}(x,y)}{v+t^{-1}w}&=&
\sum_{m=0}^{d_{i}} t^{m} p'_{i,m,d_{i}-m}(x,y)(v,w) \mbox{ ,}
\end{eqnarray*}
for $x,y,v,w$ in ${\goth g}$ and $t$ in ${\Bb C}^{*}$, since
$p_{i,0,d_{i}}$ is the map $(x,y)\mapsto p_{i}(y)$. Hence we get, for
$m=0,\ldots,d_{i}-1$,
$$p'_{i,m,d_{i}-m}(x,y)(v,w) = \dv {\varepsilon _{i,m,d_{i}-m}(x,y)}v +
\dv {\varepsilon _{i,m+1,d_{i}-m-1}(x,y)}w \mbox{ .}$$
In addition $p'_{i,d_{i},0}(x,y)$ is the linear functional
$(v,w) \mapsto \dv {\varepsilon _{i}(x)}v$.
\end{proof}

Let $\sigma$ be the map:
$$ \gtg{g} \longrightarrow {\Bb C}^{\bor g{}+\ran g{}} \mbox{ , }
(x,y) \mapsto (\poi {x,y}{}{,\ldots,}{}{\sigma }{1}{\ran g{}}) \mbox{ .}$$

\begin{proposition}\label{psc4}
Let $(x,y)$ be in $\gtg g{}$. Then $\sigma $ is smooth at $(x,y)$ if and only if  $(x,y)$
belongs to $\Omega_{{\goth g}}$.
\end{proposition}

\begin{proof}
We denote by $\sigma '(x,y)$ the differential of $\sigma $ at $(x,y)$ and we denote
by $\ker \sigma'(x,y)$ its kernel in $\gtg g{}$. Let us suppose that $(x,y)$ belongs to
$\Omega_{{\goth g}}$. For $v$ in ${\goth g}$, we denote by $X_{v}$ the subset of
elements $w$ in ${\goth g}$ such that $(v,w)$ is in $\ker \sigma'(x,y)$. By Lemmas
\ref{lsc4} and \ref{lsc1}, (ii), $v$ belongs to the orthogonal complement of
${\goth g}(x)$ in ${\goth g}$. In addition, $X_{v}$ is an  affine subspace whose tangent
space is equal to the orthogonal complement of ${\goth V}(x,y)$ in ${\goth g}$. By
Lemma \ref{l5sc1}, (i), ${\goth V}(x,y)$ has dimension $\bor g{}$ since $(x,y)$ belongs to
$\Omega _{{\goth g}}$. Consequently, $\ker \sigma '(x,y)$ has dimension at most
$\bor g{}-\rk \mathfrak{g}+2(\bor g{}-\rk\mathfrak{g})=3(\bor g{}-\rk\mathfrak{g})$, since $x$ is regular. Hence the image
of $\sigma '(x,y)$ has dimension at least $\bor g{}+\rk\mathfrak{g}$. So $\sigma '(x,y)$ is
surjective and $\sigma $ is smooth at $(x,y)$.

Conversely, suppose that $\sigma$ is smooth at $(x,y)$. For $(a,b)$ in ${\Bb C}^2$, we
denote by $\pi_{a,b}$ the linear map
$$(\poi z{i,0,d_{i}}{,\ldots,}{i,d_{i},0}{}{}{},\ i=1,\ldots,\rk\mathfrak{g}) \longmapsto
( \mbox{ } \sum\limits_{m+n=d_i} a^m b^{n} z_{i,m,n},\ i=1,\ldots,\rk\mathfrak{g}\mbox{ } )
\mbox{ ,}$$
from ${\Bb C}^{\bor g{}+\ran g{}}$ to ${\Bb C}^{\ran g{}}$. The linear map $\pi_{a,b}$ is
surjective as soon as $(a,b) \not=(0,0)$. Since $\sigma$ is smooth at
$(x,y)$, we deduce that the compound map $\sigma_{a,b} := \pi_{a,b}\rond \sigma$ is smooth
at $(x,y)$, for any $(a,b)$ in ${\Bb C}^2 \setminus \{(0,0)\}$. As
$$\sum\limits_{m+n=d_i} a^m b^{n} p_{i,m,n}(x',y')=p_i(ax'+by') \mbox{ ,}$$
for $i=1,\ldots,\rk\mathfrak{g}$ and $(x',y')$ in $\gtg g{}$, $\sigma_{a,b}$ maps
$(x',y')$ to \hbox{$(\poi {ax'+by'}{}{,\ldots,}{}{p}{1}{\ran g{}})$}. Moreover, it is the
compound map of the two following maps:
\begin{eqnarray*}
&\gtg g{} \longrightarrow {\goth g} \mbox{ , } (x',y') \mapsto (ax'+by')\\
&{\goth g} \longrightarrow {\Bb C}^{\ran g{}} \mbox{ , }
z\mapsto (\poi z{}{,\ldots,}{}{p}{1}{\ran g{}}) \mbox{ .}
\end{eqnarray*}
Therefore, the second map is smooth at $(ax+by)$, for any $(a,b)$
in ${\Bb C}^2 \setminus \{(0,0)\}$. So, by Lemma \ref{lsc1}, (ii), $(x,y)$
belongs to $\Omega_{{\goth g}}$.
\end{proof}

\begin{remark}\label{rsc4} Recall that ${\cali N}_{{\goth g}}$ is the subscheme of
$\gtg g{}$ defined by the ideal generated by the polynomials $p_{i,m,n}$. Therefore, by
Proposition \ref{psc4}, the intersection of ${\cali N}_{{\goth g}}$ and
$\Omega_{{\goth g}}$ is nothing but the set of smooth points of ${\cali N}_{{\goth g}}$.
The elements of $\Omega_{{\goth g}}$ provided by Lemma \ref{l3sc1}, (i), do not belong to
${\cali N}_{{\goth g}}$. Consider the elements described in Lemma \ref{l3sc1},
(ii). If ${\goth g}$ is equal to ${\goth {sl}}_3$, the subset
${\goth g}_{-2}' \times {\goth g}_{2}'$ has a nonempty intersection with
${\cali N}_{{\goth g}}$. Indeed, if we set
$$e:= \left[\begin{array}{ccc} 0 & 1 & 0\\
0 & 0 & 1 \\
0 & 0 & 0\end{array}\right]\mbox{, } \mbox{ }  h:= \left[\begin{array}{ccc} 2 & 0 & 0\\
0 & 0 & 0 \\
0 & 0 & -2\end{array}\right]\mbox{,  } \mbox{ }  x:= \left[\begin{array}{ccc} 0 & 0 & 0\\
1 & 0 & 0 \\
0 & -1 & 0\end{array}\right]\mbox{,  }$$
then $(x,e) \in {\cali N}_{{\goth g}} \cap ({\goth g}_{-2}' \times {\goth g}_{2}')$. But
in general this intersection can be empty. For
instance it is an empty set if ${\goth g}={\goth {sl}}_{n}$ for many $n>3$. Whatever
the  case, we will see a posteriori that there exists at least one irreducible component
of ${\cali N}_{{\goth g}}$ which has an empty intersection with $\Omega_{{\goth g}}$
(see Theorem \ref{tp1}). So we cannot hope to apply the criterion of Kostant
\cite{Kos2}. This observation was originally the reason why we introduced an other
subscheme (see the following section).
\end{remark}

\section{Principal cone and principal bicone}\label{su}
In this section, we suppose that ${\goth g}$ is simple. Then we can suppose that $p_1$ is
the Casimir element of $\ai g{}{}$, that is $p_1(x)=\dv xx$, for any $x$ in
${\goth g}$. Recall that $(e,h,f)$ is a principal ${\goth {sl}}_2$-triple of
${\goth g}$ and use the notations of the Introduction, \S\ref{notation}.

\subsection{Principal cone}\label{su1} Since $e$ is a regular nilpotent element of ${\goth g}$, the
nilpotent cone ${\goth N}_{{\goth g}}$ is the $G$-invariant closed cone generated by $e$.
According to Kostant's results \cite{Kos2}, the nilpotent cone is a complete intersection
of codimension $\rk\mathfrak{g}$. Moreover, it is proved in \cite{He} that it has rational
singularities. In this subsection, we intend to prove analogous properties for
{\it the principal cone} introduced in:

\begin{definition}\label{dsu1}
The {\it principal cone} ${\goth X}_{{\goth g}}$ of ${\goth g}$ is the
$G$-invariant closed cone generated by $h$.
\end{definition}

Recall that $w_{0}$ is the longest element of the Weyl group $W_{G}({\goth h})$. The
simple following well-known lemma turns out to be useful.

\begin{lemma}\label{lsu1} The element $w_{0}(h)$ is equal to $-h$. Moreover, there exists
a representative $g_{0}$ of $w_{0}$ in $N_{G}({\goth h})$ such that $g_{0}(e)$ is equal to
$f$.
\end{lemma}

For $i=2,\ldots,\rk\mathfrak{g}$, we define the element $q_{i}$ of $\ai g{}{}$ as follows:
$$q_i = \left\{
          \begin{array}{ll}
            p_i\mbox{ , }&\hbox{if } d_i \hbox{\textrm{ is odd }}; \\
            p_{1}(h)^{d_{i}/2}p_{i} - p_{i}(h)p_{1}^{d_{i}/2}\mbox{ , }&\hbox{otherwise.}
          \end{array}
        \right.$$

The polynomial $q_i$ is homogeneous of degree $d_i$. As the eigenvalues of
$\ad h$ are integers, not all equal to zero, $p_{1}(h) \not= 0$. In addition, as
$$p_i(h)=p_i(w_0(h))=p_i(-h)=(-1)^{d_i} p_i(h) \mbox{ ,}$$
$p_{i}(h)=0$ as soon as $d_i$ is odd. This forces $q_{i}(h)=0$, for any
$i=2,\ldots,\rk\mathfrak{g}$. So ${\goth X}_{{\goth g}}$ is contained in the
nullvariety of the functions $\poi q2{,\ldots,}{\ran g{}}{}{}{}$.

\begin{lemma}\label{l2su1}
{\rm i)} The nullvariety of $p_{1}$ in ${\goth X}_{{\goth g}}$ is equal to
${\goth N}_{{\goth g}}$.

{\rm ii)} The principal cone ${\goth X}_{{\goth g}}$ is the nullvariety of the functions
$\poi q2{,\ldots,}{\ran g{}}{}{}{}$.

{\rm iii)} The subset $({\goth X}_{{\goth g}} \setminus {\goth g}_{\reg})$ is equal to
$({\goth N}_{{\goth g}} \setminus G.e)$ and the codimension of
$({\goth X}_{{\goth g}} \setminus {\goth g}_{\reg})$ in ${\goth X}_{{\goth g}}$ is equal to
$3$.
\end{lemma}

\begin{proof} i) Prove first that ${\goth N}_{{\goth g}}$ is contained in
${\goth X}_{{\goth g}}$. Since ${\goth X}_{{\goth g}}$ is a $G$-invariant closed cone,
the relation $\exp(-t\ad e)(h)=h+2te$ for
any $t\in {\Bb C}$, implies $e \in {\goth X}_{{\goth g}}$. So ${\goth N}_{{\goth g}}$ is
contained in ${\goth X}_{{\goth g}}$ as closure of $G.e$. Let $X$ be the nullvariety of
the functions $\poi q2{,\ldots,}{\ran g{}}{}{}{}$. As $p_1(h)\not=0$, the
nullvariety of $p_{1}$ in $X$ is the nullvariety of $\poi p1{,\ldots,}{\ran g{}}{}{}{}$. So
${\goth N}_{{\goth g}}$ is the nullvariety of $p_{1}$ in $X$. Moreover,
${\goth N}_{{\goth g}}$ is also the nullvariety of $p_{1}$ in ${\goth X}_{{\goth g}}$
since ${\goth X}_{{\goth g}}$ is a subset of $X$ which contains ${\goth N}_{{\goth g}}$.

ii) Let $X$ be as in (i). We need to prove that $X$ is contained in
${\goth X}_{{\goth g}}$. Let $x$ be an element of $X$ which is not nilpotent. By (i),
$p_{1}(x)\not=0$. So there exists $t\in {\Bb C}^*$ such that $p_{1}(tx)=p_{1}(h)$. Then
$p_{i}(tx)=p_{i}(h)$, for $i=1,\ldots,\rk\mathfrak{g}$ since $x$ is in $X$. As $h$ is regular and
semisimple, $tx$ and $h$ are $G$-conjugate. Hence $x$ belongs to ${\goth X}_{{\goth g}}$.

iii) By (i) and (ii), the subset of elements of ${\goth X}_{{\goth g}}$ which are not
regular in ${\goth g}$ is equal to $({\goth N}_{{\goth g}}\backslash G.e)$. As this
subset has codimension $2$ in ${\goth N}_{{\goth g}}$, it has codimension $3$ in
${\goth X}_{{\goth g}}$ by (i).
\end{proof}

Recall that ${\goth b}_{-}$ is the Borel subalgebra containing ${\goth h}$ and
``opposite'' to ${\goth b}$, with nilpotent radical ${\goth u}_{-}$. Let ${\goth b}_{0}$
be the subspace of ${\goth b}$ generated by $h$ and ${\goth u}$ and let ${\goth b}_{0,-}$
be the subspace of ${\goth b}_{-}$ generated by $h$ and ${\goth u}_{-}$. The following
results are partially proved in \cite{Ri2}(Proposition 10.3) as pointed out in Remark
\ref{rsu1} below.

\begin{corollary}\label{csu1}
The principal cone of ${\goth g}$ is normal and it is a complete intersection of
codimension $\rk\mathfrak{g}-1$ in ${\goth g}$. The regular elements of ${\goth g}$
which belong to ${\goth X}_{{\goth g}}$ are smooth points of ${\goth X}_{{\goth g}}$. At
last, ${\goth X}_{{\goth g}}$ has rational singularities.
\end{corollary}

\begin{proof}
By Lemma \ref{l2su1}, (i) and (ii), ${\goth X}_{{\goth g}}$ is a complete intersection of
codimension $\ran g{}-1$. By Lemma \ref{lsc1}, (ii), the differentials at $x$ of
$\poi p1{,\ldots,}{\ran g{}}{}{}{}$ are linearly independent as soon as $x$ is a regular
element of ${\goth g}$. Hence the same goes for the differentials at $x$ of
$\poi q2{,\ldots,}{\ran g{}}{}{}{}$. Then, any $x$ in the union of $G. e$ and
$({\Bb C}^*)G. h$ is a smooth point of ${\goth X}_{{\goth g}}$. So by
Lemma \ref{l2su1}, (ii), ${\goth X}_{{\goth g}}$ is regular in codimension $1$. By
Serre's normality criterion \cite{Ma}(Ch. 8, Theorem 23.8), ${\goth X}_{{\goth g}}$ is a
normal subvariety.

The subalgebra ${\goth b}_{0}$ is an ideal of ${\goth b}$. The contracted product
$\sqg{0}$ is defined as the quotient of $G\times {\goth b}_{0}$ under the right action of
${\bf B}$ given by $(g,x).b=(gb,b^{-1}(x))$. The map $(g,x)\mapsto g(x)$ from
$G\times {\goth b}_{0}$ to ${\goth g}$ factorizes through the canonical map from
$G\times {\goth b}_{0}$ to $\sqg{0}$. Since $G.{\goth b}_{0}$ is equal to
${\goth X}_{{\goth g}}$, we get a surjective morphism $\pi $ from $\sqg{0}$ to
${\goth X}_{{\goth g}}$. Moreover, the morphism $\pi$ is proper since $G/{\bf B}$ is a
projective variety. Let $K$ be the field of rational
functions on $\sqg{0}$ and let $K_{0}$ be the field of rational functions on
${\goth X}_{{\goth g}}$. Let $g_{0}$ be as in Lemma \ref{lsu1}. By Lemma \ref{lsu1},
$g_{0}(h)$ is equal to $-h$. In particular, the fiber of $\pi $ at $h$ has at least two
elements. Let $(g,x)$ be an element of $G\times {\goth b}_{0}$ such that $g(x)$ is equal
to $h$. In particular, $x$ is a regular semisimple element of ${\goth g}$ which belongs
to ${\goth b}_0$ Hence there exists $b$ in ${\bf B}$ such that $b(x)$ belongs to
the line generated by $h$. As $\ad h$ and $\ad b(x)$ have the same eigenvalues, $b(x)$ is
equal to $h$ or $-h$, since the eigenvalues of $\ad h$ are integers, not all equal to
zero. If $b(x)=h$, then $(g,x)$ and $({\bf 1}_{{\goth g}},h)$ are equal
in $\sqg{0}$. Otherwise, $(g,x)$ and $(g_{0},-h)$ are equal in $\sqg{0}$.
Hence the fiber of $\pi $ at $h$ has two elements. So the same applies for the fiber of
$\pi $ at any element of the $G$-invariant cone generated by $h$ since $\pi $ is
$G$-equivariant. Hence $K$ is an algebraic extension of degree $2$ of $K_{0}$. In
particular, it is a Galois extension. Let $A$ be the integral closure of
${\Bb C}[{\goth X}_{{\goth g}}]$ in $K$ and let ${\goth X}'_{{\goth g}}$ be an affine
algebraic variety such that ${\Bb C}[{\goth X}'_{{\goth g}}]$ is equal to $A$. Then $A$
is stable under the Galois group of the extension $K$ of $K_{0}$ and
${\Bb C}[{\goth X}_{{\goth g}}]$ is the subalgebra of invariant elements of this action
since ${\goth X}_{{\goth g}}$ is normal. Hence by \cite{El2}(Lemma 1), it is enough to
prove that ${\goth X}'_{{\goth g}}$ has rational singularities.

The variety $\sqg{0}$ is smooth as a vector bundle over $G/{\bf B}$. Hence the morphism
$\pi $ factors through the canonical morphism from ${\goth X}'_{{\goth g}}$ to
${\goth X}_{{\goth g}}$. Let $\pi '$ be the morphism from $\sqg{0}$ to
${\goth X}'_{{\goth g}}$ such that $\pi $ is the compound map of $\pi '$ with
the canonical morphism from ${\goth X}'_{{\goth g}}$ to ${\goth X}_{{\goth g}}$. Then
$(\sqg{0},\pi ')$ is a desingularization of ${\goth X}'_{{\goth g}}$ since $\pi '$ is
proper and birational. By the corollary of \cite{He}(Theorem B), for
$i \in {\Bb N}^*$, the $i$-th cohomology group H$^{i}(\sqg{0},\an {\sqg{0}}{})$ is equal
to zero. Hence by \cite{Ha}(Ch. III, Proposition 8.5), for $i \in {\Bb N}^*$,
R$^{i}\pi '_{*}(\an {\sqg{0}}{})=0$ and ${\goth X}'_{{\goth g}}$ has
rational singularities since ${\goth X}'_{{\goth g}}$ is normal.
\end{proof}

\begin{remark}\label{rsu1} The $G$-invariant closed cone generated by a
semisimple element of ${\goth g}$ is not a normal variety in general
\cite{Ri2}(Proposition 10.1). Nevertheless, whenever $x$ is the semisimple element
of a ${\goth {sl}}_2$-triple of ${\goth g}$, the closed cone generated by the regular
orbit whose closure contains $x$ is normal and Cohen-Macaulay
\cite{Ri2}(Proposition 10.3).
\end{remark}

\subsection{Principal bicone} We study in this subsection various properties of the
{\it principal bicone}.

\begin{definition}\label{dsu2} The {\it principal bicone}  of
$\mathfrak{g}$ is the subset
\begin{eqnarray*}
{\cali X}_{{\goth g}} & := & \{(x,y) \in \gtg g{}  \ \vert \ P_{x,y}  \subset
{\goth X}_{{\goth g}}\}\\
& = & \{(x,y) \in \gtg g{}  \ \vert \ ax+by  \in {\goth X}_{{\goth g}}, \ \forall (a,b)
\in {\Bb C}^2\} \mbox{ . }
\end{eqnarray*}
\end{definition}

As ${\cali N}_{{\goth g}}$, the subset ${\cali X}_{{\goth g}}$ is a closed bicone of
$\gtg g{}$, and it is invariant under the actions of $G$ and GL$_{2}({\Bb C})$. For
$i=2,\ldots,\rk\mathfrak{g}$, we define the elements $q_{i,m,n}$ of $(\cg g{})^{{\goth g}}$ by the
following relation:
\begin{eqnarray}\label{qi}
q_i(ax+by)&=&\sum\limits_{m+n=d_i} a^m b^n q_{i,m,n}(x,y) \mbox{ , }
\end{eqnarray}
for $(a,b)$ in ${\Bb C}^2$ and $(x,y)$ in $\gtg g{}$. For $w$ in $W_G({\goth h})$, we
denote by ${\goth b}_{0,w}$ the subspace generated by $w(h)$ and ${\goth u}$. It is worth recalling the
following well-known result for the understanding of the principal
bicone.

\begin{lemma}\label{lsu2} If $x$ is in ${\goth h}$, then the intersection of $G.x$
and ${\goth h}$ is the orbit of $x$ under $W_G({\goth h})$.
\end{lemma}

Our purpose is to prove that ${\cali X}_{{\goth g}}$ is a complete intersection of
dimension \sloppy \hbox{$3(\bor g{}-\rk\mathfrak{g}+1)$}, what we will do in Section \ref{d}.

\begin{lemma}\label{l2su2} {\rm i)} The subset ${\cali X}_{{\goth g}}$ is the nullvariety
of the polynomial functions $q_{i,m,n}$, for $i=2,\ldots,\rk\mathfrak{g}$ and $m+n=d_i$. In
particular, any irreducible component of ${\cali X}_{{\goth g}}$ has dimension at least
$3(\bor g{}-\rk\mathfrak{g}+1)$.

{\rm ii)} The subset $\gtg s{}$ is contained in ${\cali X}_{{\goth g}}$.

{\rm iii)} For $x$ in ${\goth b}$, $(w(h),x)$ belongs to
${\cali X}_{{\goth g}}$ if and only if $x$ is in ${\goth b}_{0,w}$.
\end{lemma}

\begin{proof} i) As ${\goth X}_{{\goth g}}$ is the nullvariety of the $q_i$ by Lemma
\ref{l2su1}, (ii), ${\cali X}_{{\goth g}}$ is the nullvariety of the polynomial functions
$q_{i,m,n}$, for $i=2,\ldots,\rk\mathfrak{g}$ and $m+n=d_i$. As $d_1=2$,
${\cali X}_{{\goth g}}$ is the nullvariety in $\gtg g{}$ of $\bor g{}-\rk\mathfrak{g}-3$ regular
functions, whence the second statement.

ii) Let $T$ be the intersection of ${\cali X}_{{\goth g}}$ and $\gtg s{}$. Then $T$ is a
closed, subset of  $\gtg s{}$ invariant under GL$_{2}({\Bb C})$ and ${\bf S}$. For any
$t$ in ${\Bb C}$, $h+te$ belongs to $G. h$. So $T$ contains $(h,e)$. Hence, by Lemma
\ref{lsc2}, $\dim T \geq 5$ and, $T=T_{5}$ if and
only if $\dim T=5$. As $th$ and $e+t^{2}f$ are in the same $G$-orbit
for any $t \in {\Bb C}^{*}$, $(e,f)$ is in ${\cali X}_{{\goth g}}$. So $T$ is
equal to $\gtg s{}$ since $(e,f)$ is not in $T_{5}$.

iii) For any $t$ in ${\Bb C}^{*}$ and any $x$ in ${\goth u}$, $tw(h)+x$
belongs to the $G$-orbit of $tw(h)$. Hence $(w(h),tw(h)+x)$ belongs to
${\cali X}_{{\goth g}}$. So $\{w(h)\}\times {\goth b}_{0,w}$ is contained in
${\cali X}_{{\goth g}}$. Let $(x_{1},x_2)$ be in ${\goth h}\times {\goth u}$.
We suppose that $(w(h),x_{1}+x_{2})$ is in ${\cali X}_{{\goth g}}$. Then for any $t$ in
${\Bb C}$, $w(h)+tx_{1}+tx_{2}$ is in ${\goth X}_{{\goth g}}$. In particular,
$x_{1}+x_{2}$ is tangent at $w(h)$ to the cone ${\goth X}_{{\goth g}}$. The tangent
space at $w(h)$ of the cone ${\goth X}_{{\goth g}}$ is generated by $w(h)$ and
$[w(h),{\goth g}]$. So $x_{1}$ and $w(h)$ are collinear.
\end{proof}

\subsection{Smooths points of ${\cali X}_P$}\label{su3} As the nilpotent cone ${\goth N}_{{\goth g}}$ is contained in
the principal cone ${\goth X}_{{\goth g}}$, the nilpotent bicone ${\cali N}_{{\goth g}}$
is contained in the principal bicone ${\cali X}_{{\goth g}}$. We introduce now two other
varieties ``squeezed in between'' ${\cali N}_{{\goth g}}$ and ${\cali X}_{{\goth g}}$.
Set
\begin{eqnarray*}
{\cali Y}_{{\goth g}}&:=&\{\ (x,y) \in {\cali X}_{{\goth g}} \ \vert \ \dv xy =0  \ \}
\mbox{ , }\\
{\cali Z}_{{\goth g}}&:=&\{\  (x,y) \in {\cali Y}_{{\goth g}} \ \vert \ y \in
{\goth N}_{{\goth g}}\ \} \mbox{ . }
\end{eqnarray*}

\begin{lemma}\label{lsu3}
{\rm i)} The nilpotent bicone is the nullvariety in ${\cali X}_{{\goth g}}$ of
\sloppy \hbox{$p_{1,2,0},p_{1,1,1},p_{1,0,2}$}.

{\rm ii)} The subset ${\cali Y}_{{\goth g}}$ is the nullvariety in ${\cali X}_{{\goth g}}$
of $p_{1,1,1}$. Moreover, ${\cali Y}_{{\goth g}}$ is a $G$-invariant closed bicone and
any irreducible component of ${\cali Y}_{{\goth g}}$ has dimension at least
$3(\bor g{}-\rk\mathfrak{g})+2$.

{\rm iii)} The subset ${\cali Z}_{{\goth g}}$ is the nullvariety in ${\cali Y}_{{\goth g}}$
of $p_{1,0,2}$. Moreover, ${\cali Z}_{{\goth g}}$ is a $G$-invariant closed bicone and
any irreducible component of ${\cali Z}_{{\goth g}}$ has dimension at least
$3(\bor g{}-\rk\mathfrak{g})+1$.
\end{lemma}

\begin{proof} i) As ${\goth N}_{{\goth g}}$ is the nullvariety of $p_1$ in
${\goth X}_{{\goth g}}$, by Lemma \ref{l2su1} (i), ${\cali N}_{{\goth g}}$ is the
nullvariety in ${\cali X}_{{\goth g}}$ of the $2$-order polarizations of $p_1$, that is
$p_{1,2,0},p_{1,1,1},p_{1,0,2}$.

ii) As $p_{1,1,1}$ is the Killing form, ${\cali Y}_{{\goth g}}$ is the nullvariety of
$p_{1,1,1}$ in ${\cali X}_{{\goth g}}$. Moreover, ${\cali Y}_{{\goth g}}$ is a
$G$-invariant closed bicone since $p_{1,1,1}$ is $G$-invariant and bihomogeneous. By
Lemma \ref{l2su2}, (i), any irreducible component of ${\cali Y}_{{\goth g}}$ has
dimension at least $3(\bor g{}-\rk\mathfrak{g})+2$.

iii) By Lemma \ref{l2su1}, (i), ${\goth N}_{{\goth g}}$ is the nullvariety of $p_{1}$ in
${\goth X}_{{\goth g}}$. Hence ${\cali Z}_{{\goth g}}$ is the nullvariety of
$p_{1,0,2}$ in ${\cali Y}_{{\goth g}}$. Moreover, ${\cali Z}_{{\goth g}}$ is a
$G$-invariant closed bicone since $p_{1,0,2}$ is $G$-invariant and bihomogeneous. By
(ii), any irreducible component of ${\cali Z}_{{\goth g}}$ has dimension at least
$3(\bor g{}-\rk\mathfrak{g})+1$.
\end{proof}

\begin{remark}\label{rsu3}
By Lemma \ref{lsu3}, the subsets ${\cali Y}_{{\goth g}}$ and ${\cali Z}_{{\goth g}}$
inherit a natural structure of scheme.\\
\end{remark}

Let $P$ be a subset of the set of polynomials
$$\{p_{1,2,0},p_{1,1,1},p_{1,0,2}\mbox{ , } q_{i,m,d_{i}-m} \mbox{ , } i=2,\ldots,\rk\mathfrak{g}
\mbox{ , } m=0,\ldots,d_{i} \} \mbox{ .}$$
We denote by ${\Bb C}[P]$ the subalgebra of $\cg g{}$ generated by $P$, we denote by
$\sigma_{P}$ the morphism of affine varieties whose comorphism is the canonical injection
from ${\Bb C}[P]$ to $\cg{g}$ and we denote by ${\cali X}_{P}$ the nullvariety of $P$ in
$\gtg{g}$.

\begin{lemma}\label{l2su3}
Let $(x,y)$ be in $\Omega_{{\goth g}}$.

{\rm i)} The morphism $\sigma_{P}$ is smooth at $(x,y)$.

{\rm ii)} If $(x,y)$ is in ${\cali X}_{P}$, then $(x,y)$ is a smooth point of
${\cali X}_{P}$. Moreover, the unique irreducible component of ${\cali X}_{P}$ which
contains $(x,y)$ has dimension \sloppy \hbox{$2\,\dim {\goth g}-\vert P \vert$}.
\end{lemma}

\begin{proof}
i) Let $\rho $ be the morphism whose comorphism is the canonical injection from
${\Bb C}[P]$ to the subalgebra generated by the polynomials $p_{i,m,d_{i}-m}$ where
$i=1,\ldots,\rk\mathfrak{g}$ and $m=0,\ldots,d_{i}$. Then $\sigma_{P}$ is equal to
$\rho \rond \sigma $, where $\sigma $ is the morphism introduced in Subsection
\ref{sc4}. As $q_{i,m,d_{i}-m}$ is a linear combination of $p_{i,m,d_{i}-m}$ and
homogeneous elements in the subalgebra generated by $p_{1,2,0}$, $p_{1,1,1}$,
$p_{1,0,2}$, for $i=2,\ldots,\rk\mathfrak{g}$ and $m=0,\ldots,d_{i}$, the morphism $\rho $ is
smooth. Then by Proposition \ref{psc4}, $\sigma_{P}$ is a smooth morphism at $(x,y)$.

ii) We suppose that $(x,y)$ is in ${\cali X}_{P}$. By definition, ${\cali X}_{P}$ is the
fiber at $0$ of the morphism $\sigma_{P}$. So by (i), $(x,y)$ is a smooth point of
${\cali X}_{P}$ and the codimension in $\gtg{g}$ of the irreducible component of
${\cali X}_{P}$ which contains $(x,y)$ is equal to the dimension of ${\Bb C}[P]$. By
Proposition \ref{psc4}, the polynomials $p_{i,m,d_{i}-m}$ are algebraically independent
for \sloppy \hbox{$i=1,\ldots,\rk\mathfrak{g}$} and $m=0,\ldots,d_{i}$. So the elements of $P$ are
algebraically independent since $q_{i,m,d_{i}-m}$ is a linear combination of
$p_{i,m,d_{i}-m}$ and homogeneous elements in the subalgebra generated by
$p_{1,2,0}$, $p_{1,1,1}$, $p_{1,0,2}$, for \sloppy \hbox{$i=2,\ldots,\rk\mathfrak{g}$} and
$m=0,\ldots,d_{i}$. Hence the irreducible component of ${\cali X}_{P}$
containing $(x,y)$ has dimension $2\,\dim {\goth g}-\vert P \vert$.
\end{proof}

The principal bicone $ {\cali X}_{{\goth g}}$ has irreducible
components of the expected dimension as the following proposition shows. Actually this proposition will be not used in the
other sections.

\begin{proposition}\label{psu3} Let $w$ be in $W_{G}({\goth h})$. There is an unique
irreducible component $\chi(w)$ of ${\cali X}_{{\goth g}}$ containing $(w(h),e)$ and
$\chi(w)$ satisfies the following properties:
\begin{list}{}{}
\item {\rm 1)} $\chi (w)$ has dimension $3(\bor g{}-\rk\mathfrak{g}+1)$,
\item {\rm 2)} $\chi (ww_{0})$ is equal to $\chi (w)$,
\item {\rm 3)} $\chi (w)$ contains $(w_{0}w(h),f)$, $(e,w(h))$, $(f,w_{0}w(h))$.
\end{list}
Moreover, $\chi (w)$ contains $\gtg{s}$ if $w=1$ or $w=w_{0}$.
\end{proposition}

\begin{proof} Let $P$ be the subset $\{q_{i,m,d_{i}-m}, \ i=2,\ldots,\rk\mathfrak{g}, \
m=0,\ldots,d_{i}\}$. Then ${\cali X}_{P}$ is equal to ${\cali X}_{{\goth g}}$ by Lemma
\ref{l2su2}, (i). By Lemma \ref{l2su2}, (iii), ${\cali X}_{{\goth g}}$ contains
$(w(h),e)$ and by Lemma \ref{l3sc1}, (i), $(w(h),e)$ belongs to $\Omega_{{\goth g}}$.
Hence by Lemma \ref{l2su3}, (ii), $(w(h),e)$ is a smooth point of
${\cali X}_{{\goth g}}$. So there is an unique irreducible component of
${\cali X}_{{\goth g}}$ which contains $(w(h),e)$. Moreover, this component has dimension
$3(\bor g{}-\rk\mathfrak{g}+1)$. Let us denote it by $\chi (w)$. Then $\chi (w)$ satisfies condition
(1). As $ww_{0}(h)$ is equal to $-w(h)$, by Lemma \ref{lsu1}, $\chi (w)$ contains
$(ww_{0}(h),e)$ since $\chi (w)$ is a bicone. Moreover, by Lemma \ref{lsu1}, $\chi (w)$
contains $(w_{0}w(h),f)$ since $\chi (w)$ is $G$-invariant and
$g_{0}.(w(h),e)$ is equal to $(w_{0}w(h),f)$. As $\chi (w)$ is
also GL$_{2}({\Bb C})$-invariant, it is invariant under the involution \sloppy \hbox{
$(x,y)\mapsto (y,x)$}. So $\chi (w)$ contains $(e,w(h))$ and $(f,w_{0}w(h))$.

By Lemma \ref{l2su2}, (ii), $\gtg s{}$ is contained in ${\cali X}_{{\goth g}}$. Hence there
exists an irreducible component of ${\cali X}_{{\goth g}}$ which contains $\gtg s{}$. But
$\gtg s{}$ contains $(h,e)$, which is a smooth point of ${\cali X}_{{\goth g}}$. So
$\chi (w)$ contains $\gtg s{}$ if $w(h)$ is colinear to $h$, that is to say
$w=1$ or $w=w_{0}$.
\end{proof}

\begin{remark}
When ${\goth g}$ is simple of type $B_{\ell}$, $C_{\ell}$, $D_{2\ell }$, $E_{7}$, $E_{8}$,
$F_{4}$ or $G_{2}$, $w_{0}$ is equal to $-1$. So in these cases, for any $w$ in
$W_{G}({\goth h})$, $\chi (w)$ contains $(w(h),f)$ and $(f,w(h))$ since $\chi (w)$ is a
bicone. A quick computation shows that ${\cali X}_{{\goth g}}$ precisely has
$\vert W_{G}({\goth h}) \vert/2$ irreducible components when ${\goth g}$ is equal to
${\goth s}{\goth l}_{n}({\Bb C})$ for $n=2,3$.
\end{remark}

\subsection{Proof of Corollary \ref{csu4}} The goal of this subsection is to prove Corollary \ref{csu4}. This
corollary will be crucial for the study of the dimension of ${\cali X}_{{\goth g}}$
(see Section \ref{d}).

\begin{lemma}\label{lsu4}
Let ${\cali Z}$ be an irreducible closed subset of ${\cali Z}_{{\goth g}}$ satisfying the
two following conditions:
\begin{itemize}
\item[{\rm 1)}] $\varpi_{1}({\cali Z})$ contains a non zero semisimple element of
${\goth g}$,
\item[{\rm 2)}] $\varpi_{2}({\cali Z})$ contains a regular nilpotent element of
${\goth g}$.
\end{itemize}
Then ${\cali Z}$ has a nonempty intersection with $\Omega_{{\goth g}}$.
\end{lemma}

\begin{proof} By Condition (1), the subset ${\cali Z}_{1}$ of elements $(x,y)$ in
${\cali Z}$ such that $x$ is a non zero semisimple element of $\mathfrak{g}$ is a
nonempty open subset of ${\cali Z}$ since the subset of semisimple elements of
${\goth g}$ which belongs to ${\goth X}_{{\goth g}}$ is an open dense subset of
${\goth X}_{{\goth g}}$. On the other hand, by Condition (2), the subset ${\cali Z}_{2}$
of elements $(x,y)$ in ${\cali Z}$ such that $y$ is a regular nilpotent element of
${\goth g}$ is a nonempty open subset of ${\cali Z}$ since the subset of regular
nilpotent elements is an open dense subset of ${\goth N}_{{\goth g}}$. Then the
intersection of ${\cali Z}_{1}$ and ${\cali Z}_{2}$ is nonempty subset since ${\cali Z}$ is
irreducible. Let $(x,y)$ be in the
intersection  ${\cali Z}_{1}$ and ${\cali Z}_{2}$. Suppose that $(x,y)$ doesn't belong to $\Omega_{{\goth g}}$. We expect a
contradiction. Then there exists $t\in {\Bb C}$ such that $tx+y$ is not a regular
element of ${\goth g}$. Indeed $x$ is regular as non zero semisimple element of
${\goth X}_{{\goth g}}$. As $(x,y)$ belongs to the principal bicone, $tx +y$ is a non
regular element of ${\goth X}_{{\goth g}}$, whence
$$\dv {tx+y}{tx+y}=0 \mbox{ ,}$$
by Lemma \ref{l2su1}, (iii). But the left hand side of this equality is
equal to $t^{2}\dv xx$ since $(x,y)$ belongs to ${\cali Z}_{{\goth g}}$. Hence $t=0$ since
for any non zero semisimple element $z$ in ${\goth X}_{{\goth g}}$,
$\dv zz\not=0$. This contradicts the fact that $y$ is regular. So ${\cali Z}$ has a
nonempty intersection with $\Omega_{{\goth g}}$.
\end{proof}

\begin{corollary}\label{csu4} Let ${\cali X}$ be an irreducible component of
${\cali X}_{{\goth g}}$. We denote by ${\cali X}'$ the intersection of ${\cali X}$ and
${\cali Z}_{\goth g}$. Let us suppose that ${\cali X}'$ has an irreducible component
${\cali Z}$ satisfying the two following conditions:
\begin{itemize}
\item[{\rm 1)}] $\varpi_{1}({\cali Z})$ is equal to ${\goth X}_{{\goth g}}$,
\item[{\rm 2)}] $\varpi_{2}({\cali Z})$ is equal to ${\goth N}_{{\goth g}}$.
\end{itemize}
Then ${\cali X}'$ and ${\cali X}$ have a non empty intersection with $\Omega_{{\goth g}}$.
In particular, ${\cali X}$ has dimension \sloppy \hbox{$3(\bor g{}-\rk\mathfrak{g}+1)$}.
\end{corollary}

\begin{proof}
By Lemma \ref{lsu4}, ${\cali Z}$, ${\cali X}'$ and ${\cali X}$ have a nonempty intersection
with $\Omega _{{\goth g}}$. Then, by Lemma \ref{l2su3}, (ii), ${\cali X}$ has
dimension \sloppy \hbox{$3(\bor g{}-\rk\mathfrak{g}+1)$}.
\end{proof}

\section{Jet schemes and motivic integration}\label{jm}

We plan to study in Section \ref{d} the dimensions of the nilpotent bicone and the
principal bicone via motivic integration arguments. In view of this work, we start this
section by some basics on motivic integration.

\subsection{Jet schemes}\label{jm1} Let us first review the definition and properties
of jet schemes. Let $X$ be a complex algebraic variety. For $m \in {\Bb N}$,  the
{\it $m$-order jet scheme} $J_m(X)$ of $X$ is the scheme whose closed points over
$x \in X$ are morphisms
$$\an {X}{x} \longrightarrow {\Bb C}[t]/t^{m+1} \mbox{ . }$$
Thus, the ${\Bb C}$-valued points of $J_m(X)$ are in natural bijection with the
${\Bb C}[t]/t^{m+1}$-valued points of $X$. In particular, there are canonical
isomorphisms $J_0(X) \simeq X$ and $J_1(X) \simeq \mathrm{T}X$, where $\mathrm{T}X$ is
the total tangent space of $X$. The canonical morphisms
$$\pi _{l,m} \ : \ J_l(X) \longrightarrow J_{m}(X) \mbox{ , } $$
with $l \geq m$, obtained by truncation, induce a projective system
$(J_{m}(X),\pi _{l,m})$. The {\it space of the arcs}
$$J_{\infty }(X):=\mathrm{proj} \lim\limits_{m} J_m(X)$$
is the projective limit of the system $(J_{m}(X),\pi _{l,m})$. Denote by
$\pi _{\infty,m}$ the canonical morphism
$$\pi _{\infty,m} \ : \ J_{\infty}(X) \longrightarrow J_{m}(X) \mbox{ , }$$
for $m\in \mathbb{N}$, obtained by truncation, too. For $\nu $ in $J_{\infty }(X)$ and $\varphi $ a regular function on $X$, we denote by
$\ord {\nu \rond \varphi }$ the order of the serie $\nu \rond \varphi$. If
${\cali I}$ is an ideal of $\an X{}$, we denote by $\ord {{\cali I},\nu }$ the smallest
integer $\ord {\nu \rond \varphi }$, where $\varphi $ runs through
the ideal ${\cali I}_{\pi_{\infty,0}(\nu)}$ generated by ${\cali I}$ in the local ring
$\an X{\pi_{\infty,0}(\nu)}$. The function
$F_{{\cali I}}: \nu \mapsto \ord {{\cali I},\nu }$ is semialgebraic by
\cite{DL2}(Theorem 2.1). As a matter of fact, for $m \in {\Bb N}$, we can define an
analogous function from $J_{m}(X)$ to $\{1,\ldots,m\}$ that we denote by the same symbol.
When ${\cali I}$ is the ideal of definition of a closed subset $Z$ of $X$,
we rather denote by $F_{Z}$ the function $F_{{\cali I}}$. In particular, if ${\cali I}$ is
the ideal of definition of a divisor $D$ of $X$, we denote by $F_{D}$ the function
$F_{{\cali I}}$.\\

In \cite{Mu}, M. Musta\c{t}\v{a} proves the following result, first conjectured by
David Eisenbud and Edward Frenkel:

\begin{theorem}[Musta\c{t}\v{a}]\label{tjm1} If $X$ is locally a complete intersection
variety, then $J_m(X)$ is irreducible of dimension $\dim X(m+1)$ for all $m \geq 1$ if
and only if $X$ has rational singularities.
\end{theorem}

\begin{remark}\label{rjm1} The nilpotent cone ${\goth N}_{{\goth g}}$ of the reductive Lie
algebra ${\goth g}$ is a complete intersection and by \cite{He}(Theorem A), it has
rational singularities. Likewise, by Corollary \ref{csu1}, the principal cone
${\goth X}_{{\goth g}}$ of ${\goth g}$ is a complete intersection and has
rational singularities. As a consequence, Theorem \ref{tjm1} can be applied to
${\goth N}_{{\goth g}}$ or ${\goth X}_{{\goth g}}$. In the appendix of \cite{Mu},
D. Eisenbud and E. Frenkel apply Theorem \ref{tjm1} to ${\goth N}_{{\goth g}}$ to extend
Kostant's results in the setting of jet schemes.
\end{remark}

\subsection{Motivic integration}\label{jm2} Recall now some facts about the theory of motivic
integration. The construction of motivic integrals for smooth
spaces is due to Kontsevich \cite{Kon} and was generalized by Denef and Loeser to
singular spaces in \cite{DL2} and \cite{DL1}. Further in the paper, we will only need of
the Hodge realization of motivic integrals on the arcs of a smooth variety. We refer to
\cite{Ba}, \cite{DL2} and \cite{Lo} for more explanations, definitions and proofs
(see also \cite{Cr} for an introduction).\\

Let ${\cali V}$ be the category of complex algebraic varieties. Denote by
K$_{0}({\cali V})$ the Grothendieck ring of ${\cali V}$ and denote by $[X]$ the class in
K$_{0}({\cali V})$ of an element $X$ in ${\cali V}$.  The map $X\mapsto [X]$ naturally
extends to the category of constructible subsets of algebraic varieties. Let ${\Bb L}$ be the class of
${\Bb A}^{1}$ in K$_{0}({\cali V})$ and let ${\bf M}$ be
the localization K$_{0}({\cali V})[{\Bb L}^{-1}]$. For $m$ in ${\Bb Z}$, we denote by $F^{m}{\bf M}$ the subgroup of
${\bf M}$ generated by the elements $[X]{\Bb L}^{-r}$, where $r-\dim X \geq m$.
Then we get a decreasing filtration of the ring ${\bf M}$ and we denote by
$\widehat{{\bf M}}$ its separated completion.\\

Let $X$ be an algebraic variety of pure dimension $d$. A subset $A$ in $J_{\infty}(X)$ is
called {\it a cylinder} if $A$ is a finite union of fibers of
$\pi_{\infty,m} : J_{\infty}(X) \longrightarrow J_{m}(X)$, for $m\in \mathbb{N}$.  If $A \in J_{\infty}(X)$ is a
cylinder, we say that $A$ is {\it stable at level
$m \in {\Bb N}$} if, for any
$n \geq m$, the map
$\pi_{\infty,n+1}(J_{\infty}(X)) \longrightarrow \pi_{\infty,n}(J_{\infty}(X))$ is a
piecewise trivial fibration over $\pi_{\infty,n}(A)$ with fiber ${\Bb A}^{d}$. When $X$
is smooth, any cylinder is stable and the later additional condition is superfluous.

\begin{proposition}[\cite{DL2}(Definition-Proposition 3.2)]\label{pjm2}
There is a well-defined subalgebra ${\bf B}_{X}$ of the Boolean algebra of subsets of
$J_{\infty }(X)$ which contains the cylinders and a map $\mu _{X}$ from ${\bf B}_{X}$ to
$\widehat{{\bf M}}$ satisfying the following properties:

{\rm 1)} if $A \in {\bf B}_{X}$ is stable at level $m$, then
$\mu _{X}(A)=[\pi_{\infty,m}(A)]{\Bb L}^{-(m+1)d}$,

{\rm 2)} if $Y$ is a closed subvariety in $X$ of dimension strictly smaller than
$d$, then for any $A$ in ${\bf B}_{X}$, contained in $J_{\infty }(Y)$,
$\mu _{X}(A)=0$,

{\rm 3)} let $\poi A1{,}{2}{}{}{},\ldots$ be a sequence of elements in
${\bf B}_{X}$ whose union $A$ is in ${\bf B}_{X}$, then the sequence
$\poi A1{,}{2}{\mu }{X}{X},\ldots$ converges to $\mu _{X}(A)$ in $\widehat{{\bf M}}$,

{\rm 4)} if $A$ and $B$ are in ${\bf B}_{X}$, $A$ is contained in $B$ and
$\mu _{X}(B)$ belongs to the closure of $F^{m}{\bf M}$ in $\widehat{{\bf M}}$, then
$\mu _{X}(A)$ belongs to the closure of $F^{m}{\bf M}$ in $\widehat{{\bf M}}$.
\end{proposition}

For $A$ in  $\mathbf{B}_X$ and $\psi : A \longrightarrow {\Bb Z} \cup \{\infty\}$ a
function such that $\psi^{-1}(s) \in \mathbf{B}_X$ for any
$s \in {\Bb Z} \cup \{\infty\} $ and $\mu_X(\psi^{-1}(\infty))=0$, we can set
\begin{eqnarray}\label{eint}
\int\limits_{A}  {\Bb L}^{-\psi}\, \textrm{d}\mu_X := \sum\limits_{ s \in {\Bb Z}}
\mu_X(\psi^{-1}(s)) {\Bb L}^{-s} \mbox{ }
\end{eqnarray}
in $\widehat{{\bf M}}$, whenever the right hand side converges in $\widehat{{\bf M}}$. In
this case, we say that {\it ${\Bb L}^{-\psi}$ is integrable on $A$}.\\

Recall that the dimension of a non irreducible variety is the maximal dimension
of its irreducible components. Let $u$ and $v$ be two indeterminate variables. For any
smooth projective variety $X$, the {\it Hodge-Deligne polynomial} of $X$ is the element of
${\Bb Z}[u,v]$
\begin{eqnarray*}
h(X) := \sum_{k\in {\Bb N}}\sum_{(p,q)\in {\Bb N}^{2}\atop{p+q=k}}
(-1)^{k} h_{p,q}({\rm H}^{k}_{c}(X;{\Bb C})) u^{p}v^{q} \mbox{ , }
\end{eqnarray*}
where $h_{p,q}({\rm H}^{k}_{c}(X;{\Bb C}))$ are the Hodge-Deligne numbers of $X$. The
map $h$ factors through the ring K$_{0}({\cali V})$. So we have a morphism $h$ from
K$_{0}({\cali V})$ to ${\Bb Z}[u,v]$ such that $h(X)$ is equal to
$h([X])$ for any smooth projective variety. In particular, $h({\Bb L})$ is equal to
$uv$. What is important for us is that $h([X])$ is a polynomial whose highest degree term
is $c(uv)^{\dim X}$, where $c$ is the number of irreducible components of $X$ of
dimension $\dim X$. By continuity, the morphism $h$ uniquely  extends to a morphism from
$\widehat{{\bf M}}$ to the ring ${\Bb Z}[u,v][[u^{-1},v^{-1}]]$. The compound map
$h \circ \mu_X$, which is now well-defined, is {\it the Hodge realization of the motivic
measure}.

\subsection{Some technical results} \label{jm3}
Let $V$ be a finite dimensional vector space. We study in this
subsection various properties of specific subsets in $J_{\infty}(V)$.

For $(x,y)$ in $V\times V$, we denote by
$\nu _{x,y}$ the arc $t\mapsto x+ty$ of $J_{\infty}(V)$. Let $m$ be in ${\Bb N}^*$.  We
denote by $\nu _{x,y,m}$ the image of $\nu_{x,y}$ by the canonical projection
$\pi_{\infty,m}^{V}$ from $J_{\infty}(V)$ to $J_{m}(V)$. Let ${\bf K}$ be a connected
closed subgroup of ${\rm GL}(V)$. The
${\bf K}$-action on $V$ extends to a ${\bf K}$-action on $J_m(V)$ which is compatible
with the canonical projections $\pi_{\infty,m}^{V}$. Let $X$ be a ${\bf K}$-invariant
irreducible closed cone in $V$. We suppose that $X$ is a complete intersection in $V$
with rational singularities and we suppose that $X$ is a finite union of
${\bf K}$-orbits. Let $N$ be the dimension of $X$ and let $r$ be the codimension of $X$
in $V$.

We use now techniques developed in \cite{Mu}(Theorem 3.2) in order to prove Proposition
\ref{pjm3}. We denote by $Z$ the union of ${\bf K}$-orbits in $X$ which are not of
maximal dimension. Let $B_{X}$ be the blowing up of $V$ whose center is $X$ and let
$$\tau \ : \ B_{X} \longrightarrow V$$
be the morphism of the blowing-up. As $X$ is a complete intersection in $V$,
$\tau ^{-1}(X)$ is an integral divisor on $B_{X}$ and is locally a complete intersection.
Moreover, there exists a regular action of ${\bf K}$ on $B_{X}$ for which $\tau $ is a
${\bf K}$-equivariant morphism. By the theorem of embedded desingularization of Hironaka
\cite{Hi}, there exists a desingularization $(Y,\tilde{\tau })$ of $B_{X}$ such that
$(\tau \rond \tilde{\tau })^{-1}(X)$ is a divisor with normal crossings. Moreover, we can
find $(Y,\tilde{\tau })$ such that there exists a regular action of ${\bf K}$ on $Y$ for
which $\tilde{\tau }$ is an equivariant morphism, and
$(\tau \rond \tilde{\tau })^{-1}(X)$ is a ${\bf K}$-invariant divisor. Set
$\gamma :=\tau \rond \tilde{\tau }$. Denoting by
$\poi E1{,\ldots,}{t}{}{}{}$ the irreducible
components of $\gamma ^{-1}(X)$, we can assume that the following conditions are
fulfilled:
\begin{list}{}{}
\item a) $E_{1}$ is the only prime divisor dominating $X$,
\item b) the divisor $\gamma ^{-1}(X)$ is equal to $\sum_{i=1}^{t} a_{i}E_{i}$,
\item c) the discrepancy $W$ of $\gamma $ is equal to $\sum_{i=1}^{t} b_{i}E_{i}$,
\item d) $a_{1}$ is equal to $1$ and $b_{1}+1$ is equal to $r$,
\item e) $\gamma ^{-1}(Z)$ is contained in the union of $\poi E2{,\ldots,}{t}{}{}{}$,
\end{list}
since $\tau ^{-1}(X)$ is an integral divisor on $B_{X}$ and ${\bf K}$ has finitely many
orbits in $X$. By Condition (b), $\poi E1{,\ldots,}{t}{}{}{}$ are ${\bf K}$-invariant.
Moreover,  by Condition (e), $Z$ is the image by $\gamma $ of the union of
$\poi E2{,\ldots,}{t}{}{}{}$. So there exist $\poi c2{,\ldots,}{t}{}{}{} \in {\Bb N}$,
such that $\gamma ^{-1}(Z)=\sum_{i=2}^{t}c_{i}E_{i}$. For
$m\in {\Bb N}\cup \{\infty \}$, let
$$\gamma _{m} \ : \ J_{m}(Y) \longrightarrow J_{m}(V) $$
be the morphism induced by $\gamma $. In the remainder of
this subsection, for $m\in {\Bb N}\cup \{\infty \}$ and
$n\in {\Bb N}$ with $m \geq n$, $\pi_{m,n}$ refers to the canonical morphism
$J_m(Y) \longrightarrow J_n(Y)$.\\

As $X$ has rational singularities, the canonical injection from its canonical module to
its dualizing module is an isomorphism by \cite{Fl}(Satz 1.1). But its canonical module
is locally free of rank $1$ since $X$ is Gorenstein as a complete intersection in
$V$ \cite{Br}(Theorem 3.3.7 and Proposition 3.1.20). So $X$ has canonical
singularities. Moreover, by \cite{Mu}(Theorem 2.1), $b_{i} \geq ra_{i}$ for
$i=1,\ldots,t$. Then by the above Condition (d), we have $b_i \geq (b_{1}+1)a_i$, for
$i=2,\ldots,t$. For $k \in {\Bb N}$, for $J$ a  nonempty subset of $\{1,\ldots,t\}$ and
for $J'$ a nonempty subset of $J$, we denote by $\psi _{k,J,J'}$ the affine functional on
${\Bb Q}^{\vert J \vert}$
\begin{eqnarray*}
(\alpha _{i},i\in J)&\longmapsto &(N-\vert J' \vert)(m+1)+
\sum_{i\in J} (b_{i}+1)\alpha _{i} - \sum_{i\in J\backslash J'} \alpha _{i}
+ k\sum_{i \in J} c_{i}\alpha _{i} \mbox{ . }
\end{eqnarray*}
For $p \in {\Bb N}^*$, we set
\begin{eqnarray*}
A_{J,J',p} & \ := \ \{(\alpha _{i},i\in J)\in {\Bb N}^{\vert J \vert} \ \vert \ &
\sum_{i\in J}a_{i}\alpha _{i} \geq p \mbox{ and } \alpha _{i}\geq m+1 \mbox{ , } i \in J'
\\ &  & \mbox{ and } 1\leq \alpha _{i} \leq m \mbox{ , }i \in J\backslash J'\} \mbox{  }
\end{eqnarray*}
and we denote by $\nu _{k,J,J',p}$ the minimal value of $\psi _{k,J,J'}$ on
$A_{J,J',p}$. The two assertions of Lemma \ref{ljm3} are straightforward from the
 definitions and the preceding inequalities: $b_i \geq (b_{1}+1)a_i$, for $i=2,\ldots,t$.
Their verifications are left to the reader.

\begin{lemma}\label{ljm3} For $k\in {\Bb N}^*$, $p \geq m+1$,
$J \subseteq \{1,\ldots,t\}$ and $J'$ a nonempty subset of $J$, we have:

{\rm i)} $\nu _{0,J,J',p}=p(b_{1}+1)+(N-1)(m+1)=\nu _{k,J,J',p}$.

{\rm ii)} $\nu _{k,\{1\},\{1\},p} \leq \nu _{k,J,J',p}$.
\end{lemma}

For $J \subseteq \{1,\ldots,t\}$, we set:
$$E_{J}^{0} := \cap _{i\in J} E_{i} \backslash \cup _{i\not \in J} E_{i} \mbox{ . }$$
Let $m$ be in ${\Bb N}^*$ and let $S$ be a nonempty constructible subset of $J_m(V)$.
For $\alpha=(\alpha _{i})_{i\in J} \in ({\Bb N}^*)^{\vert J \vert}$, we denote by
$J_{\alpha }$ the subset of elements $i \in J$ such that $\alpha _{i} \geq m+1$ and we
set:
$$T_{\alpha } \ := \
\{\nu \in \pi_{m,0}^{-1}(E_{J}^{0})\cap \gamma _{m}^{-1}(S)  \
\vert \ F_{E_{i}}(\nu ) = \alpha _{i} \mbox{ , } i \in J\backslash J_{\alpha }\}
\mbox{ . }$$
For $q \in {\Bb N} \cup \{\infty\}$, with $q \geq m+1$, we set:
\begin{eqnarray*}
S^{q}_{\alpha} & := & \{ \nu \in \pi _{q ,0}^{-1}(E_{J}^{0})\cap
(\gamma _{m}\rond \pi_{q,m})^{-1}(S)
\nonumber \ \vert \ F_{E_{i}}(\nu ) = \alpha _{i} \mbox{ , } i\in J\} \mbox{ . }
\end{eqnarray*}

\begin{lemma}\label{l2jm3}  Fix $J \subseteq \{1,\ldots,t\}$, $q \geq m+1$ and
$\alpha=(\alpha _{i})_{i\in J} \in ({\Bb N}^*)^{\vert J \vert}$. Suppose that
$S^{q}_{\alpha}$ is a nonempty set. Then $S^{q}_{\alpha}$ is a locally trivial fibration
over $T_{\alpha }$ whose fiber is isomorphic to
$${\Bb C}^{(q-m)(N-\vert J_{\alpha} \vert)}\times
({\Bb C}^{*})^{\vert J_{\alpha} \vert}
\times {\Bb C}^{q\vert J_{\alpha} \vert-\sum_{i\in J_{\alpha}} \alpha _{i}} \mbox{ . }$$
\end{lemma}

\begin{proof} We follow the proof of \cite{Cr}(Proposition 2.5). Since the divisor
$\gamma^{-1}(X)=\sum\limits_{i=1}^{t} a_i E_i$ on $Y$ has only simple normal crossing,
for any $y$ in $Y$, there exists a neighborhood $U$ of $y$ in $Y$ with global coordinates
$\poi z1{,\ldots,}{N}{}{}{}$ on $U$ for which a local defining equation for
$\gamma^{-1}(X)$ is given by
\begin{eqnarray}\label{chart}
g=z_{1}^{a_1} \cdots z_{j_y}^{a_{j_y}}\mbox{ , }
\end{eqnarray}
for some $j_{y} \leq N$. We cover $Y=\bigcup U$ by finitely many charts on which
$\gamma^{-1}(X)$ has a local equation of the form (\ref{chart}), and we lift to cover
$J_{q}(Y)=\bigcup \pi_{q,0}^{-1}(U)$. Hence $S^{q}_{\alpha}$ is covered by
the subsets
\begin{eqnarray*}
U_{\alpha}^{q} &:=& \bigcap\limits_{i \in J} F_{E_i}^{-1}(\alpha_i) \cap
\pi_{q,0}^{-1}(U) \cap  \pi_{q,m}^{-1}( \gamma _{m}^{-1}(S)) \mbox{ . }
\end{eqnarray*}
As $\alpha_i \geq 1$, the subset $U_{\alpha}^{q}$ is contained in
$\pi_{q,0}^{-1}(E_{J}^{0})$. Thus, the subsets $U_{\alpha}^{q}$ are open subsets of
$S^{q}_{\alpha}$. Let $U_{\alpha}^{q}$ be such an open subset of
$S_{\alpha}^{q}$ which is nonempty.  If $J$ is not contained in
$\{1,\ldots,j_y\}$ then $E_{J}^{0} \cap U$ is empty, and so $U_{\alpha}^{q}$. So
$J$ is contained in $\{1,\ldots,j_y\}$.

For $\nu \in \pi_{q,0}^{-1}(U)$, we can view $\nu$ as an $N$-tuple
$(f_1(z),\ldots,f_N(z))$ of polynomials of degree at most $q$ with zero constant
term. We continue to argue as in the proof of \cite{Cr}(Proposition 2.5). We see that
$\nu \in F_{E_{i}}^{-1}(\alpha_i)$ if and only if the truncation of $f_i(z)$ to degree
$\alpha _{i}$ is of the form $c_{\alpha _{i}}z^{\alpha _{i}}$ where $c_{\alpha  _{i}}$
is different from $0$. Then we obtain $N-\vert J \vert$ polynomials of degree $q$ with
zero constant term, and, for each $j \in J$, a polynomial of the form
$$f_j(z)=0+\cdots+0 + c_{\alpha _{j}}z^{\alpha _{j}} +
c_{\alpha _{j+1}}z^{\alpha _{j+1}} +  \cdots + c_{\alpha _{q}}z^{\alpha _{q}} \mbox{ , }$$
with $c_{\alpha _{j}} \in {\Bb C}^*$ and $c_{k} \in {\Bb C}$, for $k > j$. So, when we
cut the $m$ first terms of each polynomial $f_{j}(z)$, the space of all such $N$-tuples
so obtained is isomorphic to
$$ {\Bb C}^{(q-m)(N-\vert J_{\alpha} \vert)}\times
({\Bb C}^{*})^{\vert J_{\alpha} \vert}
\times {\Bb C}^{q\vert J_{\alpha} \vert-\sum_{i\in J_{\alpha}} \alpha _{i}} \mbox{ . }$$
As a consequence,  $U_{\alpha}^{q}$ is isomorphic to
\begin{eqnarray*}
(\pi _{m,0}^{-1}(U)\cap T_{\alpha})  \times \ {\Bb C}^{(q-m)(N-\vert J_{\alpha} \vert)}
\ \times \ ({\Bb C}^{*})^{\vert J_{\alpha} \vert}
\times {\Bb C}^{q\vert J_{\alpha} \vert-\sum_{i\in J_{\alpha}} \alpha _{i}} \mbox{ , }
\end{eqnarray*}
whence the lemma.
\end{proof}

Let $S$ be a constructible subset of $J_{m}(V)$. For $p \geq m+1$, we
set
\begin{eqnarray*}
S_{\infty,m,p}&:=&\{ \ \nu \in (\pi_{\infty,m}^{V})^{-1}(S) \ \vert \ F_{X}(\nu) \geq p
\ \} \mbox{ . }\end{eqnarray*}
Then, for $k >0$, we set
\begin{eqnarray*}
{\bf I} _{m,p,k} &:= &h\left( \ \int_{S_{\infty ,m,p}} {\Bb L}^{-kF_{Z}(\nu )}
\dd \mu _{V}(\nu ) \ \right)
\mbox{ . }\end{eqnarray*}

\begin{proposition}\label{pjm3} Let $m,k$ be in ${\Bb N}^*$ and let $S$ be a constructible
subset of $J_{m}(V)$. We suppose that the image by $\pi_{m,0}^{V}$ of $S$ is dense in $X$
and we suppose that $S$ is ${\bf K}$-invariant.

{\rm i)} We suppose that the image by $\pi _{m,0}^{V}$ of any irreducible component of
maximal dimension of $S$ is dense in $X$. Then the highest degree terms of
${\bf I} _{m,p,k}$ does not depend on $k$ for $p$ big enough.

{\rm ii)} We suppose that the highest
degree terms of ${\bf I} _{m,p,k}$ does not depend on $k$ for $p$ big enough. Then, for
$p$ big enough, the highest degree terms of \sloppy
\hbox{$h\rond \mu_V(S _{\infty,m,p})$} only depend on the irreducible components of $S$
whose image by $\pi_{m,0}^{V}$ is dense in $X$.
\end{proposition}

\begin{proof} We apply the transformation rule for the motivic integrals to
${\bf I} _{m,p,k}$. Namely, by \cite{DL2}(Lemma 3.3), we have
\begin{eqnarray}\label{kon}
{\bf I} _{m,p,k} = h\left( \ \int_{\gamma _{\infty }^{-1}(S_{\infty ,m,p})}
{\Bb L}^{-F_{W}(\nu )-k\sum_{i=2}^{t}c_{i}F_{E_{i}}(\nu )} \dd \mu _{Y}(\nu ) \ \right)
 \mbox{ . }
\end{eqnarray}

For $J \subseteq \{1,\ldots,t\}$ and
$\alpha=(\alpha _{i})_{i \in J} \in ({\Bb N}^*)^{\vert J \vert}$, $S^{\infty}_{\alpha}$
is the inverse image by $\pi _{\infty ,q}$ of $S_{\alpha}^{q}$, for any $q \geq m+1$. So,
by Proposition \ref{pjm2}, (1), \sloppy \hbox{$h\rond \mu _{Y}(S_{\alpha }^{\infty}) =
h([S^{q}_{\alpha}]) (uv)^{-N(q+1)}$}. In addition, by Lemma \ref{l2jm3}, the subset
$S^{q}_{\alpha}$ is a locally trivial fibration over $T_{\alpha }$ whose fiber is
isomorphic to
$${\Bb C}^{(q-m)(N-\vert J_{\alpha} \vert)}\ \times \
({\Bb C}^{*})^{\vert J_{\alpha} \vert}
\times {\Bb C}^{q\vert J_{\alpha} \vert-\sum_{i\in J_{\alpha}} \alpha _{i}} \mbox{ . }$$
Then we deduce from \cite{Cr}(Theorem 3.2, (iii)) the relation:
\begin{eqnarray}\label{mot2}
\hspace{1cm} h\rond \mu _{Y}(S_{\alpha}^{\infty})  & = &
h([T_{\alpha }])(uv-1)^{\vert J_{\alpha } \vert}
\\ \nonumber
&& (uv)^{-\vert J_{\alpha } \vert-(N-\vert J_{\alpha } \vert)(m+1) -
\sum_{i\in J_{\alpha }}\alpha _{i}} \mbox{ . }
\end{eqnarray}

As the subset $\gamma ^{-1}(X)$ is equal to the union of $\poi E1{,\ldots,}{t}{}{}{}$,
the subset $\gamma _{\infty }^{-1}(S_{\infty,m,p})$ is given by the equality:
$$\gamma _{\infty }^{-1}(S_{\infty,m,p}) =
\bigcup _{J\subset \{1,\ldots,t\} \atop J\not = \emptyset}
\bigcup _{(\alpha _{i},i\in J)\in ({\Bb N}\backslash \{0\})^{\vert J \vert}
\atop \sum_{i\in J} a_{i}\alpha _{i} \geq p} S_{\alpha}^{\infty} \mbox{ , } $$
for any $p$. We set:
$$M  \ := \ 1 + m \sum_{i=1}^{t}a_{i} \mbox{ . }$$
In particular, when $p \geq M$, $J_{\alpha }$ is not empty as soon as
$\sum_{i \in J} a_{i}\alpha _{i} \geq p$. For any $p \geq M$ and any
$k \in {\Bb N}$, we have the equality:
$$ {\bf I} _{m,p,k} = \sum_{J\subset \{1,\ldots,t\}\atop{J \not = \emptyset}}
\sum_{\alpha _{i}, i\in J\atop {\sum_{i\in J} a_{i}\alpha _{i}\geq p},\alpha _{i}\geq 1}
\mathbf{S}_{J,\alpha,k}$$
where
\begin{eqnarray}\label{star}
&& \mathbf{S}_{J,\alpha,k}  :=
h([T_{\alpha }]) (uv)^{-(N-\vert J_{\alpha } \vert)(m+1) +
\sum_{i\in J\backslash J_{\alpha }}\alpha _{i} }
\\ \nonumber
 & &\hspace{.3cm} (uv-1)^{\vert J_{\alpha } \vert}
(uv)^{-\vert J_{\alpha } \vert} (uv)^{-\sum_{i\in J}\alpha _{i}(b_{i}+1)
-k\sum_{i\in J\cap \{2,\ldots,t\}} c_{i}\alpha _{i}}  \mbox{ . }
\end{eqnarray}
For $J$ nonempty subset in $\{1,\ldots,t\}$, we set:
$$\mathbf{S}'_{J,p,k} = \sum_{\alpha _{i},i\in J
\atop \sum_{i\in J} a_{i}\alpha _{i}\geq p,\alpha _{i}\geq 1}
\mathbf{S}_{J,\alpha ,k} \mbox{ . }$$
Thus ${\bf I} _{m,p,k}$ is the sum of the $\mathbf{S}'_{J,p,k}$. For any $J$, the highest
degree terms in $\mathbf{S}'_{J,p,k}$ are the terms
$\mathbf{S}_{J,\alpha ,p,k}$ for which the number
$$\dim T_{\alpha }-(N-\vert J_{\alpha } \vert)(m+1) +
\sum_{i\in J\backslash J_{\alpha }}\alpha _{i}-
\sum_{i\in J}\alpha _{i}(b_{i}+1) -
k\sum_{i\in J\cap \{2,\ldots,t\}} c_{i}\alpha _{i}$$
is maximal. In particular, for $J=\{1\}$, the highest degree term of
$\mathbf{S}'_{J,p,k}$ does not depend on $k$.\\

i) By hypothesis, the image by $\pi_{m,0}$ of any irreducible component of maximal
dimension of $\gamma _{m}^{-1}(S)$ is dense in $E_{1}$. Therefore, for $J\not=\{1\}$, the
degree of
$$h([\pi _{m,0}^{-1}(E_{J}^{0}) \cap \gamma _{m}^{-1}(S)])$$
is strictly smaller than the degree of the same expression
with $J=\{1\}$. As a consequence, when $J\not=\{1\}$, for
$\alpha=(\alpha _{i})_{i \in J}\in({\Bb N}^*)^{\vert J \vert}$, the degree of
$h([T_{\alpha }])$ is strictly smaller than the degree of
$$h([\pi _{m,0}^{-1}(E_{\{1\}}^{0}) \cap \gamma _{m}^{-1}(S)])\mbox{ . }$$
Hence by Lemma \ref{ljm3} and relation (\ref{star}), when
$p$ is big enough, the highest degree term of ${\bf I} _{m,p,k}$ is the highest degree
term of $\mathbf{S}'_{\{1\},p,k}$. In particular, it does not depend on $k$.\\

ii) Denote by $\widetilde{S}$ the union of irreducible components of $S$ whose image by
$\pi_{m,0}^{V}$ is dense in $X$. By hypothesis, $\widetilde{S}$ is a nonempty set. We set
\begin{eqnarray*}
\widetilde{S}_{\infty,m,p}&:=&\{ \ \nu \in (\pi_{\infty,m}^{V})^{-1}(\widetilde{S})
\ \vert \ F_{X}(\nu) \geq p \ \}
\mbox{  }\end{eqnarray*}
and
\begin{eqnarray*}\widetilde{{\bf I}} _{m,p,k}& := &h\left(\
\int_{\widetilde{S}_{\infty ,m,p}} {\Bb L}^{-kF_{Z}(\nu )} \dd \mu _{V}(\nu ) \ \right)
\mbox{  ,}\end{eqnarray*}
for $k >0$. It suffices to prove that the highest degree terms of ${\bf I} _{m,p,k}$ and
$\widetilde{{\bf I}}_{m,p,k}$ are the same, for $p$ big enough. Indeed, if so, we obtain
the expected result with $k=0$ in ${\bf I} _{m,p,k}$ and $\widetilde{{\bf I}}_{m,p,k}$.
By (i), the highest degree term of $\widetilde{{\bf I}}_{m,p,k}$ does not depend on $k$
for $p$ big enough. In addition, by hypothesis, the highest degree term of
${\bf I}_{m,p,k}$ does not depend on $k$ for $p$ big enough. Fix $p$ big enough such that
the highest degree terms of both $\widetilde{{\bf I}}_{m,p,k}$ and ${\bf I} _{m,p,k}$ do
not depend on $k$. As $S$ and $\widetilde{S}$ are $\mathbf{K}$-invariant, the function
$F_{Z}(\nu )$ is positive for any $\nu $ in
$S_{\infty,m,p}\backslash \widetilde{S}_{\infty,m,p}$ since ${\bf K}$ has finitely
many orbits in $X$. So we get
$$ \lim _{k\rightarrow +\infty }
h\left( \ \int_{S_{\infty,m,p}\backslash \widetilde{S}_{\infty,m,p}}
{\Bb L}^{-kF_{Z}(\nu )}
\dd \mu _{V}(\nu ) \ \right) = 0 \mbox{ . }$$
Hence for $k$ big enough, the highest degree terms of ${\bf I} _{m,p,k}$ and
$\widetilde{{\bf I}}_{m,p,k}$ are the same.
\end{proof}

Let $m,p$ be in ${\Bb N}$ with $p \geq m+1$, and let
$S$ be a constructible subset of $J_{m}(V)$. In view of Proposition \ref{pjm3}, we wish
 to investigate the highest degree term of $h\rond \mu_V(S_{\infty,m,p})$. For $q$ in
${\Bb N} \cup \{\infty\}$ and $q \geq n$, we denote by $\pi_{q,n}^{X}$ the canonical
morphism from $J_q(X)$ to $J_{n}(X)$. The following lemma is easy and helpful for Lemma \ref{l4jm3}.

\begin{lemma}\label{l3jm3}
{\rm i)} If $\nu \in J_q(X)$ for $q \in {\Bb N} \cup \{\infty\}$, then \sloppy
\hbox{$\pi_{q,n}^{V}(\nu)=\pi_{q,n}^{X}(\nu)$} for any $n \leq q$.

{\rm ii)}  For $\nu \in J_{\infty }(V)$ and $q \in \mathbb{N}^*$,
$\pi _{\infty ,q-1}^{V}(\nu ) \in J_{q-1}(X)$ if and only if \hbox{$F_{X}(\nu ) \geq q$.}
\end{lemma}

We suppose that $S$ satisfies the following conditions:
\begin{list}{}{}
\item 1) $S$ is contained in $J_m(X)$,
\item 2) the image by $\pi_{m,0}^{V}$ of $S$ is dense in $X$,
\item 3) for any $\nu \in S$, the fiber $(\pi_{p-1,m}^{X})^{-1}(\nu)$ is a nonempty set.
\end{list}
Denote by $\widetilde{S}$ the union of irreducible components of $S$ whose image by
$\pi_{m,0}^{V}$ is dense in $X$. By Conditions (1) and (2), $\widetilde{S}$ is a nonempty
constructible subset of $J_m(X)$. Let $d$ and $\tilde{d}$ be the dimension of $S$ and
$\widetilde{S}$ respectively, and let $c$ and $\tilde{c}$ be the number of irreducible
components of maximal dimension of $S$ and $\widetilde{S}$ respectively.

Denote by $X_{\reg}$ the smooth part of $X$. Notice that when $S$ is contained
in $X_{\mathrm{reg}}$, Condition (3) is automatically satisfied.

\begin{lemma}\label{l4jm3}
{\rm i)} Let $T$ be an irreducible component of $S$ whose image by $\pi_{m,0}^{V}$ is
dense in $X$. Then $(\pi_{p-1,m}^{X})^{-1}(T)$ has dimension
\sloppy \hbox{$\dim T + (N-r)(p-1-m)$}.

{\rm ii)} The highest degree term of $h \rond \mu_{V}(\widetilde{S}_{\infty,m,p})$ is
$\tilde{c}(uv)^{\tilde{d}-(N-r)(m+1)-pr}$.

{\rm iii)} The degree of $h \rond \mu_{V}(S_{\infty,m,p})$ is at least $d-(N-r)(m+1)-pr$.

{\rm iv)} Suppose that the highest degree term of $h \rond \mu_{V}(S_{\infty,m,p})$ is
$\tilde{c}(uv)^{\tilde{d}-(N-r)(m+1)-pr}$. Then $d=\tilde{d}$ and $c=\tilde{c}$.
\end{lemma}

\begin{proof} i) By hypothesis, the preimage of
$T \cap (\pi_{m,0}^{X})^{-1}(X_{\mathrm{reg}})$ by  $\pi_{p-1,m}^{X}$  is a locally
trivial fibration over
$T \cap (\pi_{m,0}^{X})^{-1}(X_{\mathrm{reg}})$ with fiber ${\Bb A}^{(N-r)(p-1-m)}$,
whence (i).

ii) Denote by $S^{(p-1)}$ and $\widetilde{S}^{(p-1)}$ the preimage by $\pi_{p-1,m}^{X}$
of $S$ and $\widetilde{S}$ respectively. Using Lemma \ref{l3jm3}, we can readily check
that the relation \sloppy \hbox{
$\widetilde{S}_{\infty,m,p}=(\pi_{\infty,p-1}^{V})^{-1}(\widetilde{S}^{(p-1)})$} holds,
since $S$ is contained in $J_m(X)$. Hence, by Proposition \ref{pjm2}, (1), we have:
\begin{eqnarray}\label{h}
h \rond \mu_V (\widetilde{S}_{\infty,m,p})=h([\widetilde{S}^{(p-1)}]) (uv)^{-Np} \mbox{ .}
\end{eqnarray}
By definition of $\widetilde{S}$, the subsets
$$(\pi_{\infty,p-1}^{X})^{-1}(\widetilde{S}^{(p-1)} ) \ \textrm{ and } \
(\pi_{\infty,m}^{X})^{-1}( \widetilde{S} )$$
are stable sets  with respect to $\mu_X$ at level $p-1$ and $m$ respectively. Then, by
Proposition \ref{pjm2}, (1), applied to $\mu_X$, we get:
\begin{eqnarray*}
&&h\rond \mu_X\left( (\pi_{\infty,p-1}^{X})^{-1}(\widetilde{S}^{(p-1)} ) \right)  \ = \
h([\widetilde{S}^{(p-1)}]) (uv)^{-(N-r)p}\mbox{ ,}\\
&&h\rond \mu_X\left( (\pi_{\infty,m}^{X})^{-1}(\widetilde{S}) \right) \ = \
h([\widetilde{S}]) (uv)^{-(N-r)(m+1)} \mbox{ .}
\end{eqnarray*}
As $ (\pi_{\infty,p-1}^{X})^{-1}(\widetilde{S}^{(p-1)} )$ and
$(\pi_{\infty,m}^{X})^{-1}(\widetilde{S})$ coincides, we deduce:
$$h([\widetilde{S}^{(p-1)}])=h([\widetilde{S}]) (uv)^{(N-r)(p-1-m)} \mbox{ .}$$
But the highest degree term of $h([\widetilde{S}])$ is $\tilde{c}(uv)^{\tilde{d}}$,
whence (ii) by Relation (\ref{h}).

iii) Analogously to (ii), we can write:
\begin{eqnarray}\label{h2}
h \rond \mu_V (S_{\infty,m,p})=h([S^{(p-1)}]) (uv)^{-Np} \mbox{ .}
\end{eqnarray}
Let $T$ be an irreducible component of $S$ of dimension $d$. As $S$ satisfies
Condition (3), we have $\pi_{p-1,m}^{X}((\pi_{p-1,m}^{X})^{-1}(T))=T$. So there is an
irreducible component $T'$ of $(\pi_{p-1,m}^{X})^{-1}(T)$ whose image by
$\pi_{p-1,m}^{X}$ is dense in $T$. As the generic fiber of $\pi_{p-1,m}^{X}:
J_{p-1}(X) \longrightarrow J_m(X)$ has dimension $(N-r)(p-1-m)$, the dimension of $T'$ is
at least $d+(N-r)(p-1-m)$. Hence, by Relation (\ref{h2}), the degree of
$h \rond \mu_V (S_{\infty,m,p})$ is at least $d-(N-r)(m+1)-pr$.

iv) By Relation (\ref{h2}), the hypothesis of (iv) means that the highest
degree term of $h([S^{(p-1)}])$ is $\tilde{c}(uv)^{\tilde{d}+(N-r)(p-1-m)}$. In
particular, $S^{(p-1)}$ has dimension $\tilde{d}+(N-r)(p-1-m)$ and the number of its
irreducible components of maximal dimension is $\tilde{c}$. If $d > \tilde{d}$, then
(iii) gives a contradiction. Hence $d=\tilde{d}$. So any
irreducible component of maximal dimension of $\widetilde{S}$ is an irreducible component
of maximal dimension of $S$. Hence, $c \geq \tilde{c}$. It remains to
prove: $c \leq \tilde{c}$.  Let $\poi T1{,\ldots,}{q}{}{}{}$ be the irreducible
components of $S$. Then
$$S^{(p-1)}= (\pi_{p-1,m}^{X})^{-1}(T_1) \cup \ldots \cup (\pi_{p-1,m}^{X})^{-1}(T_q)
\mbox{ .}$$
By Condition (3), $\pi_{p-1,m}^{X}\left((\pi_{p-1,m}^{X})^{-1}(T_i)\right)=T_i$, for any
$i=1,\ldots,q$. Consequently,
$$\pi_{p-1,m}^{X}(T_i) \varsubsetneq \pi_{p-1,m}^{X}(T_j)\mbox{ ,}$$
for any $1\leq i,j \leq q$, with $i\not=j$. Therefore there is an injection from the set
of irreducible components of $S$ into the set of irreducible components of $S^{(p-1)}$.
Moreover, by (i), any irreducible component of $S$ whose image by $\pi_{m,0}^{V}$ is
dense in $X$ provides an irreducible component of $S^{(p-1)}$ of maximal dimension, since
$d=\tilde{d}$. So $\tilde{c} \geq c$, since $\tilde{c}$ is the number of the
irreducible components of maximal dimension of $S^{(p-1)}$.
\end{proof}

\subsection{Key proposition}\label{jm4} We keep the notations of the previous subsection. We explicitly
construct in this subsection a sequence of subsets in $J_m(X)$, for $m \geq 1$,
to which we apply the results of Subsection \ref{jm3}. Here, for $q$ in
${\Bb N} \cup \{\infty\}$ and $q \geq m$, $\pi_{q,m}$ refers to the canonical projection
from $J_{q}(V)$ to $J_m(V)$. As $V$ is a vector space, there is a canonical injection
from $J_{m}(V)$ into $J_{m+1}(V)$. Furthermore, this injection is a
closed immersion. The first projection from $V \times V$ to $V$ is denoted by
$\varpi_1$. Let $T$ be a closed
bicone of $X\times V$ satisfying the two following conditions:
\begin{list}{}{}
\item 1) $T$ is ${\bf K}$-invariant under the diagonal action of ${\bf K}$ in
$V\times V$,
\item 2) for any $(x,y)$ in $T$, $y$ is a tangent vector of $X$ at $x$.
\end{list}
\mbox{ }\\

\noindent We define by induction on $m$ a subset $C_{m}$ of $J_{m}(V)$.\\

For $m=1$, we denote by $C_{1}$ the image of $T$ by the map
$(x,y)\mapsto \nu _{x,y,1}$.\\

Let us suppose that the subset $C_{m}$ of $J_{m}(V)$ is defined for some $m \geq 1$. Let
$C'_{m+1}$ be the image of $C_{m}$ by the canonical injection from $J_{m}(V)$ into
$J_{m+1}(V)$. For any $p\geq m+1$, we set
$$C'_{\infty ,m+1,p}:=
\{ \nu \in \pi _{\infty,m+1}^{-1}(C'_{m+1}) \ \vert \ F_{X}(\nu ) \geq p \} \mbox{ . }$$
Then we set
$$C_{m+1}:= \bigcap\limits_{p \geq m+1} \overline{\pi _{\infty, m+1}(C'_{\infty ,m+1,p})}
\mbox{ , }$$
where $\overline{C}$ denotes the closure of the subset $C$ in $J_{m+1}(V)$. Thus,
$C_{m+1}$ is a closed subset of $J_{m+1}(V)$.\\

For $m \geq 1$, we set
$$D_m:=\{ \ (x,y) \in X\times V \ \vert \ \nu _{x,y,m} \in C_m \ \} \mbox{ . }$$

\begin{lemma}\label{ljm4}  Let $m$ be in ${\Bb N}^*$.

{\rm i)} If $m \geq 2$, then $C_{m}$ is the
closure of $\pi _{\infty ,m}(C'_{\infty ,m,p})$ in $J_{m}(V)$, for $p$ big enough.

{\rm ii)} The subset $C_{m}$ is the image of $D_{m}$ by the
map \sloppy \hbox{$(x,y)\mapsto \nu _{x,y,m}$}. Moreover $C_{m}$ is contained in
$J_{m}(X)$.

{\rm iii)} If $m \geq 2$, then for $p \geq m+1$ and for
$\nu \in \pi_{\infty,m}(C'_{\infty,m,p})$, the fiber at $\nu$ of the canonical morphism
from $J_{p-1}(X)$ to $J_m(X)$ is a nonempty set.

{\rm iv)} The set $C_{m}$ is invariant under the action of ${\bf K}$ in
$J_{m}(V)$.

{\rm v)} If $\varpi _{1}(T)$ is equal to $X$, then the image by $\pi _{m,0}$ of $C_{m}$
is equal to $X$.
\end{lemma}

\begin{proof}
i) By definition, the sequence
$$\pi_{\infty,m}(C'_{\infty,m,m}), \ \pi_{\infty,m}(C'_{\infty,m,m+1}),
\ldots \mbox{ }$$
is weakly decreasing. Therefore, by noetherianity, the weakly decreasing sequence
$$\overline{\pi_{\infty,m}(C'_{\infty,m,m})}, \
\overline{\pi_{\infty,m}(C'_{\infty,m,m+1})},\ldots \mbox{ } $$
of closed subsets in $J_m(V)$ is stationary. Hence for $p$ big enough, $C_{m}$ is the
closure of $\pi _{\infty ,m}(C'_{\infty,m,p})$ in $J_{m}(V)$.

ii) It suffices to prove that $C_m$ is contained in the image of
$D_{m}$ by the map $(x,y)\mapsto \nu _{x,y,m}$ and that $C_m$ belongs to $J_m(X)$. We
prove the statements by induction on $m$. By definition, $C_{1}$ is the image of
$T$ by the map $(x,y)\mapsto \nu _{x,y,1}$ so $C_{1}$ is contained in $J_{1}(X)$ by
Condition (2). So $D_1$ is equal to $T$, whence the two statements for \sloppy \hbox{
$m=1$}. Let us suppose that $C_{m}$ is the image of $D_{m}$ by the map
$(x,y)\mapsto \nu _{x,y,m}$ and let us suppose that $C_{m}$ is contained in $J_{m}(X)$,
for some $m \geq 1$. Then $C'_{m+1}$ is the image of $D_{m}$ by the map \sloppy
\hbox{$(x,y)\mapsto \nu _{x,y,m+1}$}. By definition, $(x,y)$ belongs to $D_{m+1}$ if and
only if $\nu _{x,y,m+1}$ belongs to $C_{m+1}$. As $C_{m+1}$ is contained in $C'_{m+1}$,
we deduce that $C_{m+1}$ is contained in the image of $D_{m+1}$ by the map
$(x,y)\mapsto \nu _{x,y,m+1}$. In addition, for
$p \geq m+2$, $\pi _{\infty ,m+1}(C'_{\infty,m+1,p})$ is contained in
$J_{m+1}(X)$, by definition of $C'_{\infty,m+1,p}$. Hence $C_{m+1}$ is contained in
$J_{m+1}(X)$ since $J_{m+1}(X)$ is closed in $J_{m+1}(V)$.

iii) Let $\nu$ be in $\pi_{\infty,m}(C'_{\infty,m,p})$. Then, there is $\nu'$ in
$C'_{\infty,m,p}$ such that $\pi_{\infty,m}(\nu')=\nu$.  As $\nu'$ belongs to
$C'_{\infty,m,p}$, we have $F_X(\nu') \geq p$. Hence, by Lemma \ref{l3jm3}, (ii),
$\pi_{\infty,p-1}(\nu')$ belongs to $J_{p-1}(X)$. In addition, since $m \geq 2$,
$\pi_{\infty,m}(C'_{\infty,m,p})$ is contained in $C_m$. So by (ii), $\nu $ belongs to
$J_m(X)$. Therefore, by Lemma \ref{l3jm3}, (i), $\pi_{\infty,p-1}(\nu')$ is in the fiber
at $\nu $ of the canonical morphism from $J_{p-1}(X)$ to $J_m(X)$.

iv) We prove the statement by induction on $m$. As $T$ is ${\bf K}$-invariant by
Condition (1), $C_{1}$ is ${\bf K}$-invariant. We suppose $m \geq 2$ and we suppose that
$C_{m}$ is ${\bf K}$-invariant. Then $C'_{m+1}$ is ${\bf K}$-invariant. So
$\pi _{\infty,m+1}^{-1}(C'_{m+1})$ is ${\bf K}$-invariant. As $X$ is
${\bf K}$-invariant, the function $F_{X}$ is ${\bf K}$-invariant. So for any
$p \geq m+1$, $C'_{\infty,m+1,p}$ is ${\bf K}$-invariant. Hence
$\pi_{\infty,m+1}(C'_{\infty,m+1,p})$ and $C_{m+1}$ are ${\bf K}$-invariant.

v) We suppose that $\varpi _{1}(T)$ is equal to $X$. Then for any $x$ in $X$, $(x,0)$
belongs to $T$ since $T$ is a closed bicone. As a consequence, for any
$m \in {\Bb N}^*$, $C_{m}$ contains $\nu _{x,0,m}$. Hence $\pi _{m,0}(C_{m})$ is equal
to $X$.
\end{proof}

\begin{proposition}\label{pjm4}
{\rm i)} For $m$ big enough, we have
$$D_{m}=\{ (x,y) \in T \ \vert \ x+ty \in X, \ \forall t \in {\Bb C} \}
\mbox{ . }$$

{\rm ii)} We suppose that the image by $\varpi _{1}$ of any irreducible component of
maximal dimension of $T$ is equal to $X$. Then, for $m\geq 1$, the image by
$\varpi _{1}$ of any irreducible component of maximal dimension of $D_{m}$ is dense in
$X$.
\end{proposition}

\begin{proof}
i) Let $m$ be in $\mathbb{N}^*$ such that $m-1$ is strictly bigger than the degree of any
element of a generating family of the ideal of definition of $X$ in the algebra of
polynomial functions on $V$. Let $(x,y)$ be in $T$. If $x+ty$ belongs to $X$ for
any $t$ in ${\Bb C}$, then $F_{X}(\nu _{x,y})$ is equal to $\infty $. In particular,
$C_{m}$ contains $\nu _{x,y,m}$ and $D_{m}$ contains $(x,y)$. Conversely, let us suppose
that $D_{m}$ contains $(x,y)$. Then $T$ contains $(x,y)$ since $T$ contains $D_{i}$ for
any $i \in \mathbb{N}^*$. Moreover, for any $\varphi $ in the ideal of definition of $X$,
the function $t\mapsto \varphi (x+ty)$ is divisible by $t^{m}$. So by the choice of $m$,
$\varphi (x+ty)$ is equal to $0$, for any $t \in {\Bb C}$. In other words, $x+ty$
belongs to $X$ for any $t \in {\Bb C}$.

ii) By Lemma \ref{ljm4}, (ii), the statement is equivalent to the following statement: \\

{\it for any $m \in \mathbb{N}^*$, the image by $\pi _{m,0}$ of any irreducible component
of maximal dimension of $C_{m}$ is dense in $X$.}\\

We prove this statement by induction on $m$. It is true for $m=1$ by hypothesis and by
definition of $C_1$. Suppose $m \geq 2$ and suppose the statement true for $m-1$. Let
$\widetilde{C}_{m}$ be the union of irreducible components of $C_{m}$ whose image by
$\pi _{m,0}$ is dense in $X$. By Lemma \ref{ljm4}, (v), $\widetilde{C}_{m}$ is not empty.
Let $d$ and $\tilde{d}$ be the dimension of $C_{m}$ and $\widetilde{C}_m$ respectively
and let $c$ and $\tilde{c}$ be the number of their irreducible components of maximal
dimension. It is enough to prove the equalities $d=\tilde{d}$ and $c=\tilde{c}$. In fact,
in this case, any irreducible component of maximal dimension of $C_m$ is an irreducible
component of $\widetilde{C}_m$. By Lemma \ref{ljm4}, (ii), (iii) and (iv), the
conditions of Lemma \ref{l4jm3} are satisfied.

For $p \geq m+1$, we set:
\begin{eqnarray}\label{B}
\widetilde{C}_{\infty,m,p} & := & \{ \nu \in \pi_{\infty,m}^{-1}(\widetilde{C}_{m})
\ \vert \ F_{X}(\nu) \geq p\} \mbox{ . }
\end{eqnarray}
The image by $\pi _{m,0}$ of any irreducible component of
$\widetilde{C}_{m}$ is dense in $X$. Hence, by Proposition \ref{pjm3}, (i),  the highest
degree term of the element
$$ \widetilde{{\bf I}}_{m,p,k} \ := \ h\left( \ \int_{\widetilde{C}_{\infty,m,p}}
{\Bb L}^{-kF_{Z}(\nu )} \dd \mu _{V}(\nu ) \ \right)$$
of ${\Bb Z}[u,v][[u^{-1},v^{-1}]]$ does not depend on $k$, for $p$ big enough. On the
other hand, by induction hypothesis, the  image by $\pi_{\infty,m-1}$ of any irreducible
component of maximal dimension of $C_{m-1}$ is dense in $X$, so the image by
$\pi_{\infty,m}$  of any irreducible component of maximal dimension of $C'_{m}$ is dense
in $X$. Hence by Proposition \ref{pjm3}, (i)  the
highest degree term of the element
$$ {\bf I}_{m,p,k} \ := \ h\left( \ \int_{C'_{\infty,m,p}} {\Bb L}^{-kF_{Z}(\nu )}
\dd \mu _{V}(\nu ) \ \right)$$
of ${\Bb Z}[u,v][[u^{-1},v^{-1}]]$ does not depend on $k$, for $p$ big enough. At last,
applying besides Lemma \ref{ljm4}, (i), we may choose $p$ big enough such that the
following conditions hold:
\begin{itemize}
\item[1)] $C_{m}$ is the closure of $\pi_{\infty,m}(C'_{\infty,m,p})$,
\item[2)] the highest degree terms of ${\bf I} _{m,p,k}$ and
$\widetilde{{\bf I}}_{m,p,k}$ do not depend on $k$.
\end{itemize}
Using Condition (1), we can easily check that the relation
\begin{eqnarray}\label{c}
C'_{\infty,m,p}=\{ \nu \in \pi_{\infty,m}^{-1}(\pi _{\infty ,m}(C'_{\infty ,m,p}))
\ \vert \ F_X(\nu) \geq p\} \end{eqnarray}
holds. By Lemma \ref{ljm4}, (v), $\pi_{m,0}(C_m)$ is equal to $X$. So, by Condition (1),
the image of $\pi_{\infty,m}(C'_{\infty,m,p})$ by $\pi_{m,0}$ is dense in $X$. Then,
since Condition (2) holds, it results from Proposition \ref{pjm3}, (ii), that the highest
degree term of ${\bf I} _{m,p,k}$ only depends on the irreducible components of
$ \pi_{\infty,m}(C'_{\infty,m,p})$ whose image by $\pi_{m,0}$ is dense in $X$. In other
words, the highest degree terms of ${\bf I} _{m,p,k}$ and $\widetilde{{\bf I}}_{m,p,k}$
are the same.

By Condition (2) with $k=0$ in $\widetilde{{\bf I}}_{m,p,k}$ and Lemma \ref{l4jm3}, (ii),
we deduce that the highest degree term of $\widetilde{{\bf I}} _{m,p,k}$ is
$\tilde{c}(uv)^{\tilde{d}-(N-r)(m+1)-pr}$. So the highest degree term of
${\bf I} _{m,p,k}$ is $\tilde{c}(uv)^{\tilde{d}-(N-r)(m+1)-pr}$, too. Hence, by Condition
(2) with $k=0$ in ${\bf I} _{m,p,k}$, we deduce that the highest degree term of
$h\circ\mu_V(C'_{\infty,m,p})$ is $\tilde{c}(uv)^{\tilde{d}-(N-r)(m+1)-pr}$.  Then, by
Lemma \ref{l4jm3}, (iv), the two equalities $d=\tilde{d}$ and $c=\tilde{c}$ hold.
\end{proof}

\section{Dimension of the nilpotent bicone via motivic integration}\label{d}
We prove in this section the main result of this note. As in Section \ref{su}, we suppose that
${\goth g}$ is simple and we adopt the notations of previous sections. The aim of this
section is the following theorem:

\begin{theorem}\label{td}
{\rm i)} The principal bicone of ${\goth g}$ is a reduced complete intersection of
dimension \sloppy \hbox{$3(\bor g{}-\rk\mathfrak{g}+1)$}.

{\rm ii)} The subscheme ${\cali Y}_{{\goth g}}$ is a reduced complete intersection of
dimension \sloppy \hbox{$3(\bor g{}-\rk\mathfrak{g})+2$}. Moreover, any irreducible component of
${\cali Y}_{{\goth g}}$ is the intersection of ${\cali Y}_{{\goth g}}$ with an irreducible
component of ${\cali X}_{{\goth g}}$.

{\rm iii)} The subscheme ${\cali Z}_{{\goth g}}$ is a reduced complete intersection of
dimension \sloppy \hbox{$3(\bor g{}-\rk\mathfrak{g})+1$}.

{\rm iv)} The nilpotent bicone ${\cali N}_{{\goth g}}$ is a complete intersection of
dimension \sloppy \hbox{$3(\bor g{}-\rk\mathfrak{g})$}.
\end{theorem}

Let us give a brief description of our approach to prove Theorem \ref{td}. We plan to
apply Corollary \ref{csu4}. Namely, we intend to prove that any irreducible component of
maximal dimension of ${\cali X}_{{\goth g}}$ satisfies Conditions (1) and (2) of Corollary
\ref{csu4}. Then we will deduce Theorem \ref{td}, (i), from Lemma \ref{l2su2}, (i). Next,
the statements (ii), (iii) and (iv) of Theorem \ref{td} will be mostly consequences of
statement (i). The main point is therefore to study the images by $\varpi_1$ and
$\varpi_2$ of the irreducible components of maximal dimension of ${\cali X}_{{\goth g}}$,
${\cali Y}_{{\goth g}}$, ${\cali Z}_{{\goth g}}$ and ${\cali N}_{{\goth g}}$. To process, we
consider jet schemes of ${\goth X}_{\goth g}$ and ${\goth N}_{\goth g}$ and we use the
results of Section \ref{jm} about motivic integration. In fact, we remark that an element
$(x,y) \in \gtg g{}$ belongs to ${\cali X}_{\goth g}$ or ${\cali N}_{\goth g}$ if and only
if the arc $t \mapsto x+ty$ is an element of $J_{\infty}({\goth X}_{\goth g})$ or
$J_{\infty}({\goth N}_{\goth g})$ respectively. Thus ${\cali X}_{\goth g}$ and
${\cali N}_{\goth g}$ can be identified to subsets of $J_{\infty}({\goth X}_{\goth g})$
and $J_{\infty}({\goth N}_{\goth g})$ respectively. As  ${\goth X}_{\goth g}$ and
${\goth N}_{\goth g}$ are strictly contained in ${\goth g}$, the motivic measure with
respect to ${\goth g}$ of the so obtained subsets is zero by Proposition \ref{pjm2}, (2).
So we cannot expect to obtain any information from these measures. That is why we are going instead to make use of the subtler construction as described in Section \ref{jm}, \S\ref{jm4}. We will
obtain in this way subsets whose Hodge realization of the motivic measure with respect to
${\goth g}$ provide all the information we need about the varieties ${\cali X}_{\goth g}$
and ${\cali N}_{\goth g}$.

\subsection{Two lemmas}\label{d1}
We give in this subsection two lemmas useful for Subsection \ref{d2}. We denote by
$\theta $ the map
$${\Bb C}\times \gtg g{} \longrightarrow \gtg g{} \mbox{ , }
(t,x,y) \mapsto (x,y+tx) \mbox{ .}$$

\begin{lemma}\label{ld1}
Let ${\cali X}$ be a $G$-invariant irreducible closed bicone of $\gtg g{}$. Suppose
that $\varpi _{1}({\cali X})={\goth X}_{{\goth g}}$ and that
$\theta ({\Bb C}\times {\cali X})$ is contained in ${\cali X}$. If ${\cali X}_{1}$ and
${\cali X}_{2}$ are two irreducible components of the nullvariety of $p_{1,1,1}$ in
${\cali X}$ such that
$\varpi _{1}({\cali X}_{1})=\varpi _{1}({\cali X}_{2})={\goth X}_{{\goth g}}$, then
${\cali X}_{1}={\cali X}_{2}$.
\end{lemma}

\begin{proof} By hypothesis, we can suppose that
${\cali X}_{1}$ and ${\cali X}_{2}$ have codimension $1$ in ${\cali X}$ since ${\cali X}$ is
irreducible. If $(x,y)$ belongs to the nullvariety of $p_{1,1,1}$ in
${\cali X}$, with $x$ semisimple, then $p_{1,1,1}(x,x+y) \not=0$, since $\dv xx \not= 0$ by Lemma \ref{l2su1}, (i).
By hypothesis, for $i=1,2$,
$\theta ({\Bb C}\times {\cali X}_{i})$ is an irreducible constructible subset of
${\cali X}$ which strictly contains ${\cali X}_{i}$. Hence for $i=1,2$,
$\theta ({\Bb C}\times {\cali X}_{i})$ contains a dense open subset of ${\cali X}$. As a
result, for any $(t,x,y)$ in a certain dense open subset of
${\Bb C}\times {\cali X}_{1}$, $x$ is semisimple and there exists $(t',x',y')$
in ${\Bb C}\times {\cali X}_{2}$ such that $\theta (t',x',y')=\theta (t,x,y)$. From this
equality we deduce the equalities:
$$x'=x \mbox{ , } y'-y = (t-t')x \mbox{ .}$$
But $\dv xx\not=0$ since $x$ is a semisimple element of
${\goth X}_{{\goth g}}$. On the other hand, \sloppy \hbox{$\dv xy=\dv x{y'}=0$}. So
$t'=t$ and $(x,y)=(x',y')$. As a consequence \sloppy \hbox{${\cali X}_{1}={\cali X}_{2}$}.
\end{proof}

Let ${\mathrm T}{\goth X}_{{\goth g}}$ be the total tangent space of
${\goth X}_{{\goth g}}$ and let ${\mathrm T}{\goth N}_{{\goth g}}$ be the total tangent
space of ${\goth N}_{{\goth g}}$. They are closed subsets of $\gtg g{}$. Let
${\mathrm T}'{\goth X}_{{\goth g}}$ be the nullvariety of $p_{1,1,1}$ in
${\mathrm T}{\goth X}_{{\goth g}}$ and let ${\mathrm T}''{\goth X}_{{\goth g}}$ be the
nullvariety of $p_{1,0,2}$ in ${\mathrm T}'{\goth X}_{{\goth g}}$.

\begin{lemma}\label{l2d1}
{\rm i)} The subsets ${\mathrm T}{\goth X}_{{\goth g}}$ and
${\mathrm T}{\goth N}_{{\goth g}}$ are $G$-invariant closed irreducible bicones, their
images by $\varpi _{1}$ are equal to ${\goth X}_{{\goth g}}$ and
${\goth N}_{{\goth g}}$ respectively, and they have dimension $4(\bor g{}-\rk\mathfrak{g})+2$
and $4(\bor g{}-\rk\mathfrak{g})$ respectively.

{\rm ii)} The intersection of $\varpi _{1}^{-1}({\goth N}_{{\goth g}})$ and
${\mathrm T}'{\goth X}_{{\goth g}}$ is equal to ${\mathrm T} {\goth N}_{{\goth g}}$.

{\rm iii)} The subset ${\mathrm T}'{\goth X}_{{\goth g}}$ is irreducible of dimension
$4(\bor g{}-\rk\mathfrak{g})+1$. Moreover, $\varpi _{1}(\mathrm{T}' {\goth X}_{{\goth g}})$ is equal
to ${\goth X}_{{\goth g}}$.

{\rm iv)} The image of any irreducible component of $\mathrm{T}''{\goth X}_{{\goth g}}$
by $\varpi_1$ is equal to ${\goth X}_{{\goth g}}$.
\end{lemma}

\begin{proof}
i) As ${\goth X}_{{\goth g}}$ and ${\goth N}_{{\goth g}}$ are $G$-invariant closed cones,
${\mathrm T}{\goth X}_{{\goth g}}$ and ${\mathrm T}{\goth N}_{{\goth g}}$ are
$G$-invariant closed bicones. By Theorem \ref{tjm1} and Remark \ref{rjm1},
${\mathrm T}{\goth X}_{{\goth g}}$ is irreducible and has dimension \sloppy \hbox{
$4(\bor g{}-\rk\mathfrak{g})+2$}. Likewise, by Theorem \ref{tjm1} and Remark \ref{rjm1},
${\mathrm T} {\goth N}_{{\goth g}}$ is irreducible and has dimension \sloppy \hbox{
$4(\bor g{}-\rk\mathfrak{g})$}.

ii) For $x$ in $G. e$ and $y$ in ${\goth g}$, $(x,y)$ belongs to
${\mathrm T}{\goth N}_{{\goth g}}$ if and only if $y$ belongs to $[x,{\goth g}]$. In
particular, for any $x$ in $G. e$ and any $y$ in ${\goth g}$ such that $(x,y)$ belongs to
${\mathrm T}{\goth N}_{{\goth g}}$, $(x,y)$ belongs to
${\mathrm T}'{\goth X}_{{\goth g}}$ since $x$ is orthogonal to
$[x,{\goth g}]$. By (i), the intersection of ${\mathrm T}{\goth N}_{{\goth g}}$ and
$\varpi _{1}^{-1}(G.e)$ is dense in ${\mathrm T}{\goth N}_{{\goth g}}$. So
${\mathrm T}{\goth N}_{{\goth g}} \subset \varpi _{1}^{-1}({\goth N}_{{\goth g}})
\cap {\mathrm T}'{\goth X}_{{\goth g}}$. Let us prove the other inclusion.
As ${\goth X}_{{\goth g}}$ is the nullvariety in ${\goth g}$ of
$\poi q2{,\ldots,}{\ran g{}}{}{}{}$, for any $x$ in ${\goth X}_{{\goth g}}$, the subspace
of elements $y$ of ${\goth g}$ such that $(x,y)$ belongs to
${\mathrm T}{\goth X}_{{\goth g}}$ is the intersection of the kernels of the
differentials at $x$ of $\poi q2{,\ldots,}{\ran g{}}{}{}{}$. By definition of the $q_i$
(see Subsection \ref{su1}), for any $x$ in ${\goth N}_{{\goth g}}$, the subset of
elements $y$ of ${\goth g}$ such that
$(x,y)$ belongs to ${\mathrm T}{\goth X}_{{\goth g}}$ is the intersection of the kernels
of the differentials at $x$ of $\poi p2{,\ldots,}{\ran g{}}{}{}{}$. Therefore, the
intersection of ${\mathrm T}'{\goth X}_{{\goth g}}$ and
$\varpi _{1}^{-1}({\goth N}_{{\goth g}})$ is contained in
${\mathrm T}{\goth N}_{{\goth g}}$, since ${\goth N}_{{\goth g}}$ is the nullvariety of
$\poi p1{,\ldots,}{\ran g{}}{}{}{}$ in ${\goth g}$.

iii) As ${\goth X}_{{\goth g}}$ is a cone, ${\mathrm T}{\goth X}_{{\goth g}}$ contains
the diagonal of ${\goth X}_{{\goth g}}\times {\goth X}_{{\goth g}}$. So
${\mathrm T}{\goth X}_{{\goth g}}$ strictly contains ${\mathrm T}'{\goth X}_{{\goth g}}$.
Then by (i), ${\mathrm T}' {\goth X}_{{\goth g}}$ has dimension $4(\bor g{}-\rk\mathfrak{g})+1$.
Moreover, ${\mathrm T}'{\goth X}_{{\goth g}}$ is an equidimensional $G$-invariant closed
bicone, since $p_{1,1,1}$ is bihomogeneous and $G$-invariant.
Hence by (ii) and Lemma \ref{lint}, the image by $\varpi _{1}$ of any irreducible
component of ${\mathrm T}'{\goth X}_{{\goth g}}$ is equal to
${\goth X}_{{\goth g}}$. As ${\goth X}_{{\goth g}}$ is a cone, for any $t$ in
${\Bb C}$ and any $(x,y)$ in ${\mathrm T}{\goth X}_{{\goth g}}$, $(x,y+tx)$ belongs to
${\mathrm T}{\goth X}_{{\goth g}}$. Therefore, by Lemma \ref{ld1},
${\mathrm T}'{\goth X}_{{\goth g}}$ is irreducible.

iv) By (iii), $\mathrm{T}'' {\goth X}_{{\goth g}}$ is equidimensional of dimension at
least $4(\bor g{}-\rk\mathfrak{g})$. Moreover, since $p_{1,0,2}$ is
bihomogeneous and $G$-invariant, $\mathrm{T}'' {\goth X}_{{\goth g}}$ is a $G$-invariant
closed bicone. As $(e,h)$ belongs to $\mathrm{T}' {\goth X}_{{\goth g}}$,
$\mathrm{T}''{\goth X}_{{\goth g}}$ is strictly contained in
$\mathrm{T}'{\goth X}_{{\goth g}}$ and $\mathrm{T}'' {\goth X}_{{\goth g}}$ is
equidimensional of dimension $4(\bor g{}-\rk\mathfrak{g})$. In addition, as $(e,h)$ belongs to
$\mathrm{T} {\goth N}_{{\goth g}} \setminus \mathrm{T}'' {\goth X}_{{\goth g}}$, we
deduce from (ii) and Lemma \ref{l2su1}, (i), that the intersection of
$\mathrm{T}''{\goth X}_{{\goth g}}$ and $\varpi_{1}^{-1}({\goth N}_{\goth g})$ is
equidimensional of dimension $4(\bor g{}-\rk\mathfrak{g})-1$. Hence, by Lemma \ref{lint}, the image
of any irreducible component of $\mathrm{T}''{\goth X}_{{\goth g}}$ is equal to
${\goth X}_{{\goth g}}$.
\end{proof}

\subsection{Proof of Theorem \ref{td}}\label{d2} In this subsection, we apply Proposition \ref{pjm4} to
suitable $V$, $X$, $T$ and ${\bf K}$ in order to prove Theorem \ref{td}.

\begin{proposition}\label{pd2}
{\rm i)} Let ${\cali X}$ be an irreducible component of maximal dimension of
${\cali X}_{{\goth g}}$ $($resp. ${\cali Y}_{{\goth g}}$, ${\cali Z}_{{\goth g}}$$)$. Then
$\varpi _{1}({\cali X})$ is equal to ${\goth X}_{{\goth g}}$.

{\rm ii)} Let ${\cali X}$ be an irreducible component of maximal dimension of
${\cali N}_{{\goth g}}$, then $\varpi _{1}({\cali X})$ is equal to ${\goth N}_{{\goth g}}$.
\end{proposition}

\begin{proof}
i) Let ${\bf K}$ be the subgroup of GL$({\goth g})$ generated by $G$ and its homotheties.
By Corollary \ref{csu1}, ${\goth X}_{{\goth g}}$ is a ${\bf K}$-invariant irreducible
closed normal cone and it is a complete intersection in ${\goth g}$ with rational
singularities. Moreover, ${\bf K}$ has finitely many orbits in
${\goth X}_{{\goth g}}$. Let $T$ be the bicone ${\mathrm T}{\goth X}_{{\goth g}}$
(resp. ${\mathrm T}'{\goth X}_{{\goth g}}$, ${\mathrm T}''{\goth X}_{{\goth g}}$) of
${\goth X}_{{\goth g}}\times {\goth g}$. By Lemma \ref{l2d1}, (i), (resp. (iii), (iv)),
the image by $\varpi _{1}$  of any irreducible component of $T$ is equal to
${\goth X}_{{\goth g}}$. Then, by Proposition \ref{pjm4}, (i), applied to $V={\goth g}$,
$X={\goth X}_{{\goth g}}$, $T$ and ${\bf K}$, $D_{m}$ is equal to ${\cali X}_{{\goth g}}$
(resp. ${\cali Y}_{{\goth g}}$, ${\cali Z}_{{\goth g}}$), for $m$ big enough. Hence, by
Proposition \ref{pjm4}, (ii), $\varpi _{1}({\cali X)}={\goth X}_{{\goth g}}$ since
$\varpi _{1}({\cali X)}$ is closed by Lemma \ref{lint}.

ii) Let ${\bf K}$ be the subgroup $G$ of
GL$({\goth g})$.
The cone ${\goth N}_{{\goth g}}$ is an irreducible closed normal cone and ${\bf K}$ has
finitely many orbits in ${\goth N}_{{\goth g}}$. Moreover, it is a complete intersection
in ${\goth g}$ and by \cite{He}(Theorem A), ${\goth N}_{{\goth g}}$ has rational
singularities. Let $T$ be the bicone ${\mathrm T}{\goth N}_{{\goth g}}$ of
${\goth N}_{{\goth g}}\times {\goth g}$. By Lemma \ref{l2d1}, (i), $T$ is irreducible
and its image by $\varpi _{1}$ is equal to ${\goth N}_{{\goth g}}$. Then, by
Proposition \ref{pjm4}, (i), applied to $V={\goth g}$, $X={\goth N}_{{\goth g}}$, $T$ and
${\bf K}$, $D_{m}$ is equal to ${\cali N}_{{\goth g}}$, for $m$ big enough. Hence, by
Proposition \ref{pjm4}, (ii), $\varpi _{1}({\cali X)}={\goth N}_{{\goth g}}$ since
$\varpi _{1}({\cali X)}$ is closed by Lemma \ref{lint}.
\end{proof}

Let $d$ be the dimension of ${\cali X}_{{\goth g}}$.

\begin{proposition}\label{p2d2}
Let ${\cali X}$ be an irreducible component of dimension $d$ of ${\cali X}_{{\goth g}}$.

{\rm i)} The dimension of ${\cali Y}_{{\goth g}}$ is $d-1$.

{\rm ii)} The intersection of ${\cali X}$ and ${\cali Y}_{{\goth g}}$ is irreducible
and has dimension $d-1$. Moreover, this intersection is stable under the involution
$(x,y)\mapsto (y,x)$.

{\rm iii)} The intersection of ${\cali X}$ and ${\cali Z}_{{\goth g}}$ is an union of
irreducible components of ${\cali Z}_{{\goth g}}$ of maximal dimension. Moreover,
${\cali Z}_{{\goth g}}$ has dimension $d-2$.

{\rm iv)} The intersection of ${\cali N}_{{\goth g}}$ and ${\cali X}$ is equidimensional of
dimension $d-3$.

{\rm v)} The dimension of ${\cali N}_{{\goth g}}$ is equal to $d-3$.
\end{proposition}

\begin{proof}
i) Let ${\cali X}'$ be an irreducible component of ${\cali Y}_{{\goth g}}$ of maximal
dimension. As ${\cali Y}_{{\goth g}}$ is contained in ${\cali X}_{{\goth g}}$, ${\cali X}'$
is contained in an irreducible component ${\cali X}''$ of ${\cali X}_{{\goth g}}$.
Moreover, ${\cali X}'$ is an irreducible component of the nullvariety of $p_{1,1,1}$ in
${\cali X}''$. So ${\cali X}'$ has codimension at most $1$ in ${\cali X}''$. By Proposition
\ref{pd2}, (i), $\varpi _{1}({\cali X}')$ is equal to ${\goth X}_{{\goth g}}$. Hence
${\cali X}'$ contains $(h,0)$ since it is a closed bicone. As ${\cali X}''$ is an
irreducible component of ${\cali X}_{{\goth g}}$, it is invariant under the action of
GL$_{2}({\Bb C})$. Hence ${\cali X}''$ contains $(h,h)$ and so strictly contains
${\cali X}'$. So ${\cali X}'$ has dimension $\dim {\cali X}''-1$. In particular,
$\dim {\cali X}'$ is at most $d-1$. As any irreducible component of the nullvariety of
$p_{1,1,1}$ in ${\cali X}$ is contained in ${\cali Y}_{{\goth g}}$, ${\cali Y}_{{\goth g}}$
has dimension $d-1$.

ii) As ${\cali Y}_{{\goth g}}$ is stable under the involution $(x,y)\mapsto (y,x)$ and
${\cali X}$ is stable under the action of GL$_{2}({\Bb C})$, their intersection is stable
under the involution $(x,y)\mapsto (y,x)$. By (i), any irreducible component of the
intersection of ${\cali X}$ and ${\cali Y}_{{\goth g}}$ has dimension $d-1$. Hence by (i),
Proposition \ref{pd2}, (i) and Lemma \ref{ld1}, the intersection of
${\cali Y}_{{\goth g}}$ and ${\cali X}$ is irreducible since
$\theta ({\Bb C}\times {\cali X})$ is contained in ${\cali X}$.

iii) Let ${\cali Y}$ be the intersection of ${\cali Y}_{{\goth g}}$ and ${\cali X}$ and let
${\cali Z}$ be the intersection of ${\cali X}$ and ${\cali Z}_{{\goth g}}$. By definition,
${\cali Z}$ is the nullvariety of $p_{1,0,2}$ in ${\cali Y}$. Moreover, by (ii) and
Proposition \ref{pd2}, (i), $\varpi _{2}({\cali Y})$ is equal to ${\goth X}_{{\goth g}}$.
Hence by (i) and (ii), any irreducible component of ${\cali Z}$ has dimension $d-2$. Let
${\cali Z}'$ be an irreducible component of ${\cali Z}_{{\goth g}}$ of maximal dimension.
Then $\dim {\cali Z}' \geq d-2$. Suppose $\dim {\cali Z}' > d-2$. As ${\cali Z}_{{\goth g}}$
is contained in ${\cali Y}_{{\goth g}}$, ${\cali Z}'$ is an irreducible component of
maximal dimension of ${\cali Y}_{{\goth g}}$ by (i). So by Proposition \ref{pd2}, (i),
$\varpi _{2}({\cali Z}')$ is equal to ${\goth X}_{{\goth g}}$. This is impossible since
${\cali Z}' \subset {\cali Z}_{{\goth g}}$. So $\dim {\cali Z}_{{\goth g}}=d-2$ and
${\cali Z}$ is an union of irreducible components of dimension $d-2$ of
${\cali Z}_{{\goth g}}$.

iv) Let ${\cali Y}$ and ${\cali Z}$ be as in (iii). As ${\cali N}_{{\goth g}}$ is contained
in ${\cali Z}_{{\goth g}}$, the intersection of ${\cali X}$ and ${\cali N}_{{\goth g}}$ is
contained in ${\cali Z}$. By (iii) and Proposition \ref{pd2}, (i), any irreducible
component of ${\cali Z}$ is not contained in ${\cali N}_{{\goth g}}$ and has dimension
$d-2$. The intersection of ${\cali X}$ and ${\cali N}_{{\goth g}}$ is the nullvariety of
$p_{1,2,0}$ in ${\cali Z}$. Hence this intersection  is equidimensional of dimension $d-3$.

v) As ${\cali N}_{{\goth g}}$ is contained in ${\cali Z}_{{\goth g}}$, any irreducible
component of ${\cali N}_{{\goth g}}$ is contained in an irreducible component of
${\cali Z}_{{\goth g}}$. Moreover, by (iii) and Proposition \ref{pd2}, (i), any
irreducible component of dimension $d-2$ of ${\cali Z}_{{\goth g}}$ is not contained in
${\cali N}_{{\goth g}}$. Hence the dimension of ${\cali N}_{{\goth g}}$ is at most $d-3$.
So, by (iv), ${\cali N}_{{\goth g}}$ has dimension $d-3$.
\end{proof}

\begin{corollary}\label{c2d2}
{\rm i)} The subscheme ${\cali X}_{{\goth g}}$ is equidimensional of dimension \sloppy
\hbox{$3(\bor g{}-\rk\mathfrak{g}+1)$}. Moreover, any irreducible component of ${\cali X}_{{\goth g}}$
has a nonempty intersection with the intersection of $\Omega _{{\goth g}}$ and
${\cali Z}_{{\goth g}}$.

{\rm ii)} The subscheme ${\cali Y}_{{\goth g}}$ is equidimensional of dimension
$3(\bor g{}-\rk\mathfrak{g})+2$. Moreover, any irreducible component of ${\cali Y}_{{\goth g}}$ has
a nonempty intersection with $\Omega _{{\goth g}}$.

{\rm iii)} The subscheme ${\cali Z}_{{\goth g}}$ is equidimensional of dimension
$3(\bor g{}-\rk\mathfrak{g})+1$. Moreover, any irreducible component of ${\cali Z}_{{\goth g}}$ has
a nonempty intersection with $\Omega _{{\goth g}}$.
\end{corollary}

\begin{proof}
i) Let ${\cali X}$ be an irreducible component of dimension $d$ of
${\cali X}_{{\goth g}}$. Let ${\cali Z}$ and ${\cali N}$ be the intersections of
${\cali X}$ with ${\cali Z}_{{\goth g}}$ and ${\cali N}_{{\goth g}}$ respectively. By
Proposition \ref{p2d2}, (iii) and (iv), ${\cali Z}$ and ${\cali N}$ are equidimensional of
dimension $d-2$ and \sloppy \hbox{$d-3$} respectively. Hence by Proposition \ref{p2d2},
(v), ${\cali N}$ is an union of irreducible components of maximal dimension of
${\cali N}_{{\goth g}}$. As any irreducible component of ${\cali N}_{{\goth g}}$ is stable
under the action of GL$_{2}({\Bb C})$, the image by $\varpi _{2}$ of any irreducible
component of ${\cali N}$ is equal to ${\goth N}_{{\goth g}}$ by Proposition \ref{pd2},
(ii). Moreover, each irreducible component of ${\cali Z}$ contains an irreducible
component of ${\cali N}$ since ${\cali N}$ is the nullvariety of $p_{1,2,0}$ in ${\cali Z}$.
Hence the image by $\varpi _{2}$ of any irreducible component of ${\cali Z}$ is equal to
${\goth N}_{{\goth g}}$. By Proposition \ref{p2d2}, (iii), each irreducible component of
${\cali Z}$ is an irreducible component of ${\cali Z}_{{\goth g}}$ of maximal dimension.
Hence by Proposition \ref{pd2}, (i), the image by $\varpi _{1}$ of any irreducible
component of ${\cali Z}$ is equal to ${\goth X}_{{\goth g}}$. So ${\cali Z}$ satisfies
Conditions (1) and (2) of Corollary \ref{csu4}. Hence by Corollary \ref{csu4}, ${\cali Z}$
has a nonempty intersection with $\Omega _{{\goth g}}$ and ${\cali X}$ has dimension
$3(\bor g{}-\rk\mathfrak{g}+1)$. But by Lemma \ref{l2su2}, (i), any irreducible component of
${\cali X}_{{\goth g}}$ has dimension at least $3(\bor g{}-\rk\mathfrak{g}+1)$. Hence
${\cali X}_{{\goth g}}$ is equidimensional of dimension $3(\bor g{}-\rk\mathfrak{g}+1)$.

ii) By Proposition \ref{p2d2}, (ii) and Lemma \ref{lsu3}, (ii), ${\cali Y}_{{\goth g}}$ is
equidimensional of dimension $3(\bor g{}-\rk\mathfrak{g})+2$. Moreover, any irreducible component of
${\cali Y}_{{\goth g}}$ is the intersection of ${\cali Y}_{{\goth g}}$ and an irreducible
component of ${\cali X}_{{\goth g}}$. Hence by (i) and Corollary \ref{csu4}, any
irreducible component of ${\cali Y}_{{\goth g}}$ has a nonempty intersection with
$\Omega _{{\goth g}}$.

iii) By Proposition \ref{p2d2}, (ii) and Lemma \ref{lsu3}, ${\cali Z}_{{\goth g}}$ is
equidimensional of dimension $3(\bor g{}-\rk\mathfrak{g})+1$. Let ${\cali Z}$ be an irreducible
component of ${\cali Z}_{{\goth g}}$. By Proposition \ref{pd2}, (i),
$\varpi _{1}({\cali Z})$ is equal to ${\goth X}_{{\goth g}}$. Hence by (i) and Proposition
\ref{p2d2}, (v), the nullvariety of $p_{1,2,0}$ in ${\cali Z}$ is an union of irreducible
components of ${\cali N}_{{\goth g}}$ of maximal dimension. Then by Proposition \ref{pd2},
(ii), $\varpi _{2}({\cali Z})$ is equal to ${\goth N}_{{\goth g}}$ since any irreducible
component of ${\cali N}_{{\goth g}}$ is stable under the involution $(x,y)\mapsto (y,x)$.
So by Lemma \ref{lsu4}, ${\cali Z}$ has a nonempty intersection with
$\Omega _{{\goth g}}$.
\end{proof}

We can now give the proof of Theorem \ref{td}:

\begin{proof}
i) By Corollary \ref{c2d2}, (i), Lemma \ref{l2su2}, (i) and Lemma \ref{l2su3}, (ii),
${\cali X}_{{\goth g}}$ is a complete intersection of dimension $3(\bor g{}-\rk\mathfrak{g}+1)$ and
any irreducible component of ${\cali X}_{{\goth g}}$ contains a smooth point. So
${\cali X}_{{\goth g}}$ is generically reduced and by \cite{Ma}(Ch. 8, \S 23), the scheme
${\cali X}_{{\goth g}}$ is reduced.

ii) By Corollary \ref{c2d2}, (ii), Lemma \ref{lsu3}, (ii) and Lemma \ref{l2su3}, (ii),
${\cali Y}_{{\goth g}}$ is a complete intersection of dimension $3(\bor g{}-\rk\mathfrak{g})+2$ and
any irreducible component of ${\cali Y}_{{\goth g}}$ contains a  smooth point. So
${\cali Y}_{{\goth g}}$ is generically reduced and by \cite{Ma}(Ch. 8, \S 23), the scheme
${\cali Y}_{{\goth g}}$ is reduced. Moreover, by Proposition \ref{p2d2}, (ii), any
irreducible component of ${\cali Y}_{{\goth g}}$ is the intersection of
${\cali Y}_{{\goth g}}$ and an irreducible component of ${\cali X}_{{\goth g}}$.

iii) By Corollary \ref{c2d2}, (iii), Lemma \ref{l2su2}, (i) and Lemma \ref{l2su3}, (ii),
${\cali Z}_{{\goth g}}$ is a complete intersection of dimension $3(\bor g{}-\rk\mathfrak{g})+1$ and
any irreducible component of ${\cali Z}_{{\goth g}}$ contains a  smooth point. So
${\cali Z}_{{\goth g}}$ is generically reduced and by \cite{Ma}(Ch. 8, \S 23), the scheme
${\cali Z}_{{\goth g}}$ is reduced.

iv) As ${\cali N}_{{\goth g}}$ is the nullvariety of $\bor g{}+\rk\mathfrak{g}$ polynomial functions,
any irreducible component of ${\cali N}_{{\goth g}}$ has dimension at least
$3(\bor g{}-\rk\mathfrak{g})$. Hence by Corollary \ref{c2d2}, (i) and Proposition \ref{p2d2}, (v),
${\cali N}_{{\goth g}}$ is a complete intersection of dimension $3(\bor g{}-\rk\mathfrak{g})$.
\end{proof}

We deduce Theorem \ref{t2int} from Theorem \ref{td}, (i) and Proposition \ref{pd2},
(i) and we deduce the main part of Theorem \ref{tint} from Theorem \ref{td}, (iv) and
Proposition \ref{pd2}, (ii). To complete the proof of Theorem \ref{tint}, it only remains
to show that ${\cali N}_{{\goth g}}$ is a nonreduced scheme, what we will do in Section
\ref{p}.

\section{Applications to invariant theory}\label{i}
In this section, we present various applications of previous results to invariant theory.
In this section, ${\goth g}$ is supposed to be simple again.

\subsection{Properties of the algebra of $\mathrm{pol}_2\mathrm{S}(\mathfrak{g})^{\mathfrak{g}}$} We denote by
$\pol g{}2$ the subalgebra of $\cg g{}$ generated by the $2$-order polarizations of
elements of $\ai g{}{}$. Since the polynomials $\poi p1{,\ldots,}{\ran g{}}{}{}{}$ generate
$\ai g{}{}$, the polynomials $p_{i,m,n}$, where $i=1,\ldots,\rk\mathfrak{g}$, $m+n=d_i$ generate
$\pol g{}2$ as graded algebra. Besides, the morphism $\sigma $ introduced in Section \ref{sc}, \S\ref{sc4}, is the morphism of affine varieties whose comorphism is the canonical
injection from $\pol g{}2$ into $\cg{g}$. Then, by Proposition \ref{psc4}, we can state:

\begin{proposition}\label{pi1}
The subalgebra $\pol g{}2$ is a polynomial algebra in \sloppy \hbox{$\bor g{}+\ran g{}$}
variables.
\end{proposition}

We deduce from Proposition \ref{pi1}:

\begin{theorem}\label{ti1} The morphism $\sigma $ is faithfully flat. Equivalently, the
extension $\cg g{}$ of $\pol g{}2$ is faithfully flat.
\end{theorem}

\begin{proof} As $\pol g{}2$ is generated by homogeneous functions, the fiber at $0$ of
the morphism $\sigma $ has maximal dimension. On the other hand, by Proposition
\ref{pi1}, $\pol g{}2$ is a polynomial algebra in $\bor g{}+\rk\mathfrak{g}$ variables. So $\sigma $
is an equidimensional morphism and by \cite{Ma}(Ch. 8, Theorem 21.3), $\sigma $ is a
flat morphism. In particular by \cite{Ha}(Ch. III, Exercise 9.4), it is an
open morphism whose image contains $0$. So $\sigma $ is surjective. Hence $\sigma $ is
faithfully flat, according to \cite{Ma}(Ch. 3, Theorem 7.2).
\end{proof}

\subsection{The nullcone} The {\it nullcone} ${\cali V}_{{\goth g}}$ of the $G$-module $\gtg g{}$ is
the nullvariety in  $\gtg g{}$ of the augmentation ideal of the ideal of $\cg g{}$
generated by $(\cg g{})^{{\goth g}}$. The nullcone plays a leading part in invariant
theory. It is studied in \cite{Ri5}, \cite{Po}, \cite{LMP}, and recently in \cite{KrW1}
and \cite{KrW2}. By \cite{KrW1}, ${\cali V}_{{\goth g}}$ is irreducible and has dimension
$3(\bor g{}-\rk\mathfrak{g})$. Furthermore, N. Wallach and H. Kraft conjecture in \cite{KrW2} that
${\cali V}_{{\goth g}}$ is an irreducible component of ${\cali N}_{{\goth g}}$. Clearly,
Theorem \ref{tint} confirms that conjecture. Thereby, we claim:

\begin{theorem}\label{ti2} The nullcone ${\cali V}_{{\goth g}}$ is an irreducible component of
the nilpotent bicone ${\cali N}_{{\goth g}}$.
\end{theorem}

Let $k$ be in ${\Bb N}^*$ and let us denote by ${\goth g}^{k}$ the
$k$-th cartesian power of ${\goth g}$. We extend the previous notions to ${\goth g}^k$.
Let $\es S{{\goth g}^k}$ be the symmetric algebra of ${\goth g}^k$ and let
$\es S{{\goth g}^k}^{{\goth g}}$ be the subalgebra of its invariant elements under the
diagonal action of $G$ in ${\goth g}^{k}$. Then we denote by ${\cali V}_{{\goth g}}^{(k)}$
the nullvariety in ${\goth g}^k$ of the augmentation ideal
$\es S{{\goth g}^k}_{+}^{{\goth g}}$ of $\es S{{\goth g}^k}^{{\goth g}}$. For $\varphi $
in $\e Sg$, the {\it $k$-order polarizations} $\varphi _{\poi i1{,\ldots,}{k}{}{}{}}$ of
$\varphi $ are defined by the relation:
$$\varphi (t_1 x_1 + \cdots + t_k x_k) =
\sum_{(\poi i1{,\ldots,}{k}{}{}{})\in {\Bb N}^{k} } t_{1}^{i_1} \ldots t_{k}^{i_k}
\varphi _{i_1,\ldots,i_k}(x_1,\ldots,x_k) \mbox{ , }$$
for $(\poi x1{,\ldots,}{k}{}{}{})$ in ${\goth g}^{k}$ and $(\poi t1{,\ldots,}{k}{}{}{})$
in ${\Bb C}^{k}$. Then, we denote by $\pol g{}k$ the subalgebra of S$({\goth g}^{k})$
generated by the $k$-order polarizations of elements in $\ai g{}{}$. Then
$\pol g{}k$ is an homogeneous subalgebra of S$({\goth g}^{k})$ and its augmentation ideal
$\mathrm{pol}_k\e Sg_{+}^{{\goth g}}$ is generated by the $k$-order polarizations of
the elements of $\e Sg_{+}^{{\goth g}}$. We denote by ${\cali N}_{{\goth g}}^{(k)}$ the
nullvariety in ${\goth g}^k$ of the ideal $\mathrm{pol}_k\e Sg_{+}^{{\goth g}}$.
Obviously, ${\cali V}_{{\goth g}}^{(1)}={\cali N}_{{\goth g}}^{(1)}={\goth N}_{{\goth g}}$,
${\cali V}_{{\goth g}}^{(2)}={\cali V}_{{\goth g}}$ and
${\cali N}_{{\goth g}}^{(2)}={\cali N}_{{\goth g}}$.\\

The {\it index of polarization of ${\goth g}$}, denoted by
${\rm polind}({\goth g})$, is defined in \cite{LMP} as being the upper bound of
$k \in {\Bb N}^*$ such that ${\cali V}_{\goth g}^{(k)}={\cali N}_{\goth g}^{(k)}$. The
inequality ${\rm polind}({\goth g})  \geq 1$ obviously always holds. We wish next to
establish an equality.\\

Recall that $(e,h,f)$ is a principal ${\goth {sl}}_2$-triple of ${\goth g}$.

\begin{lemma}\label{li3} Let $x$ be in ${\goth g}$ satisfying $(\ad x)^2(e)=0$. Then
$(e,[x,e])$ belongs to  ${\cali N}_{{\goth g}}$.
\end{lemma}

\begin{proof} Since $(\ad x)^2(e)=0$, we have $\exp(t\ad x)(e)=e+t[x,e]$, for any $t$ in
${\Bb C}$. As ${\goth N}_{{\goth g}}$ is a closed cone in ${\goth g}$, the element
$se+t[x,e]$ belongs to ${\goth N}_{{\goth g}}$, for any $(s,t)$ in ${\Bb C}^2$. So
$(e,[x,e])$ belongs to ${\cali N}_{{\goth g}}$.
\end{proof}

By the Hilbert-Mumford criterion \cite{Kr}(Kap.\ II), an element  $(x,y)$ of $\gtg g{}$
belongs to the nullcone ${\cali V}_{\goth g}$ if and only if there is a one-parameter
subgroup $\lambda:{\Bb C}^* \rightarrow G$ such that
$\lim_{t \to 0} \lambda(t)(x,y)=0$. In particular, if $(x,y)$ belongs to the nullcone,
then the subalgebra $L$ generated by $x$ and $y$ is contained in the nilpotent cone.
Hence $L$ is nilpotent and so contained in a Borel subalgebra of ${\goth g}$. Conversely, if $x$ and $y$ belong to the nilpotent radical
of a common Borel subalgebra, then the previous criterion applies to $(x,y)$. To summarize, we have actually shown:

\begin{lemma}\label{l2i3} The element $(x,y)$ of $\gtg g{}$ belongs to the
nullcone ${\cali V}_{\goth g}$ if and only if $x$ and $y$ belong to the nilpotent radical
of a common Borel subalgebra.
\end{lemma}

The following proposition has already been noticed in \cite{LMP}(Theorem 3.16). We provide
here a shorter proof.

\begin{proposition}\label{pi3} We have: 
$${\rm polind}({\goth g}) =\left\{
\begin{array}{ll}
\infty\mbox{ , }&\hbox{\textrm{if }}{\goth g}\hbox{\textrm{ is isomorphic to }}{\goth {sl}}_2({\Bb C});\\
1\mbox{ , }&\hbox{otherwise.}
\end{array}
\right.$$
\end{proposition}

\begin{proof} Suppose first $\dim {\goth g}=3$. The elements $x$
in ${\goth g}$ such that the subspace generated by $x$ and $e$ is contained in the
nilpotent cone of ${\goth g}$ are colinear to $e$. Hence
${\cali V}_{\goth g}^{(2)}={\cali N}_{\goth g}$, and ${\cali N}_{{\goth g}}$ is irreducible.
Then for $k \geq 2$ and $\poi x1{,\ldots,}{k}{}{}{}$ in ${\goth g}$, the
subspace generated by $\poi x1{,\ldots,}{k}{}{}{}$ is contained in
${\goth N}_{{\goth g}}$ if and only if its dimension is smaller than $1$. So
${\cali V}_{\goth g}^{(k)}$ is equal to ${\cali N}_{\goth g}^{(k)}$ for any $k$.

Suppose now $\dim {\goth g} > 3$. As $\poi d1{,\ldots,}{\ran g{}}{}{}{}$ is a weakly
increasing sequence, $-2(d_{\ran g{}}-1)$ is the smallest eigenvalue of $\ad h$. Let $v$ be
a non zero eigenvector of $\ad h$ of eigenvalue $-2(d_{\ran g{}}-1)$ and set
$v_{0}=[v,e]$. Then $(\ad v)^2(e)$ is an eigenvector of $\ad h$ of eigenvalue
$-4(d_{\ran g{}}-1)+2$. As ${\goth g}$ is simple and $\dim \mathfrak{g} > 3$, we
have $d_{\ran g{}} > 2$ and $-4(d_{\ran g{}}-1)+2 < -2(d_{\ran g{}}-1)$. So $(\ad v)^{2}(e)$
is equal to $0$. Hence, by Lemma \ref{li3}, $(e,v_{0})$ belongs to
${\cali N}_{{\goth g}}$. As $v_{0}$ is an eigenvector of $\ad h$ of negative eigenvalue,
${\goth b}$ does not contain $v_{0}$. Since ${\goth b}$ is the unique Borel subalgebra
which contains $e$, we deduce from Lemma \ref{l2i3} that $(e,v_{0})$ does not belong to
${\cali V}_{{\goth g}}$. So ${\cali V}_{{\goth g}}$ is strictly contained in
${\cali N}_{{\goth g}}$. As a result, for $k \geq 2$,
${\cali V}_{\goth g}^{(k)}$ is strictly contained in ${\cali N}_{\goth g}^{(k)} $ since
${\cali V}_{{\goth g}}\times \{0_{{\goth g}^{k-2}}\}$ and
${\cali N}_{{\goth g}}\times \{0_{{\goth g}^{k-2}}\}$ are the intersections of
${\cali V}_{\goth g}^{(k)}$ and ${\cali N}_{\goth g}^{(k)}$  with
$\gtg g{}\times \{0_{{\goth g}^{k-2}}\}$ respectively.
\end{proof}

\begin{remark}
The nullcone has a natural structure of scheme. This scheme is irreducible. 
Furthermore, for $\dim {\goth g}>3$, it is not reduced. 
If this scheme were reduced, the 
extension $\cg{g}$ of $(\cg{g})^{{\goth g}}$ would be flat by arguments used in
\cite{Kos2}. But this extension is not equidimensional since the $G$-orbit in $\gtg{g}$ in
general position has dimension $2 \bor g{}-\rk\mathfrak{g}$ and $2 \bor g{}-\rk\mathfrak{g}>3(\bor g{} -\rk\mathfrak{g})$.
\end{remark}

\section{Additional properties of the nilpotent bicone}\label{p}
In this section, we give additional results about the nilpotent bicone
${\cali N}_{{\goth g}}$. We assert in the first subsection that ${\cali N}_{{\goth g}}$ is
not reduced. Next, we discuss in the second subsection the number of its irreducible components. As before, we fix a principal
$\mathfrak{sl}_2$-triple $(e,h,f)$ in ${\goth g}$ and we use the notations of the Introduction, \S\ref{notation}.

\subsection{Nonreducibility of ${\cali N}_{{\goth g}}$} According to \cite{Ma}(\S 23), a Noetherian ring $A$ is reduced if and
only if $A$ satisfies Conditions $(R_0)$ and $(S_1)$. By Theorem \ref{td}, the
subscheme ${\cali N}_{{\goth g}}$ of $\gtg g{}$ defined by the ideal
${\rm pol}_{2}\e Sg_{+}^{{\goth g}}$ is a complete intersection in $\gtg g{}$. Therefore
the quotient ring $(\cg g{})/{\rm pol}_{2}\e Sg_{+}^{{\goth g}}$ is Cohen-Macaulay,
according to \cite{Ma}(Ch.\ 8, Theorem 21.3). It is natural to ask if
$(\cg g{})/{\rm pol}_{2}\e Sg_{+}^{{\goth g}}$ is reduced. The following theorem answers
this question negatively.

\begin{theorem}\label{tp1} The scheme ${\cali N}_{{\goth g}}$ is not reduced. \end{theorem}

\begin{proof} In view of the preceding remarks, it suffices to find an irreducible
component of ${\cali N}_{{\goth g}}$ which does not contain any smooth point of
${\cali N}_{{\goth g}}$. Hence, according to Proposition \ref{psc4}, Lemma \ref{l2su3},
(ii) and Theorem \ref{ti2}, it is enough to prove that the nullcone
${\cali V}_{{\goth g}}$ has an empty intersection with the open subset
$\Omega_{{\goth g}}$. We suppose that this intersection is not empty and
we expect a contradiction. Let $(x,y)$ be in the intersection of ${\cali V}_{{\goth g}}$
and $\Omega _{{\goth g}}$. By Lemma \ref{l2i3}, we can suppose that $x$ and $y$ belongs
to ${\goth u}$. By Lemma \ref{l4sc1}, (ii), for any $(a,b)$ in
${\Bb C}^2 \setminus \{(0,0)\}$, the
centralizer of $ax+by$ is contained in ${\goth u}$. Therefore, it results from Lemma
\ref{lsc1} that the subspace ${\goth V}'(x,y)$ is contained in ${\goth u}$.
Moreover, by Lemma \ref{l5sc1}, (iii), ${\goth V}'(x,y)$ is equal to ${\goth V}(x,y)$.
But by Lemma \ref{l5sc1}, (i), ${\goth V}(x,y)$ has dimension $\bor g{} > \dim {\goth u}$,
since $(x,y)$ belongs to $\Omega _{{\goth g}}$, whence the expected contradiction.
\end{proof}

The proof of Theorem \ref{tint} is now completed by Theorem \ref{tp1}.

\subsection{Irreducible components of ${\cali N}_{{\goth g}}$} Denote by $\Phi $ the map
$$G\times \gtg{g} \longrightarrow \gtg{g} \mbox{ , }(g,x,y) \mapsto (g(x),g(y)) \mbox{ .}$$
For a subset $X$ of $\gtg{g}$, the closure in $\gtg{g}$ of the image of \sloppy \hbox{
$G\times X$} by $\Phi $ is denoted by $X^{\Phi }$. For $Y$ subset of
$\gtg g{}$ and $x$ in ${\goth g}$, we denote by $Y_{x}$ the subset of elements $y$ in
${\goth g}$ such that $(x,y)$ belongs to $Y$.

\begin{lemma}\label{lp2} Let $Y$ be a $G$-invariant closed bicone of $\gtg{g}$ and let
$x_{0}$ be a nilpotent element of ${\goth g}$. Suppose that the image of any irreducible
component of $Y$ by $\varpi _{1}$ is equal to the closure in ${\goth g}$ of
$G. x_{0}$. Suppose besides that there exists a subvariety $T$ of $G$ such
that the map $g\mapsto g(x_{0})$ is an isomorphism from $T$ to an open subset of
$G. x_{0}$ containing $x_{0}$.

{\rm i)} The subset $Y$ is equal to $(\{x_{0}\}\times Y_{x_{0}})^{\Phi }$.

{\rm ii)} If $Z$ is an irreducible component of $Y$, then $Z_{x_{0}}$ is irreducible.

{\rm iii)} The map $X\mapsto (\{x_{0}\}\times X)^{\Phi }$ is a bijection from the set
of irreducible components of $Y_{x_{0}}$ to the set of irreducible components of $Y$.
\end{lemma}

\begin{proof}
i) By our hypothesis, for any irreducible component $Z$ of $Y$, $\varpi _{1}(Z)$ is the
closure of $G. x_{0}$. So
$G. x_{0}$ is a dense open subset in $\varpi _{1}(Z)$. Then the subset $Y'$ of
elements of $Y$ whose first component is in $G. x_{0}$ is a dense open subset in
$Y$. As $Y$ is a $G$-invariant bicone, $Y_{x_{0}}$ is a cone and $Y'$ is the image by
$\Phi $ of $G\times \{x_{0}\}\times Y_{x_{0}}$. So $Y$ is equal to
$(\{x_{0}\}\times Y_{x_{0}})^{\Phi }$.

ii) Let $Z$ be an irreducible component of $Y$. Then $Z$ is a $G$-invariant closed bicone.
We denote by $\Omega $ the image of $T$ by the map $g\mapsto g(x_{0})$ and we denote by
$\widetilde{Z}$ the subset of elements $(x,y)$ in $Z$ such that $x$ belongs to $\Omega $.
By our hypothesis, $\widetilde{Z}$ is a nonempty open subset of $Z$. In particular, it
is irreducible. Let $\tau $ be the inverse map of the map $g\mapsto g(x_{0})$ from $T$ to
$\Omega $. Then the map
$$ \widetilde{Z} \longrightarrow T\times Z_{x_{0}} \mbox{ , }
(x,y) \mapsto (\tau (x),\tau (x)^{-1}(y) )\mbox{ ,}$$
is an isomorphism. So $Z_{x_{0}}$ is irreducible.

iii) Let $X$ be an irreducible component of $Y_{x_{0}}$. Then
$(\{x_{0}\}\times X)^{\Phi }$ is irreducible as the closure of the image of an
irreducible variety by a regular map. If $Z$ is an irreducible component of $Y$ which
contains $(\{x_{0}\}\times X)^{\Phi }$ then $Z_{x_{0}}$ contains $X$. So by (ii), $X$ is
equal to $Z_{x_{0}}$ and $Z$ is equal to $(\{x_{0}\}\times X)^{\Phi }$ since
$\Phi (G\times \{x_{0}\}\times X)$ is open in $Z$. We then deduce that the map
$Z\mapsto Z_{x_{0}}$ is the inverse of the map $X\mapsto (\{x_{0}\}\times X)^{\Phi }$.
\end{proof}

Let us recall that ${\bf B}_{-}$ is the normalizer in $G$ of the Borel subalgebra
\sloppy \hbox{${\goth b}_-={\goth h} \oplus {\goth u}_-$}.

\begin{lemma}\label{l2p2}
Let ${\bf U}$ be the unipotent radical of the normalizer of ${\goth b}$ in $G$. Then
there exists a closed subset $Z$ of ${\bf U}$ such that the map
$${\bf B}_{-}\times Z \rightarrow G. e \mbox{ , } (g,k) \mapsto gk(e) \mbox{ ,}$$
is an isomorphism onto an open subset $V $ of $G. e$ containing $e$.
\end{lemma}

\begin{proof}
As ${\bf B}_{-}{\bf U}$ is an open subset of $G$, its image by the map $g\mapsto g(e)$
from $G$ to $G. e$ is an open subset $V$ containing $e$. Moreover, the map
$${\bf B}_{-}\times {\bf U}.e \rightarrow V \mbox{ , } (g,x) \mapsto g(x)$$
is an isomorphism since ${\bf U}$ contains the stabilizer of $e$ in $G$ and
${\bf B}_{-} \cap {\bf U}= \{{\bf 1}_{{\goth g}}\}$. By
\cite{Pu}(Ch. I, Part. II, \S 3), there exists a complement ${\goth m}$ of
${\goth g}(e)$ in ${\goth u}$ and a basis $\poi v1{,\ldots,}{m}{}{}{}$ of ${\goth m}$ such
that the map
$${\Bb C}^{m} \rightarrow {\bf U}.e \mbox{ , }
(\poi t1{,\ldots,}{m}{}{}{}) \mapsto \exp(t_{1}\ad v_{1})\cdots \exp(t_{m}\ad v_{m})(e)$$
is an isomorphism. Then the map
$${\Bb C}^{m} \rightarrow {\bf U} \mbox{ , }
(\poi t1{,\ldots,}{m}{}{}{}) \mapsto \exp(t_{1}\ad v_{1})\cdots \exp(t_{m}\ad v_{m}) $$
is proper and its image $Z$ is closed in ${\bf U}$. Moreover, the map
$${\bf B}_{-}\times Z \rightarrow V \mbox{ , } (g,k)\mapsto gk(e) $$
is an isomorphism.
\end{proof}

Let ${\cali N}_{{\goth g},e}$ be the subset of elements $y$ of ${\goth g}$ such that
$(e,y)$ is in ${\cali N}_{{\goth g}}$. From Theorem \ref{tint}, Lemma \ref{lp2}, (iii)
and Lemma \ref{l2p2}, we deduce the following result:

\begin{corollary}\label{cp2}
The map $X\mapsto (\{e\}\times X)^{\Phi }$ is a bijection from the set of irreducible
components of ${\cali N}_{{\goth g},e}$ onto the set of irreducible components of
${\cali N}_{{\goth g}}$.
\end{corollary}

According to Section \ref{su}, \S\ref{su3}, ${\cali N}_{{\goth g},e}$ has a natural structure of
scheme. The following result is also a corollary of Theorem \ref{tint}.

\begin{corollary}\label{c2p2} The subset ${\cali N}_{{\goth g},e}$ is a complete intersection of
dimension $\bor g{}-\rk\mathfrak{g}$.
\end{corollary}

\begin{proof} As ${\cali N}_{{\goth g},e}$ is the nullvariety in ${\goth g}$ of
$\bor g{}$ polynomial functions, the dimension of any irreducible component of
${\cali N}_{{\goth g},e}$ is at least $\bor g{}-\rk\mathfrak{g}$. Let $Y$ be an irreducible
component of ${\cali N}_{{\goth g}}$. By Theorem \ref{tint}, since $G.e$ is an open dense
subset of ${\goth N}_{{\goth g}}$, $Y_e$ has dimension $\bor g{}-\rk\mathfrak{g}$. On the other
hand, by Corollary \ref{cp2}, $Y_e$ is an irreducible component of
${\cali N}_{{\goth g},e}$ and any irreducible component of
${\cali N}_{{\goth g},e}$ is obtained in this way, whence the corollary.
\end{proof}

Let ${\goth p}={\goth l}_{{\goth p}} \oplus {\goth u}_{{\goth p}}$ be a parabolic
subalgebra of ${\goth g}$ containing ${\goth b}$, where ${\goth l}_{{\goth p}}$ is the
${\goth h}$-stable Levi subalgebra of ${\goth p}$ and where ${\goth u}_{{\goth p}}$ is the
nilpotent radical of ${\goth p}$. We denote by $\varpi_{{\goth l}_{{\goth p}}}$ and
$\varpi_{{\goth u}_{{\goth p}}}$ the projections from ${\goth p}$ to
${\goth l}_{{\goth p}}$ and ${\goth u}_{{\goth p}}$ respectively.

\begin{lemma}\label{l3p2}
The map $Y\mapsto Y + {\goth u}_{{\goth p}}$ is a one-to-one correspondence between the
set of irreducible components of
${\cali N}_{{\goth l}_{{\goth p}},\varpi _{{\goth l}_{{\goth p}}}(e)}$ and
the set of irreducible components of ${\cali N}_{{\goth g},e}$ contained in ${\goth p}$.
\end{lemma}

\begin{proof} Let $Y$ be an irreducible component of
${\cali N}_{{\goth l}_{{\goth p}},\varpi _{{\goth l}_{{\goth p}}}(e)}$. Then the set
\sloppy \hbox{$Y+{\goth u}_{{\goth p}}$} is contained in the intersection of
${\cali N}_{{\goth g},e}$ and ${\goth p}$. As $e$ is a
regular nilpotent element of ${\goth g}$, $\varpi _{{\goth l}_{{\goth p}}}(e)$ is a
regular nilpotent element of ${\goth l}_{{\goth p}}$. So by Corollary \ref{c2p2}, $Y$ has
dimension $\bor l{{\goth p}}-\ran g{}$. Then $Y+{\goth u}_{{\goth p} }$ has dimension
$$\bor l{{\goth p}}-\rk\mathfrak{g}+\dim {\goth u}_{{\goth p}}= \bor g{}-\rk\mathfrak{g} \mbox{ .}$$
According to Corollary \ref{c2p2}, we deduce that $Y+{\goth u}_{{\goth p}}$ is an
irreducible component of ${\cali N}_{{\goth g},e}$. Thus, the above map is well-defined.
It is clearly injective. Prove now that it is surjective.

Let $Z$ be an irreducible component of ${\cali N}_{{\goth g},e}$ contained in ${\goth p}$.
Then $\varpi _{{\goth l}_{{\goth p}}}(Z)$ is contained in an irreducible component of
${\cali N}_{{\goth l}_{{\goth g}},\varpi _{{\goth l}_{{\goth p}}}(e)}$. Let $Y$ be such
an irreducible component. Then $Y+{\goth u}_{{\goth p}}$ is an irreducible subset
of ${\cali N}_{{\goth g},e}$ which contains $Z$. So, $Z$ is equal to
$Y+{\goth u}_{{\goth p}}$.
\end{proof}

Let $N_{{\goth g}}$ be the set of irreducible components of $\mathcal{N}_{{\goth g},e}$.
We denote by $\Upsilon_{{\goth g}}$ the set of proper parabolic subalgebras strictly
containing ${\goth b}$. Let $N_{{\goth g}}^{'}$ be the subset of irreducible components
of ${\cali N}_{{\goth g},e}$ which are not contained in any element of
$\Upsilon_{{\goth g}}$.

\begin{lemma}\label{l4p2} If $\dim {\goth g}=3$, then
$\vert N_{{\goth g}}' \vert=0$. Otherwise, $\vert N_{{\goth g}}' \vert \geq 1$.
\end{lemma}

\begin{proof} The set $\Upsilon_{{\goth g}}$ is nonempty if and only if
$\dim {\goth g} >3$. Suppose now $\dim {\goth g}> 3$. Let $\alpha^{\#}$ be the highest
root of ${\cali R}_{+}$. By Lemma \ref{li3}, ${\cali N}_{{\goth g},e}$ contains
$[e,{\goth g}^{-\alpha^{\#}}]$. So it suffices to prove that
$[e,{\goth g}^{-\alpha^{\#}}]$ is not contained in ${\goth p}$ for any ${\goth p}$ in
$\Upsilon_{{\goth g}}$. For $\alpha$ in ${\cali R}$, we fix a nonzero element
$x_{\alpha}$ in ${\goth g}^{\alpha}$. The elements $x_{\alpha}$ may be chosen so
that $e=\sum\limits_{\alpha \in \Pi} x_{\alpha}$. Suppose that there is ${\goth p}$ in
$\Upsilon_{{\goth g}}$ such that $[x_{-\alpha^{\#}},e]$ belongs to ${\goth p}$. We can
suppose that ${\goth p}$ is a maximal element of $\Upsilon_{{\goth g}}$ corresponding to
a simple root $\beta$. A short discussion shows that $\beta$ satisfies the two following
conditions:
\begin{list}{}{}
\item 1) for all $\alpha $ in $\Pi \backslash \{\beta \}$, $\alpha^{\#} - \alpha $ is
not a root,
\item 2) $\alpha^{\#} -\beta $ is a root and $n_\beta=1$, where $n_{\beta}$ is the
coordinate at $\beta $ of $\alpha^{\#}$ in the basis $\Pi$.
\end{list}

A quick look on the classification of root systems shows up that there is no simple root
satisfying both Conditions (1) and (2). As we obtain a contradiction, the result follows.
\end{proof}

\begin{proposition}\label{pp2} The number of irreducible components of
${\cali N}_{{\goth g}}$ is equal to $\vert N_{{\goth g}} \vert$ and
$\vert N_{{\goth g}} \vert$ satisfies the following recursive relation:
$$\vert N_{{\goth g}} \vert = \vert N_{{\goth g}}^{'} \vert +
\left( \sum\limits_{{\goth p} \in \Upsilon_{{\goth g}}}
\prod \limits_{{\goth l} \textrm{ simple factor}
\atop \textrm{ of } [{\goth l}_{{\goth p}},{\goth l}_{{\goth p}}]}
\vert N_{{\goth l}}^{'} \vert \right) + 1 \mbox{ . }$$
\end{proposition}

\begin{proof} By Corollary \ref{cp2}, the number of irreducible components of
${\cali N}_{{\goth g}}$ is $\vert N_{{\goth g}} \vert$. If $\dim {\goth g}=3$,
then ${\cali N}_{{\goth g},e}$ is the line generated by $e$, $\Upsilon_{{\goth g}}=\emptyset$, $|N_{{\goth g}}^{'}|=0$ and $\vert N_{{\goth g}} \vert=1$, whence the expected relation. Suppose now $\dim {\goth g}>3$. We can write
\begin{eqnarray*}
N_{{\goth g}} & = & N_{{\goth g}}^{'} \cup \left( \bigcup\limits_{{\goth p} \in
\Upsilon_{{\goth g}}} N_{{\goth g}}^{({\goth p})}\right) \cup \{\goth u\}
\mbox{ , }
\end{eqnarray*}
where, for ${\goth p}$ in $\Upsilon_{{\goth g}}$, $N_{{\goth g}}^{({\goth p})}$ is the
set of irreducible components of ${\cali N}_{{\goth g},e}$ contained in ${\goth p}$ and
not contained in any element of $\Upsilon _{{\goth g}}$ strictly
contained in ${\goth p}$. Thus the terms of the above union are pairwise disjoint and we
have
\begin{eqnarray*}
\vert N_{{\goth g}}\vert & = & \vert N_{{\goth g}}^{'} \vert  +
\sum_{{\goth p}\in \Upsilon _{{\goth g}}} \vert N_{{\goth g}}^{({\goth p})} \vert +1
\mbox{ .}
\end{eqnarray*}
It remains to prove the equality
\begin{eqnarray}\label{r2}
\vert N_{{\goth g}}^{({\goth p})} \vert \ = \
\prod \limits_{{\goth l} \textrm{ simple factor}
\atop \textrm{ of } [{\goth l}_{{\goth p}},{\goth l}_{{\goth p}}]}
\vert N_{{\goth l}}^{'} \vert \mbox{ ,}
\end{eqnarray}
for any ${\goth p}$ in $\Upsilon _{{\goth g}}$. Let ${\goth p}$ be in
$\Upsilon _{{\goth g}}$. By Lemma \ref{l3p2}, the map $Y\mapsto Y+{\goth u}_{{\goth p}}$
is a bijection from $N_{{\goth l}_{{\goth p}}}$ to the subset of irreducible components of
${\cali N}_{{\goth g},e}$ contained in ${\goth p}$. Moreover, $Y+{\goth u}_{{\goth p}}$
belongs to $N_{{\goth g}}^{({\goth p})}$ if and only if $Y$ belongs to
$N'_{{\goth l}_{{\goth p}}}$. So
$\vert N'_{{\goth l}_{{\goth p}}} \vert = \vert N_{{\goth g}}^{({\goth p})} \vert$. Let
$\poi {{\goth l}}1{,\ldots,}{m}{}{}{}$ be the simple factors of
${\goth l}_{{\goth p}}$ and let $Y$ be in $N_{{\goth l}_{{\goth p}}}$. As
remarked in Subsection \ref{notation}, $Y=Y_1 \times \cdots \times Y_m$ where $Y_{i}$ is
an irreducible component of ${\cali N}_{{\goth l}_{i},e_i}$, if $e_i$ is the component of
$\varpi_{{\goth l}_{{\goth p}}}(e)$ on ${\goth l}_i$, for $i=1,\ldots,m$. Then,
$Y \in N'_{{\goth l}_{{\goth p}}}$ if and only if, for any $i\in \{1,\ldots,m\}$, $Y_{i}$
is not contained in any element of $\Upsilon_{{\goth l}_i}$, that is $Y_i \in N'_{{\goth l}_{\goth p}}$. Hence
$$ \vert N'_{{\goth l}_{\goth p}} \vert \ = \
\prod_{i=1}^{m} \vert N'_{{\goth l}_{i}} \vert \mbox{ ,}$$
and the proposition follows, by Relation (\ref{r2}).
\end{proof}

\begin{remark} Considering the formula established in Proposition \ref{pp2}, we see that the
number $\vert N_{{\goth g}}\vert$ becomes considerably big while the dimension of
${\goth g}$ grows up. For example we have:
$$\vert N_{{\mathfrak{sl}}_2} \vert =1 \mbox{ , } \vert N_{{\mathfrak{sl}}_3}\vert  \geq 2
\mbox{ , } \vert N_{{\mathfrak{sl}}_4}\vert  \geq 4 \mbox{ , } \vert N_{{\mathfrak{sl}}_5} \vert
\geq 7 \mbox{ , } \vert N_{{\mathfrak{sl}}_6}\vert  \geq 12 \mbox{ ,} \ldots \mbox{ . }$$
\end{remark}


\end{document}